\documentclass[11pt,a4paper,twoside]{article}
\usepackage{titling}
\usepackage{titletoc}
\usepackage{url}

\usepackage[utf8]{inputenc}
\usepackage[T1]{fontenc}
\usepackage{amsmath}
\usepackage{amsfonts}
\usepackage{amssymb}
\usepackage{amsthm}

\usepackage{mathrsfs}
\usepackage{mathtools}
\usepackage{stmaryrd}
\usepackage[all]{xy}
\usepackage{makeidx}
\usepackage{natbib}
\usepackage[french,english]{babel}
\usepackage{xspace}
\usepackage{etoolbox}

\makeatletter\def\th@plain{\slshape}\makeatother
\makeatletter\patchcmd{\th@remark}{\itshape}{\slshape}{}{}\makeatother

\newcounter{bidon}
\newcommand{\rdb}{\refstepcounter{bidon}}

\newcommand\Subsubsection[1]{
\rdb\addcontentsline{toc}{subsubsection}{#1} 
\subsubsection*{#1}}

\renewcommand\paragraph[1]{
\rdb
\addcontentsline{toc}{paragraph}{#1} 
\medskip \noindent $\bullet$ \textbf{#1}}

\newcommand{\DT}{\fD_\gT}
\newcommand{\VT}{\fV_\gT}

\newcommand \He {\mathsf{He}}

\newcommand{\di}{\,{\vert}\,}

\newcommand {\junk}[1]{}

\newcommand\ndsp{\textstyle}

\newcommand \noi {\noindent}

\newcommand \sni {\smallskip\noi}

\newcommand \alb {\allowbreak}

\renewcommand \le{\leqslant}
\renewcommand \leq{\leqslant}
\renewcommand \preceq{\preccurlyeq}

\renewcommand \geq{\geqslant}

\newcommand\sta{^\star}

\newcommand \bl {^\bullet}
\newcommand{\bul}{^{\bullet}}
\newcommand{\Abul}{\gA\bul}

\newcommand \eci {^\circ}

\newcommand \etoz{$^*$}

\newcommand\equidef{\buildrel{{\rm def}}\over{\;\Longleftrightarrow\;}} 
\newcommand\eqdef{\buildrel{\rm def}\over {\;=\;}}
\newcommand\eqdefi{\buildrel{\rm def}\over {\;=\;}}

\newcommand{\pref}[1]{\textup{\hbox{\normalfont(\ref{#1})}}}

\newcommand \aqo[2] {#1\sur{\gen{#2}}\!}

\newcommand \gen[1] {\left\langle{#1}\right\rangle}

\newcommand \so[1] {\left\{{#1}\right\}} 
\newcommand \sO[1]{\big\{{#1}\big\}}
\newcommand \sotq[2]{\so{\,#1\,\vert\,#2\,}}
\newcommand \sotQ[2]{\sO{\,#1\;\big\vert\;#2\,}}
\newcommand \sur[1] {\!\left/#1\right.}
\newcommand \und[1] {\underline{#1}}

\newcommand \U {\mathrm{U}}

\newcommand \som {\sum\nolimits}
\newcommand \Ex {{\exists}}
\newcommand \Tt {{\forall}}

\newcommand\lrb[1] {\llbracket #1 \rrbracket}
\newcommand\lrbn {\lrb{1..n}}

\newcommand\lrbm {\lrb{1..m}}

\newcommand \vda {\,\vdash\,}

\newcommand \dar[1] {\MA{\downarrow \!#1}}
\newcommand \uar[1] {\MA{\uparrow \!#1}}

\newcommand\tsbf[1]{\textsf{\textbf{\textup{#1}}}}
\newcommand\lab[1]{\item[\tsbf{#1}]}
\newcommand\Lab[1]{\rdb\item[\tsbf{#1}]\label{Ax#1}}
\newcommand\Tsbf[1]{\hyperref[Ax#1]{\tsbf{#1}}}

\newcommand\sA[1]{\hbox{\small\usefont{T1}{pzc}{m}{it}#1}\,}
\newcommand\sa[1]{\hbox{\large\usefont{T1}{pzc}{m}{it}#1}\,}

\newcommand\Sa[1]{\hyperref[theorie#1]{\sa{#1}}}

\newcommand \snic[1] {\sni\centerline{$#1$}

\smallskip}

\newcommand \eoe {\hbox{}\nobreak\hfill
\vrule width .5em height .5em depth 0mm \par \smallskip}

\newcommand \ov[1] {\overline{#1}}

\newcommand \wi[1] {\widetilde{#1} }

\newcommand \UneCol[1]{%
\sni\mbox{\hspace{.02\textwidth}%
\parbox[t]{.98\textwidth}{#1}%
}}

\newcommand \DeuxCols[2]{%
\sni\mbox{\hspace{.02\textwidth}%
\parbox[t]{.475\textwidth}{#1}%
\hspace{.03\textwidth}%
\parbox[t]{.475\textwidth}{#2}}}

\newcommand \DeuxRegles[2]{%
\vspace{-1em}\DeuxCols
{\begin{enumerate}  #1
\end{enumerate}
}
{\begin{enumerate}  #2
\end{enumerate}
}
\vspace{-.3em}
}

\newcommand \Regles[1]{%
\vspace{-1em}\UneCol{
\begin{enumerate}
{#1} 
\end{enumerate}
}
\vspace{-.3em}
}

\newcommand \labu {\lab{$\bullet$}}

\makeatletter 
\def\revddots{\mathinner{\mkern1mu\raise\p@ 
\vbox{\kern7\p@\hbox{.}}\mkern2mu 
\raise4\p@\hbox{.}\mkern2mu\raise7\p@\hbox{.}\mkern1mu}} 
\makeatother

\newcommand \BB{\mathbb {B}}

\newcommand \NN{\mathbb {N}} 
\newcommand \ZZ{\mathbb {Z}}

\newcommand \QQ{\mathbb {Q}} 
\newcommand \RR{\mathbb {R}} 

\newcommand \Bo{\BB\mathrm{o}}

\newcommand \gk {\mathbf{k}}
\newcommand \gA {\mathbf{A}}
\newcommand \gB {\mathbf{B}}
\newcommand \gC {\mathbf{C}}

\newcommand \gK {\mathbf{K}}
\newcommand \gL {\mathbf{L}}

\newcommand \gR {\mathbf{R}}

\newcommand \gS {\mathbf{S}}
\newcommand \gT {\mathbf{T}}
\newcommand \gV {\mathbf{V}}

\newcommand \gZ {\mathbf{Z}}

\newdimen\xyrowsp
\newcommand{\SCO}[6]{
\xymatrix @R = \xyrowsp {
                                  &1 \ar@{-}[dl] \ar@{-}[dr] \\
#3 \ar@{-}[ddr]                   &   & #6 \ar@{-}[ddl] \\
                                  &\bullet\ar@{-}[d] \\
                                  &\bullet   \\
#2 \ar@{-}[ddr] \ar@{-}[uur]      &   & #5 \ar@{-}[ddl] \ar@{-}[uul] \\
                                  &\bullet \ar@{-}[d] \\
                                  &\bullet  \\
#1 \ar@{-}[uur]                   &   & #4 \ar@{-}[uul] \\
                                  & 0 \ar@{-}[ul] \ar@{-}[ur] \\
}
}

\renewcommand \mod {\;\mathrm{mod}\;}

\newcommand \val {\MA{\mathsf{val}}}

\newcommand\MA[1]{\mathop{#1}\nolimits}

\newcommand \Hdim {\MA{\mathsf{Hdim}}}
\newcommand \HeA {{\Heit\gA}}
\newcommand \Heit {\MA{\mathsf{Heit}}}

\newcommand \Jdim {\MA{\mathsf{Jdim}}}

\newcommand \Jspec {\MA{\mathsf{Jspec}}}
\newcommand \jspec {\MA{\mathsf{jspec}}}

\newcommand \Kdim {\MA{\mathsf{Kdim}}}

\newcommand \Max {\MA{\mathsf{Max}}}
\newcommand \Min {\MA{\mathsf{Min}}}
\newcommand \OQC {\MA{\mathsf{Oqc}}}

\newcommand \Spec {\MA{\mathsf{Spec}}}
\newcommand \Sper {\MA{\mathsf{Sper}}}
\newcommand \Spev {\MA{\mathsf{Spev}}}
\newcommand \Reel {\MA{\mathsf{Reel}}}
\newcommand \SpecA {\Spec\gA}
\newcommand \SperA {\Sper\gA}

\newcommand \SpecT {\Spec\gT}
\newcommand \Val {\MA{\mathsf{Val}}}
\newcommand \Zar {\MA{\mathsf{Zar}}}

\newcommand \ZarA {{\Zar\gA}}

\newcommand \Sac {\MA{\mathsf{Fsac}}}

\newcommand \Sace {\MA{\mathsf{Fsace}}}

\newcommand \cA {\mathcal{A}}

\newcommand \cI {\mathcal{I}}

\newcommand \cL {\mathcal{L}}

\newcommand \cP {\mathcal{P}}

\newcommand \rD {\mathrm{D}}

\newcommand \rJ {\mathrm{J}}

\newcommand \rR {\mathrm{R}}

\newcommand \rU {\mathrm{U}}

\newcommand \DA {\rD_\gA}
\newcommand \JA {\rJ_\gA}
\newcommand \JT {\rJ_\gT}

\newcommand\fa{\mathfrak{a}}
\newcommand\fb{\mathfrak{b}}
\newcommand\fc{\mathfrak{c}}

\newcommand\fA{\mathfrak{A}}
\newcommand\fB{\mathfrak{B}}
\newcommand\fC{\mathfrak{C}}
\newcommand\fD{\mathfrak{D}}

\newcommand\fII{\mathfrak{I}}
\newcommand\fj{\mathfrak{j}}
\newcommand\fJ{\mathfrak{J}}
\newcommand\fF{\mathfrak{F}}
\newcommand\ff{\mathfrak{f}}

\newcommand\fm{\mathfrak{m}}

\newcommand\fp{\mathfrak{p}}
\newcommand\fP{\mathfrak{P}}

\newcommand\fq{\mathfrak{q}}
\newcommand\fR{\mathfrak{R}}

\newcommand\fV{\mathfrak{V}}

\newcommand \vu {\vee} 
\newcommand \vi {\wedge} 
\newcommand \Vu {\bigvee}
\newcommand \Vi {\bigwedge}

\newcommand \Un {\mathbf{1}}
\newcommand \Deux {\mathbf{2}}

\newcommand \ux {{\underline{x}}}

\newcommand \uy{{\underline{y}}}

\newcommand \an {a_1,\ldots,a_n}

\newcommand \xn {x_1,\ldots,x_n}

\newcommand \Xn {X_1,\ldots,X_n}

\newcommand \ym {y_1,\ldots,y_m}

\newcommand \AX {\gA[X]}

\newcommand \AXn {\gA[\Xn]}

\newcommand \lfs {f_1,\ldots,f_s}

\newcommand \Pfe {{\rm P}_{{\rm fe}}}

\renewcommand{\SCO}[6]{
\xymatrix @R = .2em @C =4em{
                                  &1 \ar@{-}[dl] \ar@{-}[dr] \\
#3 \ar@{-}[ddr]                   &   & #6 \ar@{-}[ddl] \\
                                  &\bullet\ar@{-}[d] \\
                                  &\bullet   \\
#2 \ar@{-}[ddr] \ar@{-}[uur]      &   & #5 \ar@{-}[ddl] \ar@{-}[uul] \\
                                  &\bullet \ar@{-}[d] \\
                                  &\bullet  \\
#1 \ar@{-}[uur]                   &   & #4 \ar@{-}[uul] \\
                                  & 0 \ar@{-}[ul] \ar@{-}[ur] \\
}
}

\newcommand{\SCOR}[8]{
\xymatrix @R =.2em @C =4em {
                                  &\bullet \ar@{-}[ddl] \ar@{-}[ddr] \\
                                  &#7\ar@{-}[u] \\
#3 \ar@{-}[ddr]                   &   & #6 \ar@{-}[ddl] \\
                                  &\bullet\ar@{-}[d] \\
                                  &\bullet   \\
#2 \ar@{-}[ddr] \ar@{-}[uur]      &   & #5 \ar@{-}[ddl] \ar@{-}[uul] \\
                                  &\bullet \ar@{-}[d] \\
                                  &\bullet  \\
#1 \ar@{-}[uur]                   &   & #4 \ar@{-}[uul] \\
                                  &#8\ar@{-}[d] \\
                                  &\bullet \ar@{-}[uul] \ar@{-}[uur] \\
}
}

\renewcommand \DeuxRegles[2]{%
\vspace{-1.5em}\DeuxCols
{\begin{enumerate}  #1
\end{enumerate}
}
{\begin{enumerate}  #2
\end{enumerate}
}
\vspace{-.5em}
}

\renewcommand \Regles[1]{%
\vspace{-1.5em}\UneCol{
\begin{enumerate}
{#1} 
\end{enumerate}
}
\vspace{-.5em}
}

\DeclareMathAlphabet{\mathpzc}{OT1}{pzc}{m}{it}

\usepackage{stixgras}
\RequirePackage{amsbsy}

\newcommand\Exists{\boldsymbol{\stixexists}}

\newcommand\VDash{\boldsymbol{\stixvdash}}
\newcommand\Land{\boldsymbol{\stixwedge}}

\newcommand\Top{\boldsymbol{\stixtop}}
\newcommand\Bot{\boldsymbol{\stixbot}}

\newcommand \vdg{\VDash}
\newcommand \Vd {\,\vdg}
\newcommand \vd {\,\,\vdg}
\newcommand \vii{\Land}

\newcommand \vdi[1] {\mathrel{\,\,\vdg_{#1}}}
\newcommand \Vdi[1] {\mathrel{\Vd_{#1}}}

\newcommand \sibrouillon[1]{}

\usepackage[bookmarksopen=false,breaklinks=true,%
      backref=page,pagebackref=true,plainpages=false,%
      hyperindex=true,pdfstartview=FitH,%
      pdfpagelabels=true,
      colorlinks=true,linkcolor=blue,%
      citecolor=red,urlcolor=red,
      colorlinks=true%
      ]%
   {hyperref}

\begin{document}

\thispagestyle{empty}
~ 
\vspace{3cm}

\noindent In this file you find the English version starting on the page  numbered \pageref{beginenglish}.

\medskip \noindent  {\Large \bf Spectral spaces versus distributive lattices: a dictionary}

\bigskip \noindent  
Then the French version begins on the page numbered \pageref{beginfrench}.

\medskip\noindent   {\Large \bf Treillis distributifs et espaces spectraux, un petit dictionnaire} 

\bigskip  This article is a slightly expanded version of the chapter
``Spectral spaces versus distributive lattices: a dictionary",
 in the book \textsl{Advances in rings, modules and factorizations. Selected papers based on the presentations at the international conference on rings and factorizations, Graz, Austria, February 19--23, 2018}, published by Springer in 2020, ISBN 978-3-030-43415-1. We have also fixed some typographical bugs and improved the bibliography.

\newpage
\thispagestyle{empty}

~

\pagestyle{headings}
\patchcmd{\sectionmark}{\MakeUppercase}{}{}{}
\setcounter{page}{0}\renewcommand\thepage{E\arabic{page}}

\pagestyle{headings}
\patchcmd{\sectionmark}{\MakeUppercase}{}{}{}

\def\proofname{\textsl{Proof}}

\begingroup

\title{Spectral spaces versus distributive lattices: a dictionary}
\author{Henri Lombardi 
(\thanks{Laboratoire de Mathématiques de Besançon, Université Marie et Louis Pasteur. \url{http://hlombardi.free.fr/}
email: \texttt{henri.lombardi@univ-fcomte.fr}. 
}~) 
}
\maketitle
\startcontents[english]
\setcounter{tocdepth}{4}
\newcommand \Glio {\MA{\mathsf{Liog}}}

\newcommand{\vou}{\MA{\tsbf{ or }}}
\newcommand{\Vou}{\MA{\tsbf{OR}}}
\newcommand \EXists[1] {\tsbf{Introduce }{#1}\tsbf{ such that }\,}
\newcommand \vet {{\tsbf{,}}\,}
\newcommand \Atcl {\mathrm{Clat}}
\newcommand \tcl {\mathit{Clt}}

\newcommand\comm{\rdb
\noi{\it Comment. }}

\newcommand\COM[1]{\rdb
\noi{\it Comment #1. }}

\newcommand\comms{\rdb
\noi{\it Comments. }}

\newcommand\Pb{\rdb
\noi{\bf Problem. }}

\newcommand \rem{\rdb
\noi{\sl Remark. }}

\newcommand \REM[1]{\rdb
\noi{\sl Remark#1. }}

\newcommand \rems{\rdb
\noi{\sl Remarks. }}

\newcommand \exl{\rdb
\noi{\bf Example. }}

\newcommand \EXL[1]{\rdb
\noi{\bf Example: #1. }}

\newcommand \exls{\rdb
\noi{\bf Examples. }}

\newcommand\gui[1]{``{#1}''}

\newcommand \thref[1] {Theorem~\ref{#1}}
\newcommand \paref[1] {page~\pageref{#1}}
\newcommand \pstfref[1] {Positivstellensatz formel~\ref{#1}}
\newcommand \pstref[1] {Positivstellensatz~\ref{#1}}

\newcommand \num {{n$^{\mathrm{ o}}$}}

\newcommand\subsubsec[1] {\subsubsection*{#1}}

\newcommand \hdr {induction hypothesis\xspace}
\newcommand \ssi {if and only if\xspace}
\newcommand \cnes {necessary and sufficient condition\xspace}
\newcommand \spdg {without loss of generality\xspace}
\newcommand \Propeq {T.F.A.E.\xspace}
\newcommand \propeq {t.f.a.e.\xspace}
\newcommand \disept {17$^{th}$ Hilbert's problem\xspace}

\def \cad {\textit{i.e.}\ }


\newcommand \Amo {$\gA$-module\xspace}
\newcommand \Amos {$\gA$-modules\xspace}

\newcommand \Bmo {$\gB$-module\xspace}
\newcommand \Bmos {$\gB$-modules\xspace}

\newcommand \Zmo {$\gZ$-module\xspace}
\newcommand \Zmos {$\gZ$-modules\xspace}

\newcommand \ZZmo {$\ZZ$-module\xspace}
\newcommand \ZZmos {$\ZZ$-modules\xspace}

\newcommand \Ali {$\gA$-\ali}
\newcommand \Alis {$\gA$-\alis}

\newcommand \Alg {$\gA$-\alg}
\newcommand \Algs {$\gA$-\algs}

\newcommand \kev {$\gk$-vector space\xspace}
\newcommand \kevs {$\gk$-vector spaces\xspace}

\newcommand \Kev {$\gK$-vector space\xspace}
\newcommand \Kevs {$\gK$-vector spaces\xspace}

\newcommand \klg {$\gk$-\alg}
\newcommand \klgs {$\gk$-\algs}

\newcommand \Klg {$\gK$-\alg}
\newcommand \Klgs {$\gK$-\algs}


\newcommand \agq {algebraic\xspace}

\newcommand \alg {algebra\xspace}
\newcommand \algs {algebras\xspace}

\newcommand \agB {Boolean \alg}

\newcommand \algo{algorithm\xspace}
\newcommand \algos{algorithms\xspace}

\newcommand \algq{algorithmic\xspace}

\newcommand \ali {\lin map\xspace}
\newcommand \alis {\lin maps\xspace}

\newcommand \anar {\ari \ri}
\newcommand \anars {\ari \ris}
\newcommand \Anars {\Ari \ris}

\newcommand \ari{arith\-metic\xspace}

\newcommand \auto {automorphism\xspace}
\newcommand \autos {automorphisms\xspace}


\newcommand \cac {algebraically closed field\xspace}
\newcommand \cacs {algebraically closed fields\xspace}

\newcommand \carn{characterization\xspace}  
\newcommand \carns{characterizations\xspace}  

\newcommand \coe {coefficient\xspace}
\newcommand \coes {coefficients\xspace}

\newcommand \coh {coherent\xspace}
\newcommand \cohz {coherent\xspace}

\newcommand \coli {linear combination\xspace}
\newcommand \colis {linear combinations\xspace}

\newcommand \com {comaximal\xspace}

\newcommand \coo {coordinate\xspace}
\newcommand \coos {coordinates\xspace}


\newcommand \ddp {Pr\"ufer domain\xspace}
\newcommand \ddps {Pr\"ufer domains\xspace}

\newcommand \ddk {Krull dimension\xspace}

\newcommand \dfn{definition\xspace}  
\newcommand \dfns{definitions\xspace}  

\newcommand \discri{discriminant\xspace}
\newcommand \discris{discriminants\xspace}

\newcommand \dok {Dedekind domain\xspace}
\newcommand \doks {Dedekind domains\xspace}

\newcommand \dve {divisibility\xspace}

\newcommand \dvz {zero divisor\xspace}
\newcommand \dvzs {zero divisors\xspace}

\newcommand \eco{\com \elts}  

\newcommand \egt{equality\xspace} 
\newcommand \egts{equalities\xspace} 

\newcommand \elr{elementary\xspace}  

\newcommand \elt{element\xspace}  
\newcommand \elts{elements\xspace}  

\def \endo {endomorphism\xspace}
\def \endos {endomorphisms\xspace}

\newcommand \entrel {entailment relation\xspace}
\newcommand \entrels {entailment relations\xspace}

\newcommand \eqv  {equivalent\xspace}

\newcommand \evc{vector space\xspace} 
\newcommand \evcs{vector spaces\xspace} 


\newcommand \fab {bounded \fcn}
\newcommand \fabs {bounded \fcns}

\newcommand \fac {total \fcn}
\newcommand \facz {total \fcnz}

\newcommand \fap {partial \fcn}
\newcommand \faps {partial \fcns}

\newcommand \fcn {factorisation\xspace}
\newcommand \fcns {factorisations\xspace}

\newcommand \fdi{strongly discrete\xspace} 


\newcommand\gmq{geometric\xspace}

\newcommand\gne{generalized\xspace}

\newcommand\gnl{general\xspace}

\newcommand\gnlt{generally\xspace}

\newcommand\gnn{generalisation\xspace}
\newcommand\gnns{generalisations\xspace}

\newcommand\gnq{generic\xspace}

\newcommand\grl{$\ell$-group\xspace}
\newcommand\grls{$\ell$-groups\xspace}

\newcommand \gtr{generator\xspace}  
\newcommand \gtrs{generators\xspace}  


\newcommand \homo {homomorphism\xspace}
\newcommand \homos {homomorphisms\xspace}

\newcommand \id {ideal\xspace}
\newcommand \ids {ideals\xspace}

\newcommand \idd {de\-ter\-mi\-nantal \id}
\newcommand \idds {de\-ter\-mi\-nantal \ids}
\newcommand \iddz {de\-ter\-mi\-nantal \idz}
\newcommand \iddsz {de\-ter\-mi\-nantal \idsz}

\newcommand \idema {maximal \id}
\newcommand \idemas {maximal \ids}

\newcommand \idep {prime \id}
\newcommand \ideps {prime \ids}

\newcommand \idemi {minimal prime\xspace}
\newcommand \idemis {minimal primes\xspace}

\newcommand \idf {Fitting \id}
\newcommand \idfs {Fitting \ids}

\newcommand \idm {idempotent\xspace}
\newcommand \idms {idempotents\xspace}

\newcommand \idtr {indeterminate\xspace}
\newcommand \idtrs {indeterminates\xspace}

\newcommand \ifr {fractional \id}
\newcommand \ifrs {fractional \ids}

\newcommand \itf {\tf \id}
\newcommand \itfs {\tf \ids}

\newcommand \iso {isomorphism\xspace}
\newcommand \isos {isomorphisms\xspace}

\newcommand \iv {invertible\xspace}

\newcommand \lec {reader\xspace}

\newcommand \lgb {local global\xspace}

\newcommand \lin {linear\xspace}

\newcommand \lon {localisation\xspace}
\newcommand \lons {localisations\xspace}

\newcommand \lop {\lot principal\xspace}

\newcommand \losd {\lot \sdz\xspace}

\def \lot {locally\xspace}

\newcommand \mlp {principal \lon matrix\xspace}
\newcommand \mlps {principal \lon matrices\xspace}

\newcommand \mnp {manipulation\xspace}
\newcommand \mnps {manipulations\xspace}
\newcommand \mnr {\elr \mnp}
\newcommand \mnrs {\elr \mnps}

\newcommand \mo {monoid\xspace}
\newcommand \mos {monoids\xspace}
\newcommand \moco {\com \mos}

\newcommand \mpf {\pf module\xspace}
\newcommand \mpfs {\pf modules\xspace}

\newcommand \mpn {\pn matrix\xspace}
\newcommand \mpns {\pn matrices\xspace}

\newcommand \mpr {\pro module\xspace}
\newcommand \mprs {\pro modules\xspace}

\newcommand \mprn {\prn matrix\xspace}
\newcommand \mprns {\prn matrices\xspace}

\newcommand \mptf {\ptf module\xspace}
\newcommand \mptfs {\ptf modules\xspace}

\newcommand \mrc {projective module of constant rank\xspace}
\newcommand \mrcs {projective modules of constant rank\xspace}


\newcommand \ncr{necessary\xspace}

\newcommand \ncrt{necessarily\xspace}

\newcommand \ndz {regular\xspace}

\newcommand \noe {Noetherian\xspace}
\newcommand \noco {\noe\coh}

\newcommand \nst {Nullstellensatz\xspace}
\newcommand \nsts {Nullstellens\"atze\xspace}

\newcommand \odz {Zariski open set\xspace}

\newcommand \oqc {\qc open set\xspace}
\newcommand \oqcs {\qc open sets\xspace}


\newcommand \pa {saturated pair\xspace}
\newcommand \pas {saturated pairs\xspace}

\newcommand \pb{problem\xspace}  
\newcommand \pbs{problems\xspace}

\newcommand \peq {purely equational\xspace}

\newcommand \pf {finitely presented\xspace}

\newcommand \plg {\lgb principle\xspace}
\newcommand \plgs {\lgb principles\xspace}

\newcommand \pn {presentation\xspace}
\newcommand \pns {presentations\xspace}

\newcommand \pol {polynomial\xspace}
\newcommand \pols {polynomials\xspace}

\newcommand \polcar {characteristic \pol}

\newcommand \prc {rank constant \pro}

\newcommand \prmt {précisely\xspace}

\newcommand \prn {projection\xspace}
\newcommand \prns {projections\xspace}

\newcommand \pro {projective\xspace}

\newcommand \proi {potential prime\xspace}
\newcommand \prois {potential primes\xspace}

\newcommand \proc {potential chain\xspace}
\newcommand \procs {potential chains\xspace}

\newcommand \proel {elementary \proc}
\newcommand \proels {elementary \procs}
\newcommand \proelo {\proel of length }
\newcommand \proelos {\proels of length }

\newcommand \prolo {\proc of length }
\newcommand \prolos {\procs of length }

\newcommand \prt {property\xspace}
\newcommand \prts {properties\xspace}

\newcommand \pst {Positivstellensatz\xspace}
\newcommand \psts {Positivstellens\"atze\xspace}

\newcommand \ptf {\tf \pro}


\newcommand \qc {quasi-compact\xspace}

\newcommand \qi {quasi integral\xspace}

\newcommand \rcf {real closed field\xspace}
\newcommand \rcfs {real closed fields\xspace}

\newcommand \rdl {linear dependance relation\xspace}
\newcommand \rdls {linear dependance relations\xspace}

\newcommand \rdi {integral dependance relation\xspace}
\newcommand \rdis {integral dependance relations\xspace}

\newcommand \ri {ring\xspace}
\newcommand \ris {rings\xspace}


\newcommand \sad {dynamical algebraic structure\xspace}
\newcommand \sads {dynamical algebraic structures\xspace}

\newcommand \sdz {without \dvz}
\newcommand \sdzz {without \dvzz}

\newcommand \sgr {\gtr set\xspace}
\newcommand \sgrs {\gtr sets\xspace}

\newcommand \sli {\lin \sys}
\newcommand \slis {\lin \syss}

\newcommand \sys {system\xspace}
\newcommand \syss {systems\xspace}

\newcommand \talg {Horn theory\xspace}
\newcommand \talgs {Horn theories\xspace}

\newcommand \tco {coherent theory\xspace}
\newcommand \tcos {coherent theories\xspace}

\newcommand \tdy {dynamical theory\xspace}
\newcommand \tdys {dynamical theories\xspace}

\newcommand \tel {regular theory\xspace}
\newcommand \tels {regular theories\xspace}

\newcommand \telri {cartesian theory\xspace}
\newcommand \telris {cartesian theories\xspace}

\newcommand \tf {finitely generated\xspace}

\newcommand \tfo {formal theory\xspace}
\newcommand \tfos {theory formelles\xspace}

\newcommand \tgm {\gmq theory\xspace}
\newcommand \tgms {\gmq theories\xspace}

\newcommand \Tho {Theorem\xspace}
\newcommand \tho {theorem\xspace}
\newcommand \thos {theorems\xspace}

\newcommand \tpe {purely equational theory\xspace}

\newcommand \trdi {distributive lattice\xspace}
\newcommand \trdis {distributive lattices\xspace}

\newcommand \vfn {verification\xspace}
\newcommand \vfns {verifications\xspace}

\newcommand \zed {zero-dimensional\xspace}


\newcommand \cov {constructive\xspace}

\newcommand \coma {\cov \maths}
\newcommand \clama {classical \maths}

\renewcommand \cot {constructively\xspace}

\newcommand \mathe {mathematical\xspace}
\newcommand \maths {mathematics\xspace}

\newcommand \matn {mathematician\xspace}

\newcommand \pte {excluded middle principle\xspace}

\newcommand \prco {\cov proof\xspace}
\newcommand \prcos {constructive proofs\xspace}

\newcommand \tcg {compactness theorem\xspace}
\newcommand \Tcgi {The \tcg implies the following result. }
%



\theoremstyle{plain}
\newtheorem{theorem}{Theorem}[section]
\newtheorem{thdef}[theorem]{Theorem and definition}
\newtheorem{lemma}[theorem]{Lemma}
\newtheorem{corollary}[theorem]{Corollary}
\newtheorem{proposition}[theorem]{Proposition}
\newtheorem{propdef}[theorem]{Proposition and definition}
\newtheorem{plcc}[theorem]{Concrete local-global principle}
\newtheorem{fact}[theorem]{Fact}
\newtheorem*{Principleofcoveringbyquotients}{Principle of covering by quotients}
\theoremstyle{definition}
\newtheorem{conjecture}[theorem]{Conjecture}
\newtheorem{definition}[theorem]{Definition}
\newtheorem{definitions}[theorem]{Definitions}
\newtheorem{notation}[theorem]{Notation}
\newtheorem{definota}[theorem]{Definition and notation} 
\newtheorem{convention}[theorem]{Convention}
\newtheorem{problem}[theorem]{Problem}
\newtheorem{question}[theorem]{Question}

\theoremstyle{remark}
\newtheorem{remark}[theorem]{Remark}
\newtheorem{remarks}[theorem]{Remarks}
\newtheorem{comment}[theorem]{Comment}
\newtheorem{comments}[theorem]{Comments}
\newtheorem{example}[theorem]{Example}
\newtheorem{examples}[theorem]{Examples}

\rdb
\label{beginenglish}

\begin{abstract}
The category of \trdis is, in \clama, anti\-equi\-valent to the category of spectral spaces. We give here some examples and a short dictionary for this
antiequivalence. We propose a translation of several abstract theorems (in \clama) into constructive ones, even in the case where points of a spectral space have no clear constructive content. 
\end{abstract}

\noindent Note: This article is a slightly expanded version of the chapter
\gui{Spectral spaces versus distributive lattices: a dictionary},
 in the book \textsl{Advances in rings, modules and factorizations. Selected papers based on the presentations at the international conference on rings and factorizations, Graz, Austria, February 19--23, 2018}, published by Springer in 2020, ISBN 978-3-030-43415-1. We have also fixed some typographical bugs and improved the bibliography.

\newpage

\printcontents[english]{}{1}{}

\section*{Introduction} 
\addcontentsline{toc}{section}{Introduction}

This paper is written in Bishop's style of \coma   (\cite{Bi67,BB85,BR1987,CACM,MRR,Yen2015}).
We give a short dictionary between classical and \coma w.r.t.\ \prts of spectral spaces and of the associated dual \trdis.
We give several examples of how this works.

\section{Distributive lattices and spectral spaces: some general facts}\label{secdival0}
References: 
\cite{CC00,CL05,DST2019,Lom06,Lom-tgac,Sto37},  \cite[Chapter 4]{BW74} and \cite[Chapters XI and XIII]{CACM}.

\subsection{The seminal paper by Stone}

In \clama, a \textsl{prime ideal} $\fp$ of a \trdi $\gT\neq \Un$ is an ideal whose complement $\ff$ is a filter ({\sl
a prime filter}).  The quotient lattice $\gT/(\fp=0,\allowbreak\ff=1)$ is isomorphic to~$\Deux$.  Giving a \idep of~$\gT$ is the same thing as giving a lattice morphism $\gT\rightarrow \Deux$.
We will write $\theta_\fp:\gT\to\Deux$ the morphism
corresponding to $\fp$.

If $S$ is a system of generators for a 
\trdi $\gT$, a \idep~$\fp$ of $\gT$ is characterised by its trace $\fp\cap S$ (cf.\  \cite{CC00}).

\smallskip The (Zariski) \textsl{spectrum of the \trdi $\gT$} is the set $\Spec
\gT$ whose elements are \ideps of $\gT$, with the following topology: an open basis is provided by the subsets $\DT(a)\eqdefi\sotq{\fp\in\Spec
\gT}{a\notin\fp}=\sotq{\fp}{\theta_\fp(a)=1} .
$
One has
\begin{equation} \label{eqDa}
\left.\begin{array}{rclcrcl}
  \DT(a\vi b)   & =  & \DT(a)\cap \DT(b) ,&\quad & \DT(0)  & =  & 
\emptyset  ,\\
  \DT(a\vu b)   & =  & \DT(a)\cup \DT(b) ,&&  \DT(1) & =  &  
\Spec\gT.
  \end{array}
\right\}
\end{equation}

The complement of $\DT(a)$ is a \textsl{basic closed set} denoted by $\VT(a)$.
This notation is extended to $I\subseteq\gT$: we let
$\VT(I)\eqdefi\bigcap_{x\in I}\VT(x)$.  {If $\fII$} is the ideal generated by $I$, one
has $\VT(I)=\VT(\fII)$.  The closed set $\VT(I)$ is also called  \textsl{the subvariety of $\SpecT$ defined by $I$}.

The closure of a point $\fp\in\SpecT$ is provided by all $\fq\supseteq \fp$. Maximal ideals are the closed points of $\SpecT$. The spectrum $\SpecT$ is empty iff $0=_\gT1$.

The  spectrum of a \trdi is the paradigmatic example of a \textsl{spectral space}.
Spectral spaces can be characterised as the topological spaces satisfying the following \prts:
\begin{itemize}
\item the space is \qc,\footnote{The nowadays standard terminology is \qc, as in Bourbaki and Stacks, rather than compact.}
\item every open set is a union of \oqcs,
\item the intersection of two \oqcs is a \oqc,
\item for two distinct points, there is an open set containing one of them but not the other,
\item every irreducible closed set is the closure of a point.
\end{itemize}
The \oqcs then form a \trdi, the supremum and the infimum being the union and the intersection, respectively. A continuous map between spectral spaces is said to be \textsl{spectral} if the inverse image of every \oqc is a \oqc.

Stone's fundamental result  \cite{Sto37}  can be stated as follows.
{\sl The category of \trdi is, in \clama, antiequivalent to the category of spectral spaces.}

Here is how this works.
{\sl \begin{enumerate}
\item The \oqcs of $\SpecT$ are exactly the $\fD_\gT(u)$'s.
\item The map $u\mapsto \fD_\gT(u)$ is well-defined and it is an \iso of \trdis.
\end{enumerate}
}\label{NOTASpecT}
In the other direction, if $X$ is a spectral space we let $\OQC(X)$ be the \trdi formed by its \oqcs.
If $\xi:X\to Y$ is a spectral map, the map 
$${\OQC(\xi):\OQC(Y)\to\OQC(X),\quad U\mapsto\xi^{-1}(U)}
$$
is a morphism of \trdis.
This defines $\OQC$ as a contravariant functor.

\medskip Johnstone calls  \textsl{coherent spaces} the spectral spaces (\cite{Joh1986}). \cite{BW74} give them the name \textsl{Stone space}. The name \textsl{spectral space} is given by Hochster in a famous paper \cite{Hoc1969} where he proves that all spectral spaces can be obtained as Zariski spectra of commutative rings.

\smallskip In \coma, spectral spaces may have no points.
So it is necessary to translate the classical stuff about spectral spaces into a constructive rewriting about \trdis.
It is remarkable that all useful spectral spaces in the literature correspond to simple \trdis. 

\smallskip  Two other natural spectral topologies can be defined on $\SpecT$ by changing the definition of basic open sets.
When one chooses the $\fV(a)$'s as basic open sets, one gets the spectral space corresponding to $\gT\eci$ (obtained by reversing the order).
When one chooses boolean combinations of the $\fD(a)$'s as basic open sets
one gets the constructible topology (also called the patch topology). 
This spectral space can be defined as the spectrum of
 $\Bo(\gT)$ (the \agB generated by $\gT$). 

\Subsubsection{Spectral subspaces versus quotient lattices} 

The following theorem explains that the notion of \textsl{spectral subspace} is translated by the notion of \textsl{quotient \trdi}. Some détails are added. See also the \thref{th-dico-trdi-spec-mor1}.

\begin{theorem}[Subspectral spaces]
\label{propSESP}  Let $\gT'$ be a quotient lattice of $\gT$ and $\pi:\gT\to\gT'$ the quotient morphism.  Let us write $X'=\Spec\gT'$, $X=\Spec\gT$ and
$\pi^\star:X'\to X$ the dual map of $\pi$.  
\begin{enumerate}
\item \sloppy  $\pi^\star$ identifies  $X'$ with a \emph{topological subspace} of $X$. 
Moreover $\OQC(X')=\allowbreak{\sotq{U\cap X'}{U\in\OQC(X)}}$. 
We say that \emph{$X'$ is a subspectral space of~$X$.}
\item A subset $X'$ of  $X$ is a subspectral space of~$X$ \ssi \\
-- the induced topology by $X$ on $X'$ is spectral, and\\
--  $\OQC(X')=\sotq{U\cap X'}{U\in\OQC(X)}$.
\item A subset $X'$ of  $X$ is a subspectral space \ssi
it is closed for the patch topology.
\item If $Z$ is an arbitrary subset of 
$X=\Spec\gT$, its closure for the patch topology
is given by 
$X'=\Spec\gT'$, where
$\gT'$ is the quotient lattice of~$\gT$ defined by the following preorder
$\preceq$:
\begin{equation} \label{eqSSES}
a\preceq b\quad \Longleftrightarrow\quad (\DT(a)\cap Z)\subseteq 
(\DT(b)\cap Z)
\end{equation}
\end{enumerate}
\end{theorem}

\begin{proposition}[Basic open and closed subsets]
\label{propositionOFBSES}  Let $\gT$ be a \trdi and  $X=\Spec\,\gT$.
\begin{enumerate}
\item $\DT(a)$ is a spectral subspace of $X$ canonically homeomorphic to
 $
\Spec(\gT/(a=1))$.
\item  $\VT(b)$ is a spectral subspace of $X$ canonically homeomorphic to
 $
\Spec(\gT/(b=0))$.
\end{enumerate}
\end{proposition}

\begin{proposition}[Closed subsets of $\Spec\,\gT$]
\label{propositionFSES} ~
\begin{enumerate}
\item Any closed subset of $\Spec\,\gT$ hs the form
$\VT(\fJ):=\bigcap_{x\in \fJ}\VT(x)$ where $\fJ$ is an \id of $\gT$.
It is a spectral subspace and corresponds to the quotient lattice  
$\gT/(\fJ=0)$.
\item
The intersection of a family of closed subsets corresponds to the sup of the family of corresponding ideals. The union of two closed subsets  corresponds to the intersection of the corresponding ideals.
\item The closure of $\DT(x)$ corresponds to the quotient 
$\gT/((0:x)=0)$.
\item\label{enumtra}
The quotient lattice $\gT/((a:b)=0)$ corresponds to 
$\ov{\VT(a)\cap
\DT(b)}$.
\end{enumerate}
\end{proposition}

\Subsubsection{Gluing \trdis and spectral subspaces} 

\smallskip  \noindent Let $(\xn)$ be a system of \eco in a commutative ring $\gA$.
Then the canonical morphism $\gA\to\prod_{i\in\lrbn}\gA[1/x_i]$ identifies $\gA$ with a finite subproduct of localisations of $\gA$. 

\smallskip Similarly a \trdi 
can be recovered from a finite number of good quotient lattices.

\begin{definition}
\label{defRecolTD}
Let $\gT$ be a \trdi and $(\fa_i)_{i\in\lrbn}$ (resp. 
$(\ff_i)_{i\in\lrbn}$)
a  finite family of \ids (resp. of filters)  of $\gT$.  We say that the \ids
$\fa_i$ \textsl{cover $\gT$} if~~$\bigcap_i\fa_i=\so{0}$. Similarly we say that
the filters $\ff_i$ \textsl{cover $\gT$} if~~$\bigcap_i\ff_i=\so{1}$.
\end{definition}

Let $\fb$ be an ideal of $\gT$; we write $x\equiv y\mod\fb$ as meaning  $x\equiv y\mod (\fb=0)$. Let us recall that for $s\in\gT$ the quotient $\gT/(s=0)$ is isomorphic to the principal filter $\uar s$ (one sees this filter as a \trdi with $s$ as $0$ element).

\begin{fact}
\label{factRecolTD}
Let $\gT$ be a \trdi, $(\fa_i)_{i\in\lrbn}$ a finite family 
of principal ideals ($\fa_i=\dar s_i$) and $\fa=\bigcap_i\fa_i$.
\begin{enumerate}
\item If $(x_i)$ is a family in $\gT$ s.t.\ for each $i,j$ one has $x_i\equiv x_j\,\mod\,\fa_i\vu\fa_j$, then there exists a unique $x$ 
modulo
$\fa$ satisfying:  $x\equiv x_i\,\mod\,\fa_i\;(i\in\lrbn)$.
\item Let us write $\gT_i=\gT/(\fa_i=0)$, 
$\gT_{ij}=\gT_{ji}=\gT/(\fa_i\vu\fa_j=0)$,
$\pi_i:\gT\to\gT_i$ and $\pi_{ij}:\gT_i\to\gT_{ij}$ the canonical maps.
If the ideals $\fa_i$ cover $\gT$, the system $(\gT,(\pi_i)_{i\in\lrbn})$  is the inverse limit of the diagram $$((\gT_i)_{1\leq i\leq n},(\gT_{ij})_{1\leq 
i<j\leq
n};(\pi_{ij})_{1\leq i\neq j\leq n}).$$
\item The analogous result works with quotients by principal filters.

 {\small
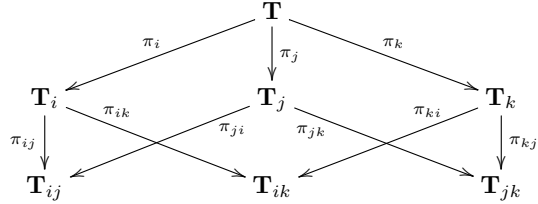
\hspace*{10em}{
\xymatrix @R=2em @C=7em{
          &  \gT \ar[rd]^{\pi _{k}}\ar[d]^{\pi _{j}}\ar[ld]_{\pi _{i}}\\
 \gT _i\ar[d]_{\pi _{ij}}\ar@/-0.75cm/[dr]^{\pi _{ik}} &
     \gT _j\ar@/-.8cm/[dl]_{\pi _{ji}}\ar@/-.8cm/[dr]^{\pi _{jk}} &
        \gT _k\ar@/-0.75cm/[dl]_{\pi _{ki}}\ar[d]^{\pi _{kj}} &
\\
 \gT _{ij}  & 
    \gT _{ik}   & 
      \gT _{jk}   
}
}}

\end{enumerate}
\end{fact}

We have also a gluing procedure described in the following proposition.\footnote{In commutative algebra, a similar procedure works for \Amos (\cite[XV-4.4]{CACM}) But in order to glue commutative rings, it is necessary to pass to the category of Grothendieck schemes.}

\begin{proposition}[Gluing \trdis]
\label{propRecolTD} 
Let~$I$ be a finite set and a diagram of \trdis  
$${\big((\gT_i)_{i\in I},(\gT_{ij})_{i<j\in I},(\gT_{ijk})_{i<j<k\in I};
(\pi_{ij})_{i\neq j},(\pi_{ijk})_{i< j, j\neq k\neq i}\big)}
$$
and a family of \elts 
${(s_{ij})_{i\neq j\in I}\in \prod\nolimits_{i\neq j\in I}\gT_{i}}$
 satisfying the following \prts
\begin{itemize}
\item the diagram is commutative, 
\item if $i\neq j$, $\pi_{ij}$ is a quotient morphism w.r.t. the ideal  $\dar s_{ij}$,
\item if $i$, $j$, $k$ are distinct, $\pi_{ij}(s_{ik})=\pi_{ji}(s_{jk})$ and $\pi_{ijk}$ is a quotient morphism w.r.t. the ideal $\dar\pi_{ij}(s_{ik})$.   
\end{itemize}

\medskip  {\small
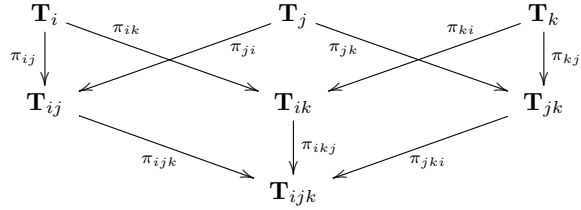
\hspace*{10em}
\xymatrix @R=2em @C=7em{
 \gT_i\ar[d]_{\pi _{ij}}\ar@/-0.75cm/[dr]^{\pi _{ik}} &
     \gT_j\ar@/-.8cm/[dl]_{\pi _{ji}}\ar@/-.8cm/[dr]^{\pi _{jk}} &
        \gT_k\ar@/-0.75cm/[dl]_{\pi _{ki}}\ar[d]^{\pi _{kj}} &
\\
 ~\gT_{ij}~ \ar[rd]_{\pi _{ijk}} & 
    ~\gT_{ik}~  \ar[d]^{\pi _{ikj}} & 
      ~\gT_{jk}~  \ar[ld]^{\pi _{jki}} 
\\
   &  ~\gT_{ijk}~ 
\\
}
}

\smallskip \noindent Let $\big(\gT\,;\,(\pi_i)_{i\in I}\big)$ be the limit of the diagram. Then there exist $s_i$'s in $\gT$ such that the principal ideals $\dar s_i$ cover $\gT$ and the diagram is isomorphic to the one in Fact~\ref{factRecolTD}.
More precisely  each~$\pi_i$ is a quotient morphism  w.r.t. the \id $\dar s_i$ and $\pi_i(s_j)=s_{ij}$ for all $i\neq j$.

\noindent The analogous result works with quotients by principal filters.
\end{proposition}

\begin{remark} \label{remRecSSSpec} 
{\rm The reader can translate the previous result in gluing of spectral spaces.   
}\end{remark}

\Subsubsection{Heitmann lattice and J-spectrum} 

\smallskip \noindent An \id $\fm$ of a  \trdi $\gT$ is
 \textsl{maximal} when $\gT/(\fm=0) \simeq \Deux$, \cad if $1\notin\fm$ and
$\Tt x\in\gT\;(x\in\fm$ or $\Ex y\in \fm \;x\vu y=1)$.

In \clama we have the following result.
\begin{lemma}
\label{lemHspec1}
The intersection  of all \idemas containing
an \id  $\fJ$ is  called 
the \emph{Jacobson radical of  $\fJ$} and is equal to 
\begin{equation} \label{eqRJJ}
\JT(\fJ)\,=\, \sotq{a\in\gT}{\forall x\in\gT \;(a\vu x  = 1 
\Rightarrow \Ex
z\in \fJ \;\;z\vu x=1)}.
\end{equation}
We write $\JT(b)$ for $\JT(\dar b)$.
The ideal $\JT(0)$ is the \emph{Jacobson radical of $\gT$}.
\end{lemma}

In \coma,  \egt \pref{eqRJJ} is used as \dfn.

The \textsl{Heitmann lattice of $\gT$}, denoted by $\He(\gT)$, 
is the quotient of~$\gT$ corresponding to the following preorder 
$\preceq_{\He(\gT)}$: 
\begin{equation} \label{eqdefHeT}
\begin{array}{rclcl}\qquad 
a\preceq_{\He(\gT)} b & \equidef  &   \JT(a)\subseteq\JT(b)&\Longleftrightarrow&a\in \JT(b)  
  \end{array}
  \end{equation}
Elements of $\He(\gT)$ can be identified with 
\ids $\JT(a)$, via the canonical map
$$ \gT\longrightarrow \He(\gT),\quad a\longmapsto \JT(a)$$

The next definition follows the remarkable paper \cite{Hei84}. 
In this paper, R. Heitmann explains that the usual  notion
of j-spectrum of a commutative ring is not the good one in the non-Noetherian case because it does not correspond to a spectral space in the sense of Stone. He introduces the following modification of the usual \dfn: instead of considering the set of \ideps which are intersections of \idemas, he proposes to consider the closure of the maximal spectrum in the prime spectrum, the closure being taken for the patch topology.

\begin{definition}
\label{defHspec1}
Let $\gT$ be a \trdi.
\begin{enumerate}
\item The \textsl{maximal spectrum of~$\gT$}, denoted by  $\Max\gT$, is the topological subspace of $\Spec\gT$ provided by the \idemas of~$\gT$. 
\item  The \textsl{j-spectrum of $\gT$}, denoted by $\jspec\gT$, is the topological subspace of $\Spec\gT$ provided by the primes $\fp$ s.t. \hbox{$\JT(\fp)=\fp$}, \cad the \ideps $\fp$ 
which are intersections of \idemas.
\item The \textsl{Heitmann $\rJ$-spectrum  of $\gT$}, denoted by 
$\Jspec\gT$, is the closure of $\Max\gT$ in $\Spec\gT$ 
for the patch topology. It is a spectral subspace of $\Spec\gT$, canonically homeomorphic to $\Spec(\He(\gT))$. 
\item The \textsl{minimal spectrum of~$\gT$}, denoted by $\Min\gT$, is the topological subspace of $\Spec\gT$ provided by \idemis of~$\gT$. 
\end{enumerate}
\end{definition}

In general $\Max\gT$,  $\jspec\gT$ 
and $\Min\gT$ are not spectral spaces.

\begin{theorem}
\label{thDK3}
$\Jspec\gT$ is a spectral subspace of 
$\Spec\gT$ canonically homeomorphic to $\Spec(\He(\gT))$. 
\end{theorem}

\subsection{Distributive lattices and \entrels}

A particularly important rule for distributive lattices, known as
 \textsl{cut}, is
\begin{equation}\label{coupure1}
 \bigl(\,x\vi a\; \leq\;  b\,\bigr)\quad\&\quad  \bigl(\,a\; \leq\; x\vu  b\,\bigr)
\quad \Longrightarrow \quad a \leq\;  b.
\end{equation}

For $A\in\Pfe(\gT)$ (finitely enumerated subsets of~$\gT$)  we write
$$\ndsp \Vu  A:=\Vu _{x\in A}x\qquad {\rm and}\qquad \Vi  A:=\Vi _{\!x\in A}x.
$$
We denote by $A \vda B$ or $A \vdash_\gT B$ the relation defined as follows over the set $\Pfe(\gT)$
$$
A \vda B \; \; \equidef\; \; \Vi A\;\leq \;
\Vu B.
$$

This relation satisfies the following axioms, in which we write $x$ for $\{x\}$ and $A, B$  for $A\cup B$.
$$\arraycolsep3pt\begin{array}{rcrclll}
&    & a  &\vda& a    &\; &(R)     \\[1mm]
A \vda B &   \; \Longrightarrow \;  & A,A' &\vda& B,B'   &\; &(M)     \\[1mm]
(A,x \vda B)\;\;
\&
\;\;(A \vda B,x)  &  \; \Longrightarrow \; & A &\vda& B &\;
&(T).
\end{array}$$
\rdb
We say that the relation is \textsl{reflexive}, \label{remotr} \textsl{monotone} and
\textsl{transitive}.
The third rule (transitivity) can be seen as a version of  rule \pref{coupure1} and is also called the \textsl{cut} rule.

\begin{definition}
\label{defEntrel}
For an arbitrary set $S$, a relation over $\Pfe(S)$ which is reflexive, monotone and transitive is called an \textsl{\entrel}.
\end{definition}

The following \tho is fundamental. It says that the three \prts of \entrels are exactly what is needed for the interpretation in the form of a \trdi to be adequate.

\begin{theorem}[Fundamental \tho of  \entrels] \label{thEntRel1} 
See  \cite[Theorem 1]{CC00}, \cite[\hbox{XI-5.3}]{CACM}, \cite[Satz 7]{Lor1951}. \\
Let $S$ be a set with an \entrel $\vdash_S$ on $\Pfe(S)$. We consider the \trdi~$\gT$ defined by \gtrs and relations as follows: the \gtrs are the \elts of $S$ and the relations are the
$$ A\; \vdash_\gT \;  B
$$
each time that $A\; \vdash_S \; B$.  Then, for all $A$,
 $B$ in $\Pfe(S)$, we have
$$  A\; \vdash_\gT \;  B
\; \Longrightarrow \; A\; \vdash_S \;  B.
$$
\end{theorem}

\smallskip 
\rem The relation $x\vdash_S y$ is a priori a preorder, and not an order, on~$S$. Let us denote by $\ov x$ the \elt $x$ seen in the ordered set~$\ov{S}$ defined by this preorder. For a subset~$A$ of~$S$ let us denote $\ov A=\sotq{\ov x}{x\in A}$.
In the \tho we consider a \trdi $\gT$ which gives on $S$ the same \entrel as $\vdash_S$.
Strictly speaking, we should have written  $ \ov A\; \vdash_\gT \; \ov B$
instead of $A\; \vdash_\gT \;  B$ since the \egt in $\gT$ is coarser than in $S$. In particular, it is $\ov{S}$, and not $S$, which can be identified with a subset of $\gT$.
\eoe

\section{Spectral spaces in \alg}

The usual spectral spaces in algebra are (always?) understood as spectra of distributive lattices associated to coherent theories describing relevant
algebraic structures.
We describe this general situation and give some examples.

\subsection{Dynamical algebraic structures, \trdis and associated spectra}

References: 
\cite{CLR01} and \cite{Lom-tgac}.
The paper \cite{CLR01} introduces the general notion of \gui{dynamical theory} and of \gui{dynamical proof}.
See also the paper \cite{BC2005} which illustrates the usefulness of these notions.

\Subsubsection{Dynamical theories and \sads}

\smallskip  \noindent Dynamical theories are a version \gui{without logic, purely computational} of \textsl{coherent theories} (here, we say \gui{theory} for \gui{first order formal theory}).

Dynamical theories use only dynamical rules, \cad deduction rules of the form
\[
\Gamma  \vd   \Exists{\und{y^1}}\, \Delta_1
\vou \cdots\vou \Exists{\und{y^m}}\,\Delta_m
\]
where $\Gamma$ and the $\Delta_i$'s are lists of atomic formulae in the language $\cL$ of the theory $\sa{T}=(\cL,\cA)$.

The computational meaning of \gui{$\Exists{\und{y}}\, \Delta$} is \gui{$\EXists{\hbox{fresh variables\ }\und{y}}\, \Delta$}.
The computational meaning of \gui{$ U\;\tsbf{or}\; V\;\tsbf{or}\; W$} is \gui{open three branches of computations, in the first one, the predicates in $U$ are valid \dots}. 

Axioms (elements of $\cA$) are dynamical rules, and theorems are valid dynamical rules (validity is described in a simple way and uses only a computational machinery) 

A \textsl{\sad} for a \tdy $\sa{T}$ is given through a presentation $(G,R)$ by generators and relations. Generators are the element of $G$ and they are added to the constants in the language. Relations are the elements of $R$. 
They are dynamical rules without free variables and they are added to the axioms of the theory.

A \sad is intuitively thought of as an incompletely specified algebraic structure. The notion corresponds to lazy evaluation in Computer Algebra.

Purely equational algebraic structures correspond to the case where the only predicate is equality and the axioms are Horn rules.

Dynamical theories whose axioms  contain neither $\vou$ nor $\,\Exists\,$ nor $\Bot$ are called \textsl{Horn theories} (algebraic theories in \cite{CLR01}).
E.g.\ theories of  absolutely flat rings and of pp-rings can be given as Horn theories.  

\smallskip
A coherent theory is a \textsl{finitary geometric theory}. In general geometric theories
we accept  dynamical rules that use infinite disjunctions at the right of $\vd$. In this paper we speak only of finitary \tgms.

A fundamental result about finitary \tdys says that adding the classical  logic to a \tdy does not change valid rules: first order classical logic is conservative over \tdys \cite[Theorem 1.1]{CLR01}.  

\Subsubsection{Distributive lattices associated to a \sad}

\smallskip \noindent Let  $\gA=\big((G,R),\sa{T}\big)$ be a \sad for $\sa{T}=(\cL,\cA)$.

\paragraph{First example.} If $P(x,y)$ belongs to $\cL$ and if  $\tcl=\mathrm{Clt}(\gA)$ is the set of closed terms of~$\gA$, we get the following \entrel $\vdash_{\gA,P}$ for $\tcl\times \tcl$: 
\begin{equation} \label {eq1}
\begin{aligned} 
 (a_1,b_1),\dots,(a_n,b_n)  &\,\,\vdash_{\gA,P}    (c_1,d_1),\dots,(c_m,d_m)   
  \qquad\quad     \equidef     \\[.3em] 
   P(a_1,b_1)\vet \dots\vet  P(a_n, b_n) & \Vdi{\gA}   P(c_1, d_1)\vou \dots\vou   P(c_m, d_m) 
 \end{aligned}
\end{equation}

Intuitively the \trdi $\gT$ generated by this \entrel represents the \gui{truth values} of~$P$ in the \sad $\gA$.
In fact to give an \elt $\alpha:\gT\to\Deux$ of $\Spec\gT$ amounts to giving the value $\Top$ (resp. $\Bot$) to $P(a,b)$ 
when $\alpha(a,b)=1$  (resp. $\alpha(a,b)=0$). 

\paragraph{Second example, the Zariski lattice of a commutative ring.}~

\smallskip \noindent Let \sa{Lr} be a \tdy of nontrivial local rings, e.g.\ with signature $$(\,\cdot=0,\U(\cdot)\mathrel{;}\cdot+\cdot,\cdot\times \cdot,-\,\cdot,0,1\,).$$
This is an extension of the \peq theory of commutative rings.
The predicate $\U(x)$ is defined as meaning the invertibility of $x$, 

\DeuxRegles{
\labu $\,\,\U(x)\vd \Exists y\,  xy=1$
}
{
\labu $\,\,xy=1\vd \U(x)$
}

\noindent and the axioms of nontrivial local rings are written as

\DeuxRegles {\Lab{LR} $\,\, \U(x+y) \Vd \U(x) \vou \U(y)$
}
{
\labu $\,\,\U(0)\vd \,\Bot$}
  
The \textsl{Zariski lattice} $\ZarA$
of a commutative ring $\gA$ 
is defined as the \trdi generated by the \entrel $\vdash_{\ZarA}$ for $\gA$ defined as
\begin{equation} \label {eqZarclass}
\begin{aligned} 
 a_1,\dots,a_n  &\,\,\vdash_{\ZarA}    c_1,\dots,c_m   
  \qquad\quad     \equidef     \\[.3em] 
\U(a_1)\vet \dots\vet  \U(a_n) & \Vdi{\sA{Lr}(\gA)}   \U(c_1)\vou \dots\vou   \U(c_m) 
 \end{aligned}
\end{equation}

Here $\sa{Lr}(\gA)$ is the \sad of type $\sa{Lr}$ over $\gA$.

We get the following equivalence (we call it a \textsl{formal \nst}): 
$$
a_1,\dots,a_n \,\vdash_{\ZarA}    c_1,\dots,c_m \Longleftrightarrow \exists k>0\;\;(a_1 \cdots a_n)^k\in\gen{c_1,\dots,c_m}
$$

So,  $\ZarA$ can be identified with the set of ideals $\DA(\ux)=\sqrt[\gA]{\gen{\ux}}$, with
$\DA(\fj_1)\vi\DA(\fj_2)=\DA(\fj_1\fj_2)$ and
$\DA(\fj_1)\vu\DA(\fj_2)=\DA(\fj_1+\fj_2)$.

Now, the usual Zariski spectrum  $\SpecA$ is canonically homeomorphic to $\Spec(\ZarA)$. Indeed, to give a point of $\SpecA$ (a \idep) amounts to giving an epimorphism $\gA\to\gB$ where $\gB$ is a local ring, 
or also, that is the same thing, to giving a minimal model of $\sa{Lr}(\gA)$.
This corresponds to the intuition of \gui{forcing the ring to be a local ring}.

\paragraph{More generally.} Let us consider a set $S$ of closed atomic formulae of the \sad~$\gA=\big((G,R),\sa T\big)$. We define a corresponding \entrel (with the $A_i$'s and $B_j$'s in $S$):

\vspace{-.6em}
\begin{equation} \label {eq2}
\begin{aligned}
 A_1,\dots,A_n  &\,\,\vdash_{\gA,S} B_1,\dots,B_m   
   \qquad\quad \equidef     \\[.3em] 
    A_1\vet \dots\vet  A_n &\Vdi{\gA} B_1\vou\dots\vou B_m. 
 \end{aligned}
\end{equation}
We may denote by $\Zar(\gA,S)$ this \trdi.

\paragraph{Points of a spectrum and models in \clama.} 
With a good choice of predicates in the language, to give a point of the spectrum of the corresponding lattice amounts often to giving a minimal model of the \sad.
This is the case when all existence axioms in the theory imply unique existence.
The topology of the spectrum is in any case strongly dependent on the choice of predicates.

\paragraph{The complete Zarisiki lattice of a \sad $\gA$} is defined by choosing for $S$ the set 
 $\Atcl(\gA)$ of all closed atomic formulas of~$\gA$. 
When the theory has no existential axioms, this lattice corresponds to the \entrel for  $\Atcl(\gA)$ generated by the axioms of $\sa{T}$, replacing the variables with arbitrary closed terms of $\gA$.

\subsection{A very simple case} 
\label{subsecspecbanal}

Let $\sa{T}$ be a \talg.
Any \sad $\gS=((G,R),\sa{T})$ of type \sa{T} defines an ordinary \agq structure $\gB$ and there is no significant difference between \sads and ordinary \agq structures.

The minimal models df $\gS$ are (identifiable with) the quotient structures
 $\gC=\gB /\!\!\sim$.
Let us consider the complete Zariski lattice $\Zar (\gB) $.
The points of \(\Spec(\gB):=\Spec(\Zar(\gB))\) are  identified with these quotient structures of $\gB$. An element of \(\Zar(\gS)=\Zar(\gB)\) can be identified with a quotient of \(\gB\) when we add only a finite number of relations in the \pn.  

\smallskip For example, let us consider the \tpe $\sa{T}=\sa{Mod}_\gA$ (the theory of modules over a fixed ring  $\gA$). For an \(\gA\)-module \(M\) we get the lattice $\Zar (M)$ generated by the  \entrel $\ym\vdash \xn$ on~$M$ defined by:

\Regles{\labu $\,\, y_1= 0,\dots\vet y_m= 0\vd x_1= 0\vou\dots\vou x_n= 0$}

\noindent or by:
 
\Regles{\labu 
$\,\,$one $x_k$ is in the module $\gA y_1+\dots+\gA y_m=\gen{\ym}$ (\textsl{formal \nst for \lin \alg}).}

Points of \(\Spec(M)\) can be identified with the quotient structures \(M/N\) of \(M\). Elements of \(\Zar M\) define the \oqcs in \(\Spec(M)=\Spec(\Zar M)\). Finitely generated submodules of \(M\)  give in this way a basis of \oqcs of \(\Spec(M)\).
$$
\fD(y_1\vi\dots\vi y_m):=\fD(y_1)\cap\dots\cap\fD(y_m) =\sotQ{N\subseteq M}{\gen{y_i}_{i\in\lrbm}\subseteq N} \quad  \hbox{where }y_i\hbox{'s}\in M.
$$ 

\smallskip It might be that these  kinds of lattices and spectra are too simple to lead to interesting results in algebra.

\subsection{The real spectrum of a commutative ring} 
\label{subsecspecreel}

The real spectrum  $\Sper\gA$ of a commutative ring corresponds to the intuition of \gui{forcing the ring $\gA$ to be an ordered (discrete\footnote{We ask the order relation to be decidable.})} field. 

A point of $\Sper\gA$ can be given as an epimorphism
 \hbox{$\varphi:\gA\to\gK$}, where $(\gK,\gC)$ is an ordered field.\footnote{$\gC$ is the cone of nonnegative \elts.} Moreover two such morphisms $\varphi:\gA\to\gK$ and $\varphi':\gA\to\gK'$ define the same point of the spectrum if  there exists an \iso of ordered fields $\psi:\gK\to\gK'$ making the suitable diagram commutative.

We write \gui{$x\geq 0$} the predicate over  $\gA$ corresponding to \gui{$\varphi(x)\geq 0$ in $\gK$}. We get the following axioms:

\DeuxRegles{
\labu $\vd x^2\geq 0$
\labu $\,\,x\geq0\vet y\geq0\vd x+y\geq0$
\labu $\,\,x\geq0\vet y\geq0\vd xy\geq0$
}
{
\labu $\,\,-1\geq0\vd \Bot$
\labu $\,\,-xy\geq0\vd x\geq 0\vou y\geq 0$
}

This means that $\sotq{x\in\gA}{x\geq 0}$ is a prime cone: to give a model of this theory is the same thing as to give a point of $\Sper\gA$.

In order to get the usual topology of $\Sper\gA$, it is necessary to use the opposite predicate $x<0$. For the sake of comfort, we take $x>0$.  
This predicate satisfies the dual axioms to those for $-x\geq 0$:

\DeuxRegles{
\labu $\,\,-x^2>0\vd \Bot$
\labu $\,\,x+y>0\vd x>0\vou y>0$
\labu $\,\,xy>0\vd x>0 \vou -y>0$
}
{
\labu $\vd 1 >0$
\labu $\,\,x>0\vet y>0\vd xy>0$
}

So the \textsl{real lattice} of $\gA$, denoted by $\Reel(\gA)$,
is the \trdi generated by the minimal  \entrel for $\gA$ satisfying  the following relations  (we write $\rR(a)$ instead of $a$):

\DeuxRegles{
\labu $\,\,\rR(-x^2) \vdash  $
\labu $\,\,\rR(x+y) \vdash \rR(x ), \rR(y) $
\labu $\,\,\rR(xy) \vdash \rR(x)  , \rR(-y) $
}
{
\labu $\,\,\vdash \rR(1)  $
\labu $\,\,\rR(x) , \rR(y) \vdash \rR(xy) $
}

So $\Spec(\Reel\gA)$ is isomorphic to $\Sper\gA$, viewed as the set of prime cones of $\gA$. The spectral topology
admits the basis of open sets $$\fR(\an)=\sotq{\fc\in \Sper\gA}{\&_{i=1}^n -a_i\notin \fc}.$$

This approach to the real spectrum was proposed in~\cite{CC00}. 

An important point is the following \textsl{formal Positivstellensatz}.
\begin{theorem}[Formal Positivstellensatz for ordered fields] \label{thPstformelreel}~ 
\Propeq 
\begin{enumerate}
\item We have $
\rR(x_1),\dots,\rR(x_k)\vdash\rR(a_1),\dots,\rR(a_n) 
$
in the lattice $\Reel \gA$.
\item We have $
x_1>0\vet\dots \vet x_k>0\vd a_1> 0\vou\dots\vou  a_n> 0 
$ 
in the theory of ordered fields over  $\gA$.
\item We have $
x_1>0\vet\dots \vet x_k>0\vet a_1\leq  0\vet\dots\vet  a_n\leq  0\vd \Bot 
$ 
in the theory of ordered fields over~$\gA$.
\item We have an \egt $s+p=0$  in $\gA$, with $s$ in the  \mo generated by the $x_i$'s and~$p$ in the cone generated by the~$x_i$'s and the~$-a_j$'s.
\end{enumerate}
\end{theorem}
%

\subsection{Linear spectrum of a lattice-group} 

The theory of lattice-groups, denoted by \sa{Lgr} is a \tpe over the signature \hbox{$(\cdot=0;\cdot+\cdot,-\cdot,\cdot\vu\cdot,0)$}.
The following rules express that $\vu$ defines a join semilattice and the compatibility of $\vu$ with $+$:

\DeuxRegles{
\lab{sdt1} $\vd x\vu x=x $
\lab{sdt2} $\vd  x\vu y=y\vu x$
}
{
\lab{sdt3} $\vd  (x\vu y)\vu z=x\vu (y \vu z)$
\lab{grl} $\vd  x+(y\vu z)=(x+y)\;\vu\;(x+z)$
}

We get  the theory \sa{Liog} by adding to \sa{Lgr} the axiom $\Vd x\geq 0 \vou -x\geq 0$.

The linear spectrum of an \grl $\Gamma$ corresponds to the intuition of \gui{forcing the group to be linearly ordered}. So a point of this spectrum can be given as a minimal model of the \sad $\sa{Liog}(\Gamma)$, or equivalently by a linearly ordered group $G$ quotient of~$\Gamma$, or as the kernel $H$ of the canonical morphism~\hbox{$\pi:\Gamma\to G$}.
This subgroup $H$ is a \textsl{prime solid subgroup} of $\Gamma$.

The \textsl{linear lattice} of $\Gamma$, denoted by $\Glio(\Gamma)$, is generated by the \entrel for $\Gamma$ defined in the following way:
\[ 
\begin{aligned} 
 a_1,\dots,a_n    &\,\,\vdash_{\Glio \Gamma } \;b_1,\dots,b_m   
\qquad  \equidef     \\ 
 a_1\geq 0\vet \dots\vet a_n\geq 0 &\vdi{\sA{Liog}(\Gamma)} 
 \, b_1\geq 0\vou\dots\vou b_m\geq 0
 \end{aligned}
\]
The spectral space previously defined is (isomorphic to) $\Spec(\Glio\Gamma)$.   
We have a \textsl{formal \pst} for this \entrel ($m,n\neq 0$): with \(a^-= (-a) \vu 0\), we get
\[ 
\begin{aligned} 
 a_1,\dots,a_n\vdash_{\Glio(\Gamma)} b_1,\dots,b_m   
\;\Longleftrightarrow\;      
\exists k>0\;     (b_1^- \vi \dots\vi b_m^-)\leq k(a_1^- \vu \dots\vu a_n^-)
 \end{aligned}
\]

\subsection{Valuative spectrum of a commutative ring}
\label{subsecspecval}

The valuative spectrum  $\Spev\gA$ of a commutative ring
corresponds to the intuition of \gui{forcing the ring to be a valued field}. 
A point of this spectrum is given by an epimorphism  $\varphi:\gA\to\gK$ where $(\gK,\gV)$ is a valued field.\footnote{$\gV$ is a valuation ring of $\gK$.}  Moreover two such morphisms $\varphi:\gA\to\gK$ and $\varphi':\gA\to\gK'$ define the same point of the spectrum if  there exists an \iso of valued fields $\psi:\gK\to\gK'$ making the suitable diagram commutative.

We denote by $x\di y$ the predicate over  $\gA\times \gA$ corresponding to \gui{$\varphi(x)$ divides\footnote{I.e., $\exists z\in \gV\,z\varphi(x)=\varphi(y)$.} $\varphi(y)$ in $\gK$}. We get the following axioms for the dynamical theory \(\sa{val}\):

\DeuxRegles{
\labu $\vd 1 \di  0$
\labu $\vd -1 \di  1$
\labu $\,\,a \di  b \vd ac \di  bc$
\labu $\vd a \di  b \vou b\di a$
}
{
\labu $\,\,0 \di  1\vd \Bot$ 
\labu $\,\,a \di  b \vet b \di  c \vd a \di  c$
\labu $\,\,a \di  b \vet a \di  c \vd a \di  b + c$
\labu $\,\,ax \di  bx  \vd a \di  b \vou 0 \di x$
}

 Let \(\gA\) be a commutative ring. The \sad $\sa{val}(\gA)$ is defined by adding the positive diagram of the commutative ring \(\gA\).

Any predicate $x\di y$ over $\gA\times \gA$ satisfying the axioms of \(\sa{val}(\gA)\) defines a point in $\Spev\gA$. 
So we define the \textsl{valuative lattice} of $\gA$, denoted by $\val(\gA)$
as generated by the minimal \entrel for $\gA\times \gA$ satisfying the following relations.

\DeuxRegles{
\labu $\,\,\vdash (1 ,  0)$
\labu $\,\,\vdash (-1 ,  1)$
\labu $\,\,(a ,  b) \vdash (ac ,  bc)$
\labu $\,\,\vdash (a ,  b) , (b, a)$
}
{
\labu $\,\,(0,1)\vdash $  
\labu $\,\,(a ,  b) , (b ,  c) \vdash (a ,  c)$
\labu $\,\,(a ,  b) , (a ,  c) \vdash (a ,  b + c)$
\labu $\,\,(ax ,  bx)  \vdash (a ,  b) , (0 , x)$
}

The two spectral spaces $\Spec(\val \gA )$ and $\Spev\gA$ can be identified. The spectral topology of $\Spec(\val \gA )$ corresponds in $\Spev\gA$ to the topology generated by the basic open sets  $\fD(a,b)=\sotq{\varphi\in \Spev\gA}{\varphi(a)\di \varphi(b)}.$

We have a formal Valuativstellensatz.

\begin{theorem}[Formal Valuativstellensatz 
for a \sad \(\sa{val}(\gA)\)] \label{thPstformelVal}~\\
Let $\gA$ be a commutative ring, \propeq 
\begin{enumerate}
\item One has $
\,\,  a_1\di b_1\vet \dots\vet a_n\di b_n  \vd    c_1\di d_1 \vet \dots\vet c_m\di d_m  
$  in the \sad~\(\sA{val}(\gA)\).
\item One has $
 (a_1,b_1),\dots,(a_n,b_n)  \,\vdash \,   (c_1,d_1),\dots,(c_m,d_m) 
$
in the lattice $\val\gA$.
\item Introducing \idtrs $x_i$'s ($i\in\lrbn$) and $y_j$'s ($j\in\lrbm$)
we have in the ring  $\gA[\ux,\uy]$ an \egt  
$$
 d \big(1+\som_{j=1}^my_jP_j(\ux,\uy)\big)\in \gen{(x_ia_i-b_i)_{i\in\lrbn},(y_jd_j-c_j)_{j\in\lrbm}} 
$$ 
where $d$ is in the  \mo generated by the $d_j$'s, and the $P_j(\xn,y_1,\alb\dots,y_m)$'s are in $\ZZ[\ux,\uy]$.
\end{enumerate}
\end{theorem}

\subsection{Heitmann lattice and J-spectrum of a commutative ring}

In a commutative ring the \textsl{Jacobson radical of an \id 
$\fJ$} denoted by $\JA(\fJ)
$ is defined in \clama as the intersection of the \idemas containing~$\fJ$. In \coma we use the classically equivalent \dfn
\begin{equation} \label{eqRadJac}
\JA(\fJ)\eqdefi\sotq{x\in\gA}{\Tt y\in\gA,\;\; 1+xy
\hbox{ is invertible modulo } \fJ}.
\end{equation}
We write $\JA(x_1,\ldots ,x_n)
$ for
$\JA(\gen{x_1,\ldots ,x_n})$.  The \id  $\JA(0)$ is called
\textsl{the Jacobson radical of~$\gA$}.

The \textsl{Heitmann lattice of $\gA$} is
$\He(\ZarA)$, denoted by $\HeA$; it is a quotient of $\ZarA$.  In fact $\HeA$ can be identified with the set of ideals $\JA(x_1,\ldots ,x_n)$, with
$\JA(\fj_1)\vi\JA(\fj_2)=\JA(\fj_1\fj_2)$ and
$\JA(\fj_1)\vu\JA(\fj_2)=\JA(\fj_1+\fj_2)$.

We denote by $\Jspec(\gA)$ the spectral space $\Spec(\HeA)$. In \clama it is the closure (for the patch topology) of the maximal spectrum in $\SpecA$. We call it the (Heitmann) J-spectrum of $\gA$. It is a subspectral space of $\SpecA$.
When $\gA$ is \noe, $\Jspec(\gA)$ coincides with the subspace
 $\jspec(\gA)$ of $\SpecA$ made of the \ideps
which are  intersections of \idemas.

\smallskip \rem 
 $\JA(x_1,\ldots ,x_n)$ is a radical ideal 
but not generally the radical of a \itf. \eoe

\section{A short dictionary}
\label{secAntiEquiv}

References: \cite{BW74,CC00,CL2001-2018}.

\smallskip In this section we consider the following context:  $f:\gT\to\gT'$ is a morphism of \trdis and $\Spec(f)$, denoted by $f\sta: X'=\Spec\gT'\to X=\Spec\gT$, is the dual morphism.

When results are given for \trdis they  are also valid for commutative rings, considering each time the lattice or the morphism $\Zar(\bullet)$.

\smallskip We first recall some usual \dfns in \clama.
\begin{itemize}
\item The morphism $f$ is said to be  \textsl{lying over} when~$f\sta$ is onto: any \idep of $\gT$ is the preimage of a \idep in~$\gT'$.
\item The morphism $f$ is said to be  \textsl{going up} when one has: \textsl{if $\fq\in X'$, $f\sta(\fq)=\fp$, and $\fp\subseteq\fp_2$ in~$ X $, then there exists  $\fq_2\in X'$ such that
$\fq\subseteq\fq_2$ and $f\sta(\fq_2)=\fp_2$}.
\item In a similar way $f$ is said to be  \textsl{going down} when one has: \textsl{if $\fq\in X'$, $f\sta(\fq)=\fp$, and $\fp\supseteq\fp_2$ in~$ X $, then there exists  $\fq_2\in X'$ such that $\fq\supseteq\fq_2$ and  $f\sta(\fq_2)=\fp_2$}.
\item The morphism $f$ \textsl{has the incomparability property} when one has: if $\fq_1\subseteq \fq_2\in X$ and $f\sta(\fq_1)=f\sta(\fq_2)$ in~$X'$ then $\fq_1= \fq_2$.
\item  The spectral space $ X $ is said to be  \textsl{normal} if for all \(x\), the closure  \(\ov{\so z}\) contains a unique closed point.
\item The spectral space  $\SpecT$ is said to be  \textsl{completely normal} if for all $x,y,z$ such that $x\in\ov{\so z}$ and $y\in\ov{\so z}$ one has $x\in\ov{\so y}$ or $y\in\ov{\so x}$. 
\end{itemize}

\subsection{Properties of morphisms}

\begin{theorem} \label{th-dico-trdi-spec-mor1} \emph{\cite[Theorem~IV-2.6]{BW74}}
In \clama we have the following equivalences.   
\begin{enumerate}
\item $f\sta$ is onto ($f$ is lying over) $\Longleftrightarrow$  $f$ is injective $\Longleftrightarrow$ $f$ is a monomorphism $\Longleftrightarrow$ $f\sta$ is an epimorphism.
\item $f$ is an epimorphism $\Longleftrightarrow$ $f\sta$ is a monomorphism $\Longleftrightarrow$ $f\sta$ is injective.
\item $f$ is onto\footnote{In other words, $f$  is a quotient morphism.} $\Longleftrightarrow$ $f\sta$ is an \iso on its image, which is a subspectral space of $X$.
\end{enumerate} 
\end{theorem}

There are bijective morphisms of spectral spaces that are not isomorphisms.  E.g.\ the morphism  $\Spec(\Bo(\gT))\to\Spec\gT$ is rarely an \iso and the lattice morphism $\gT\to\Bo(\gT)$ is an injective epimorphism which is rarely onto.

\begin{lemma} \label{lemLYO}
Let  $S$ be a  system of generators for $\gT$. The morphism  $f$ is lying over \ssi for all
$a_1,\dots,a_n,b_1,\dots,b_m\in S$ we have
$$
f(a_1),\dots,f(a_n)\vdi {\gT'} f(b_1),\dots,f(b_m)
\;\;\Rightarrow\;\;
a_1,\dots,a_n\vdi {\gT} b_1,\dots,b_m
$$ 
\end{lemma}

\begin{proposition}[Going up versus lying over] \label{propGu} 
In \clama \propeq (see \cite{CL2001-2018}):
\begin{enumerate}
\item  For each \idep $\fq$ of $\gT' $ and $\fp=f^{-1}(\fq)$, the morphism 
   $f':\gT/(\fp=0)\to \gT'/(\fq=0)$  
is lying over.
\item  For each ideal $I$ of $\gT'$ and $J:=f^{-1}(I)$, the morphism 
$f_I:\gT/(J=0)\to \gT'/(I=0)$ is lying over. 
\item  For each  $y\in \gT'$ and $J=f^{-1}(\dar y)$, the morphism 
$f_y:\gT/(J=0)\to \gT'/(y=0)$ is lying over. 
\item  For each $a,c\in \gT$ and $y\in \gT'$ we have
$$
f(a)\,\vdash_{\gT'}\, f(c),\, y \quad\Longrightarrow\quad\exists x\in \gT \;\;\; a \vdash_{\gT}\, c ,\, x \;\;\; \hbox{and}\;\;\; f(j)\leq_{\gT'} x.
$$
\end{enumerate} 
\end{proposition}

\begin{theorem} \label{th-dico-trdi-spec-mor2} 
In \clama we have the following equivalences  (\cite{CL2001-2018}).   
\begin{enumerate}
\item $f$ is going up $\Longleftrightarrow$  for each $a,c\in\gT$ and $y\in\gT'$ we have
$$
f(a)\leq f(c)\vu y \;\Rightarrow\;\exists x\in\gT\; (a\leq c \vu x \hbox{ and } f(x)\leq y).
$$ 
\item $f$ is going down $\Longleftrightarrow$  for each $a,c\in\gT$ and $y\in\gT'$ we have
$$
f(a)\geq f(c)\vi y \;\Rightarrow\;\exists x\in\gT\; (a\geq c \vi x \hbox{ and } f(x)\geq y).
$$
\item $f$ has the \prt of incomparability $\Longleftrightarrow$ $f$ is zero-dimensional.\footnote{See \Tho \ref{th-dico-trdi-spec-dim2}.}   
\end{enumerate} 
\end{theorem}
%

\begin{theorem} \label{th-dico-trdi-spec-mor3} 
In \clama \propeq   
\begin{enumerate}
\item $\Spec(f)$ is an open map.
\item  There exists a map $\wi f:\gT'\to \gT$
with the following \prts.
\begin{enumerate}
\item \label{i2a} For  $c\in\gT$ and $b\in\gT'$, one has $b\leq f(c) \Leftrightarrow \wi f(b)\leq c$. \\
In particular, $b\leq f(\wi f(b))$ and $\wi f(b_1\vu b_2)=\wi f(b_1)\vu \wi f(b_2)$.
\item \label{i2b} For  $a,c\in\gT$ and $b\in\gT'$, one has \hbox{$f(a)\vi b\leq f(c) \Leftrightarrow a\vi\wi f(b)\leq c $}.
\item \label{i2c} For  $a\in\gT$ and $b\in\gT'$, one has $\wi f(f(a)\vi b)=a\vi \wi f(b)$.
\item \label{i2d} For  $a\in\gT$, one has $\wi f(f(a))=\wi f(1)\vi a$. 
\end{enumerate}
\item  There exists a map $\wi f:\gT'\to \gT$
satisfying \prt \ref{i2b}.  
\item  For  $b\in \gT$ the g.l.b.\ {$\Vi\limits_{b\leq f(c)} c$} exists, and if we write it $\wi f(b)$, 
the \prt \ref{i2b} holds.  
\end{enumerate} 
\end{theorem}
For this result in locales' theory see \cite[section 1.6]{Bor3}. We give now a proof for spectral spaces.
Implications concerning item \textsl{1} need \clama. The other equivalences are \cov.

\begin{lemma} \label{lemcroiss}
Let $f:A\to A'$ be a nondecreasing map between ordered sets ${(A,\leq )}$ and $(A',\leq')$ and $b\in A'$. An \elt $b_1\in A$ satisfies the equivalence 
\[ \forall x\in A\,\,\, (\,b\leq ' f(x) \,\Longleftrightarrow\, b_1 \leq  x\,)
\]
\vspace{-1.5em}

\noindent \ssi 
\begin{itemize}
\item on the one hand $b\leq' f(b_1)$, 
\item and on the other hand 
{$b_1= \Vi_{x:b\leq' f(x)}x$}. 
\end{itemize}
In particular, if~$b_1$ exists,  it is uniquely determined.   
\end{lemma}
%
\begin{proof}
If $b_1$ satisfies the equivalence, one has $b\leq' f(b_1)$ since $b_1\leq b_1$. If $z\in A$ satisfies the implication $\forall x\in A\,(b\leq ' f(x) \Rightarrow z \leq  x)$, we get $z\leq b_1$ since $b\leq' f(b_1)$. So when~$b_1$ satisfies the equivalence it is the maximum of $S_b\eqdef\bigcap_{b\leq' f(x)}\dar x\subseteq A$, \cad the g.l.b. of $\sotq{x\in A}{b\leq 'f(x)}$. \\
Conversely, if such a g.l.b.\ $b_1$ exists, it satisfies the implication  $\forall x\in A\,(b\leq ' f(x) \Rightarrow b_1 \leq  x)$ since $b_1\in S_b$. Moreover, if $b\leq 'f(b_1)$ we have the converse implication  $\forall x\in A\,( b_1 \leq  x\Rightarrow b\leq ' f(x))$
because if $b_1\leq x$ then $b\leq' f(b_1)\leq 'f(x)$.
\end{proof}
\begin{proof}[\textsl{Proof of \thref{th-dico-trdi-spec-mor3}}]~\\ 
\noindent \textsl{3} $\Rightarrow$ \textsl{2}. The \prt \textsl{\ref{i2a}} is the particular case of \textsl{\ref{i2b}} with $a=1$. The \prt \textsl{\ref{i2d}} is the particular case of \textsl{\ref{i2c}} with $b=1$. It remains to see that \textsl{\ref{i2b}} implies \textsl{\ref{i2c}}. Indeed
\[ 
\begin{aligned} 
\wi f(f(a)\vi b)     &= \Vi\nolimits_{c:f(a)\vi b\leq f(c)}c \quad \hbox{(Lemma \ref{lemcroiss})}\\
   &= \Vi\nolimits_{c:a\vi \wi f(b)\leq c}c \qquad \hbox{(item \textsl{\ref{i2b}})}  \\ 
&=a\vi \wi f(b)     
 \end{aligned}
\]
\textsl{1} $\Rightarrow$ \textsl{3}. We assume the map $f\sta:\Spec\gT'\to\Spec\gT$ to be open. If $b\in \gT'$, the \oqc $\fD_{\gT'}(b)=\fB$ has as image a \oqc of $\gT$, written as~\hbox{$f\sta(\fB)=\fD_{\gT}(\wi b)$} for a unique $\wi b\in\gT$. We write $\wi b=\wi f(b)$ and we get a map $\wi f:\gT'\to \gT$. 
\\
It remains to see that item \textsl{\ref{i2b}} is satisfied. 
For $a,c\in\gT$ let us write $\fA=\fD_{\gT}(a)$, $\fC=\fD_{\gT}(c)$ \hbox{and $g=f\sta$}. We have to prove the equivalence  \textsl{\ref{i2b}}, written as
\[ 
\begin{aligned} 
g^{-1}(\fA) \cap \fB \subseteq g^{-1}(\fC) \,\Longleftrightarrow\,  \fA \cap g(\fB) \subseteq \fC  
 \end{aligned}
\]  
For the direct implication, we consider an $x\in \fB$ such that $g(x)\in \fA$. We have to show that $g(x)\in \fC$. But $x\in g^{-1}(\fA) \cap \fB$, so $x\in g^{-1}(\fC)$, \cad $g(x)\in \fC$.\\
For the converse implication, we transform the r.h.s. by $g^{-1}$. This operation respects inclusion and intersection.
We get $g^{-1}(\fA) \cap g^{-1}(g(\fB)) \subseteq g^{-1}(\fC)$ and we conclude by noticing that $\fB\subseteq g^{-1}(g(\fB))$.

\smallskip \noindent  \textsl{2} $\Rightarrow$ \textsl{1}.
We show that $f\sta(\fD_{\gT'}(b))=\fD_{\gT}(\wi f(b))$.\\
First we show $f\sta(\fD_{\gT'}(b))\subseteq\fD_{\gT}(\wi f(b))$.
Let $\fp'\in\Spec\gT'$ with $b\notin \fp'$ and let 
$$\fp=f\sta(\fp')=f^{-1}(\fp').$$ 
If one had $\wi f(b)\in\fp$ one would have $f(\wi f(b))\in f(\fp)\subseteq \fp'$ and since $b\leq f(\wi f(b))$, $b\in\fp'$. So we have $\fp\in \fD_{\gT}(\wi f(b))$.\\
For the reverse inclusion, let us consider a $\fp\in\fD_{\gT}(\wi f(b))$. 
As $\wi f$ is non decreasing and respects~$\vu$,  the inverse image $\fq={\wi f}^{-1}(\fp)$
is an ideal. 
\\
We have $b\notin \fq$ because if $b\in \fq$ we have  
\hbox{$\wi f(b)\in\wi f({\wi f}^{-1}(\fp)) \subseteq\fp$}.  
\\
If $y\in \fq$ then $\wi f(y)=z\in\fp$ so $y\leq f(z)$ for a $z\in\fp$ (item~\textsl{\ref{i2a}}). Conversely if  $y\leq f(z)$ for a  $z\in\fp$, then $\wi f(y)\leq \wi f(f(z))\leq z$ (item~\textsl{\ref{i2d}}), so $\wi f(y)\in\fp$.
So we get
$$ 
\fq={\wi f}^{-1}(\fp)=\sotq{y\in\gT'}{\exists z\in \fp\,\,y\leq f(z)}. 
$$ 
So $f^{-1}(\fq)=\sotq{x\in\gT}{\exists z\in \fp\,\,f(x)\leq f(z)}$. But $f(x)\leq f(z)$ is equivalent to $x\vi \wi f(1)\leq z$ (item \textsl{\ref{i2b}} with $b=1$). Moreover $\wi f(1)\notin\fp$ since $\wi f(b)\leq \wi f(1)$ and $\wi f(b)\notin\fp$. So 
$$
f^{-1}(\fq)=\sotQ{x\in\gT}{\exists z\in \fp\,\,x\vi \wi f(1)\leq z}=\sotQ{x\in\gT}{x\vi \wi f(1)\in\fp}=\fp .
$$
Nevertheless it is possible that $\fq$ be not a \idep. In this case let us consider an \id $\fq'$ which is maximal among those satisfying $f^{-1}(\fq')=\fp$ and $\wi f(b)\notin\fq'$. We want to show that $\fq'$ is prime. Assume we have $y_1$ and $y_2\in\gT'\setminus\fq'$ such that $y=y_1\vi y_2\in\fq'$.
By maximality there is an \elt $z_i\in \gT\setminus \fp$ such that $f(z_i)$
is in the  \id generated by $\fq'$ and $y_i$ ($i=1,2$), i.e. $f(z_i)\leq x_i\vu y_i$ with $x_i\in\fq'$.  
Taking  $z=z_1\vi z_2$ (it is in $\gT\setminus \fp$) and $x=x_1\vu x_2$ we get $f(z_i)\leq x\vu y_i$ and $f(z)=f(z_1)\vi f(z_2)\leq x \vu y_i$, so $f(z)\leq x \vu y\in\fq'$, and finally $z\in f^{-1}(\fq')=\fp$: a contradiction.   

\smallskip \noindent   \textsl{4} $\Leftrightarrow$ \textsl{3}. Use Lemma
\ref{lemcroiss} by noticing that \textsl{\ref{i2b}} implies \textsl{\ref{i2a}}.
\end{proof}

\subsection{Dimension \prts}

\begin{theorem}[Dimension of spaces] \label{th-dico-trdi-spec-dim1}  \emph{See \cite{CL2003,Lom02}, \cite[chapter XIII]{CACM}.} In \clama 
\propeq
\begin{enumerate}
\item The spectral space $\Spec(\gT)$ is of Krull dimension $\leq n$ (with the meaning of chains of primes)  
\item  
For each sequence $(x_0,\dots,x_n)$ in $\gT$ there exists a \emph{complementary sequence}  $(y_0,\dots,y_n)$, which means
\begin{equation}\label{eqC2G}
\left.\arraycolsep3pt
\begin{array}{rcl}
1& \vda  &   y_n, x_n\\
 y_n,  x_n & \vda  &  y_{n -1}, x_{n -1}  \\
\vdots~~~~& \vdots  &~~~~  \vdots \\
  y_1, x_1& \vda  &  y_0, x_0  \\
y_0, x_0& \vda  & 0     
\end{array}
\right\}
\end{equation}
\end{enumerate}
\end{theorem}

E.g.\ for the dimension $n\leq 2$, the inequalities in  \pref{eqC2G} correspond to the following diagram in~$\gT$.
$$\SCO{x_0}{x_1}{x_2}{y_0}{y_1}{y_2}$$

A zero-dimensional \trdi is a \agB.

\begin{theorem}[Dimension of morphisms] \label{th-dico-trdi-spec-dim2}  \emph{See \cite{CL2001-2018}, \cite[section XIII-7]{CACM}.} 
Let $\gT\subseteq \gT'$ and $f$ be the inclusion morphism. In \clama \propeq
\begin{enumerate}
\item The morphism $\Spec(f):\Spec(\gT')\to\Spec(\gT)$ has Krull dimension $\leq n$.
\item  For any sequence  $(x_0,\dots,x_n)$ in $\gT'$
there exists an integer $k\geq 0$ and \elts $a_1,\ldots,a_k\in \gT$ such that for each partition  $(H,H')$ of $\{1,\ldots,k\}$, there exist $ y_0,\dots,y_n\in \gT'$  such that
\begin{equation} \label {eqdefDiTrRel}
\begin{array}{rclll}
\Vi_{j\in H'} a_j & \vda  &  y_n,\;x_n    \\
y_n,\;x_n& \vda  &y_{n-1},\;x_{n-1}      \\
\vdots\qquad & \vdots  & \qquad  \vdots    \\
y_1,\;x_1& \vda  &  y_0,\;x_0    \\
y_0,\;x_0& \vda  &  \Vu_{j\in H} a_j    \\
\end{array}
\end{equation}
\end{enumerate} 
\end{theorem}
E.g.\ for the relative dimension $n\leq 2$, the inequalities in  \pref{eqdefDiTrRel} correspond to the following diagram in~$\gT$.
with $u=\Vi_{j\in H'} a_j$ and $i=\Vu_{j\in H} a_j$.
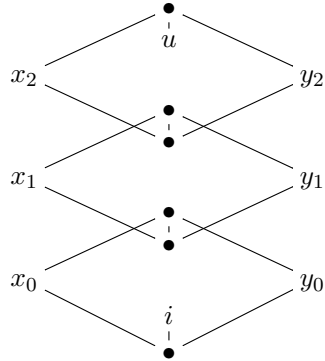
$$\SCOR{x_0}{x_1}{x_2}{y_0}{y_1}{y_2}{u}{i}$$

Note that the dimension of the morphism $\gT\to\gT'$ is less than or equal to the dimension of~$\gT'$: take $k=0$ in item \textsl{2}. 

\smallskip The Krull dimension of a ring $\gA$ and of a morphism $\varphi:\gA\to\gB$ can be defined as those of $\ZarA$ and~$\Zar\varphi$. They are denoted 
$\Kdim \gA$ et  $\Kdim \varphi$. 
In \clama, this coincides with the usual \dfns for $\Kdim \gA$ and  $\Kdim \varphi$ 
(see \cite{CL2001-2018}). 

Moreover the paper \cite{CL2001-2018} 
shows that the \dfns given 
in the book \cite{CACM} coincide in \coma with \(\Kdim(\ZarA)\)
and \(\Kdim(\Zar \varphi)\). In particular we get the  result in Proposition \ref{propDimRelAc}. 

A commutative ring $\gA$ is zero-dimensional when for each $a\in\gA$ there exist $n\in\NN$
and $x\in \gA$ such that $x^n(1-xa)=0$.
A reduced zero-dimensional ring\footnote{Such a ring is also called absolutely flat or von Neumann regular.} is a ring in which any element $a$ has a \textsl{quasi-inverse} $b=a\bl$, i.e.\ such that $aba=a$ and $bab=b$.

\begin{proposition} \label{propDimRelAc}
Let $\gA\bl$ the reduced zero-dimensional ring generated by $\gA$. Then the Krull dimension of a morphism $\rho:\gA\to\gB$ equals the Krull dimension of the ring $\gA\bl\otimes_\gA \gB$. 
\end{proposition}

\subsection{Properties of spaces}

\begin{theorem} \label{th-dico-trdi-spec-esp0} \cite[Chapter II, 4.2, 4.4, 4.9]{Joh1986} \Propeq
\begin{enumerate}
\item The spectral space $\Spec(\gT)$ is Hausdorff. 
\item The topological space $\Spec(\gT)$ is compact, Hausdorff, and totally disconnected. 
\item The \trdi \(\gT\) is a \agB. 
\end{enumerate}
\end{theorem}

 A \trdi $\gT$ is said to be \textsl{normal} if each time one has $a \vu b = 1$ in $\gT$  there exist $x, y$ such that   $a \vu x = b \vu y = 1$ and  $x \vi y = 0$. See \cite{Weh2019,DST2019}.
Note that when replacing $x$ and $y$ with $x_1=x\vu(a\vi b)$ and $y_1=y\vu(a\vi b)$ we get $a \vu x_1 = b \vu y_1 = 1$ and  $x_1 \vi y_1 = a\vi b$.

\begin{theorem} \label{th-dico-trdi-spec-esp1}  \Propeq
\begin{enumerate}
\item The spectral space $\Spec(\gT)$ is normal. 
\item The \trdi \(\gT\) is normal. 
%
%
\end{enumerate}
\end{theorem}
\begin{theorem} \label{th-dico-trdi-spec-esp2}  \Propeq
\begin{enumerate}
\item The spectral space $\Spec(\gT)$ is completely normal. 
\item Each interval $[a,b]$ in $\gT$, seen as a \trdi, is normal.  
\item For all $a,b\in \gT$ there exist $x, y$ such that   $a \vu b = a \vu y = x \vu b$ and  $x \vi y = 0$.
\end{enumerate}
\end{theorem}
\begin{theorem} \label{th-dico-trdi-spec-esp3}  \Propeq
\begin{enumerate}
\item Any \oqc in $\Spec(\gT)$ is a finite union of irreducible \oqcs. 
\item For all \(a_1,\dots,a_n,b_1,\dots,b_m\) 
one has  $a_1,\dots,a_n\vdash_\gT b_1,\dots,b_m$ \ssi there is a~\(j\) such that
 $a_1,\dots,a_n\vdash_\gT b_j$. 
\item 
The \trdi $\gT$ is constructed from a  \sad corresponding to a \talg. 
%
%
\end{enumerate}
\end{theorem}

\section{Some examples}

We give in this section \cov versions of classical theorems.
Often, the theorem has exactly the same wording than the classical one.
But now, these theorems have a clear computational content, which was impossible when using classical definitions.
Sometimes the new theorem is stronger than the previously known classical results (e.g. \thref{thSerre} or \ref{thSwan} or \ref{thBassCancel2}).

\subsection{Relative dimension, lying over, going up, going down}

See \cite{CL2001-2018} and \cite[sections XIII-7 and XIII-9]{CACM}.

\begin{theorem} 
\label{thDim1} 
Let  $\rho:\gA\to\gB$ be a morphism of commutative rings or  \trdis. 
\\
If  $\Kdim\,\gA\leq m$  and $\Kdim\,\rho\leq n$, 
then  $\Kdim\,\gB\leq (m+1)(n+1)-1$.
\end{theorem}

\begin{theorem} \label{thLYGUKdim}
If a morphism $\alpha:\gA\to \gB$ of \trdis or of commutative rings is lying over and going up
(or lying over and going down) one has 
$\Kdim(\gA)\leq \Kdim(\gB)$. 
\end{theorem}

\begin{lemma}\label{cor2thKdimMor}
Let  $\rho:\gA\to\gB$ be a morphism of commutative rings.
If  $\gB$ is generated by  primitively  \agq \elts\footnote{An \elt of $\gB$ is said to be primitively  \agq over $\gA$
if it annihilates a \pol in $\AX$ whose \coes are comaximal.} over $\gA$,
then $\Kdim \rho\leq0$ and so $\Kdim\gB\leq\Kdim\gA$.
\end{lemma}

\begin{lemma} 
\label{lemGU1} Let $\varphi:\gA\to \gB$ be a  morphism of commutative rings. 
 The morphism $\varphi$ is lying over \ssi
for each ideal  $\fa$ of~$\gA$  and each $x\in \gA$, one has
$\, \varphi(x)\in \varphi(\fa)\gB\, \Rightarrow \, x\in\sqrt[\gA]{\fa}.
$  

\end{lemma}

\begin{lemma} 
\label{lemGU2} Let $\varphi:\gA\to \gB$ be a morphism of commutative rings. 
\Propeq
\begin{enumerate}
\item The morphism $\varphi$ is going up (i.e. the morphism $\Zar\,\varphi$ is going up).
\item For any ideal $\fb$ of $\gB$, with $\fa=\varphi^{-1}(\fb)$, the morphism
$\varphi_\fb:\gA/\fa\to \gB/\fb$ is lying over. 
\item The same thing with \tf \ids $\fb$.
\item (In \clama) the same thing with \ideps.
\end{enumerate}
\end{lemma}

\begin{lemma} \label{lemGuAnFp}
Let $\gA\subseteq \gB$ be a faithfully flat $\gA$-\alg. The morphism $\gA\to \gB$ is lying over and going up. So $\Kdim\, \gA\leq \Kdim\, \gB$.
\end{lemma}

\begin{lemma}[A classical Going Up] 
\label{corLy1} 
 Let $\gA\subseteq \gB$ be commutative rings with $\gB$ 
integral over $\gA$. Then the morphism $\gA\to \gB$ is lying over and going up. 
So $\Kdim\, \gA\leq \Kdim\, \gB$.
\end{lemma}

\begin{lemma}\label{lem1Gdown}
\label{lemGD2} Let $\varphi:\gA\to \gB$  be a morphism of commutative rings. 
\Propeq
\begin{enumerate}
\item The morphism $\varphi$ is going down.
\item For $b$, $a_1$, \ldots, $a_q\in \gA$ and $y\in \gB$ such that~$\varphi(b)y\in\sqrt[\gB]{\gen{\varphi(a_1,\dots,a_q)}}$,
there exist 
$x_1$, \dots, $x_p\in \gA$ such that
$$
\gen{bx_1,\dots,bx_p}\subseteq  \sqrt[\gA]{\gen{a_1,\dots,a_q}} \;\hbox{ and }\;y\in\sqrt[\gB]{\gen{\varphi(x_1), \dots, \varphi(x_p)}}.
$$
\item (In \clama) for each \idep $\fp$ of $\gB$ \hbox{with $\fq=\varphi^{-1}(\fp)$} the morphism $\gA_\fq\to \gB_\fp$ is lying over.
\end{enumerate}
\end{lemma}

\begin{theorem}[Going down] 
\label{thGDplat}  
Let $\gA\subseteq \gB$ be commutative rings. The inclusion morphism $\gA\to \gB$ is going down in the following cases. 
\begin{enumerate}
\item $\gB$ is flat over $\gA$.
\item $\gB$ is a domain integral over $\gA$, and $\gA$ is integrally closed.
\end{enumerate}
\end{theorem}

\subsection{Kronecker, Forster-Swan, Serre and Bass Theorems}

References: \cite{CLQ2004,CLQ2006} and \cite[chapter XIV]{CACM}.

\begin{theorem}[Kronecker-Heitmann theorem,
with Krull dimension, without Noetherianity]
\label{thKroH} ~
\begin{enumerate}
\item Let $n\geq 0$.
If $\Kdim \gA <n$ and $b_1$, \dots, $b_n\in\gA$,  there exist
$x_1$, \dots, $x_n$ such that for all $a\in\gA$,
$\DA(a,b_1,\dots,b_n) = \DA(b_1+ax_1,\dots,b_n+ax_n)$.
\item Consequently in a ring wuth Krull dimension $\leq n$, every \itf has the same nilradical as an \id generated by at most $n+1$ \elts.
\end{enumerate}
\end{theorem}

For a commutative ring $\gA$ we define $\Jdim\gA$ (J-dimension of $\gA$) as being $\Kdim(\Heit\gA)$. In \clama it is the dimension of the Heitmann \hbox{J-spectrum} $\Jspec(\gA)$.

Another dimension, called Heitmann dimension and denoted by $\Hdim(\gA)$,
has been introduced in \cite{CLQ2004,CLQ2006}. One has always 
$\Hdim(\gA)\leq \Jdim(\gA)\leq \Kdim(\gA)$. The following results with $\Jdim$ hold also for $\Hdim$.

\begin{definition}\label{defiStableRange}
 A ring $\gA$ is said to have \textsl{stable range} \textsl{(of Bass)} less than or equal to $n$ when  
unimodular vectors of length $n+1$ may be shortened in the following meaning
$$1 \in \gen {a,\an}\;\Longrightarrow\;\exists \,\xn,\;1 \in \gen {a_1 + x_1a, \ldots, a_n + x_na}.$$
\end{definition}

\begin{theorem}[Bass-Heitmann Theorem, without
Noetherianity]
\label{Bass}  
Let $n\geq 0$. If $\Jdim \gA <n$, then $\gA$ has \emph{stable range} $\leq n$. 
In particular  each stably free \Amo of rank $\geq n$ is free.
\end{theorem}

A matrix is said to be of rank $\geq k$ when the minors of size $k$ are comaximal.

\begin{theorem}
\label{thSerre} \emph{(Serre's Splitting Off \tho, for $\Jdim$)}\\
Let $k\geq 1$ and $M$ be a \pro \Amo   of rank $\geq k$, or more generally isomorphic to the image of a matrix of rank $\geq k$.\\
Assume that $\Jdim  \gA< k$.
Then $M\simeq N\oplus \gA$ for a suitable module~$N$ isomorphic to the image of a matrix of rank $\geq k-1$. 
\end{theorem}

\begin{corollary}\label{corthSerre}
Let $\gA$ be a ring such that $\Jdim \gA\le h$  and  $M$ be an \Amo 
isomorphic to the image of a matrix of rank $\geq h+s$.  Then~$M$
has a direct summand which is a free submodule of rang~$s$.  Precisely, if~$M$ is
the image of a matrix $F\in\gA^ {n\times m}$  of rank $\geq h+s$, one has $M= N\oplus L$
where $L$ is a direct summand that is free of rank $s$ in $\gA^n$, and $N$ the image of a matrix of \hbox{rank $\geq h$}.
\end{corollary}

In the following theorem we use the notion of \tf module \textsl{locally generated by~$k$ elements}.
In \clama this means that after \lon at any \idema,~$M$ is generated by $k$ \elts. A classically equivalent \cov \dfn  is that the $k$-th Fitting \id of $M$ is equal to~$\gen{1}$.

\begin{theorem}[Forster-Swan \tho for $\Jdim$]
\label{thSwan} If $\Jdim\gA\leq k$  and if the \Amo $M=\gen{y_1 , \dots,  y_{k+r+s}}$ is  \lot generated by $r$ \elts, then it is 
generated by $k+r$
\elts: one can compute 
$
z_1,\, \ldots, \,z_{k+r}$  
in  $\gen{y_{k+r+1},\ldots,y_{k+r+s}}
$
such that
$M$  is generated by $(y_1+z_1, \ldots, y_{k+r}+z_{k+r})$.
\end{theorem}

Given two modules $M$ and $L$ we say that $M$ \textsl{is cancellative for $L$} if $M\oplus L\simeq N\oplus L$ implies $M \simeq N$.%

\begin{theorem}[Bass' cancellation \tho, with $\Jdim$]
\label{thBassCancel2} 
Let $M$ be a \ptf \Amo of rank $\geq k$.
If $\Jdim \gA< k$,
then $M$ is cancellative for every \ptf \Amo. I.e.\
 if $Q$ is \ptf and $M\oplus Q\simeq N\oplus Q$, then~$M\simeq N$.
\end{theorem}

Theorems \ref{thSerre}, \ref{thSwan} and \ref{thBassCancel2} were conjectured  in \cite{Hei84} (Heitmann proved these theorems for the Krull dimension without Noetherianity assumption).

\subsection{Other results concerning Krull dimension}

In \cite{CACM} Theorem XII-6.2 gives the following important characterisation.
\textsl{An integrally closed \coh \ri $\gA$ of Krull dimension at most $1$ is a Prüfer domain.} This explains in a constructive way the nowadays classical definition of Dedekind domains as Noetherian, integrally closed domains of Krull dimension $\leq 1$, and the fact that, from this \dfn, in \clama, one is able to prove that \tf nonzero ideals are invertible.

In \cite[chapter XVI]{CACM} there is  a \prco of the Lequain-Simis \tho. This proof uses the Krull dimension. 

In \cite[section 2.6]{Yen2015} we find the following new result, with a  \prco. \textsl{If $\gA$ is a ring of Krull dimension $\leq d$, then the stably free modules of rank $>d$ over $\AX$ are free.}


\subsection{Open map}

A \tho by Grothendieck in commutative \alg says that if a ring morphism $\varphi:\gA\to\aqo{\AXn}{\lfs}$ is flat, then the dual morphism  $\Spec \varphi$ is an open map.

It should be interesting to give a constructive proof of this result. The first case should be  $\aqo{\AXn}{f}$: we know that $\varphi$ is flat \ssi the ideal generated by the  \coes of $f$ is generated by an idempotent.

\rdb
\addcontentsline{toc}{section}{References}

\small

\endgroup
\stopcontents[english]

\clearpage
\newpage
\thispagestyle{empty}
~
\clearpage
\newpage

\setcounter{page}{1}\renewcommand\thepage{F\arabic{page}}\renewcommand\theHsection{F\arabic{section}}

\clearpage\setcounter{section}{0}\selectlanguage{french}\def\frenchproofname{\textsl{Démonstration}}


\newcommand \Glio {\MA{\mathsf{Glio}}}

\newcommand{\vou}{\MA{\tsbf{ ou }}}
\newcommand{\Vou}{\MA{\tsbf{OU}}}
\newcommand \EXists[1] {\tsbf{Introduire }{#1}\tsbf{ tel que }\,}
\newcommand \vet {\tsbf{,}\;}
\newcommand \Atcl {\mathrm{Atcl}}



\newcounter{MF}
\newcommand\stMF{\stepcounter{MF}}

\newcommand{\lec}{\stMF\ifodd\value{MF}lecteur \else 
lectrice \fi}
\newcommand{\lecz}{\stMF\ifodd\value{MF}lecteur\else lectrice\fi}

\newcommand{\lecs}{\stMF\ifodd\value{MF}lecteurs \else 
lectrices \fi}
\newcommand{\lecsz}{\stMF\ifodd\value{MF}lecteurs\else 
lectrices\fi}

\newcommand{\alec}{\stMF\ifodd\value{MF}au lecteur \else%
à la lectrice \fi}
\newcommand{\alecz}{\stMF\ifodd\value{MF}au lecteur\else%
à la lectrice\fi}

\newcommand{\dlec}{\stMF\ifodd\value{MF}du lecteur \else%
de la lectrice \fi}
\newcommand{\dlecz}{\stMF\ifodd\value{MF}du lecteur\else%
de la lectrice\fi}

\newcommand{\llec}{\stMF\ifodd\value{MF}le lecteur \else la lectrice \fi}
\newcommand{\llecz}{\stMF\ifodd\value{MF}le lecteur\else la lectrice\fi}

\newcommand{\Llec}{\stMF\ifodd\value{MF}Le lecteur \else La lectrice \fi}

\newcommand{\lui}{\ifodd\value{MF}lui \else
elle \fi}
\newcommand{\luiz}{\ifodd\value{MF}lui\else
elle\fi}

\newcommand{\celui}{\ifodd\value{MF}celui \else
celle \fi}

\newcommand{\ceux}{\ifodd\value{MF}ceux \else
celles \fi}

\newcommand{\er}{\ifodd\value{MF}er \else
ère \fi}

\newcommand{\eux}{\ifodd\value{MF}eux \else
elles \fi}

\newcommand{\eUx}{\ifodd\value{MF}eux \else
euse \fi}

\newcommand{\leux}{\ifodd\value{MF}leux \else
leuse \fi}

\newcommand{\il}{\ifodd\value{MF}il \else
elle \fi}

\newcommand{\ien}{\ifodd\value{MF}ien \else
ienne \fi}

\newcommand{\e}{\ifodd\value{MF} \else e \fi}
\newcommand{\ez}{\ifodd\value{MF}\else e\fi}

\newcommand{\n}{\ifodd\value{MF}n \else nne \fi}
\newcommand{\nz}{\ifodd\value{MF}n\else nne\fi}

\makeatletter
\newcommand{\la}{\@ifstar{\ifodd\value{MF}le\else
la\fi}{\stMF\ifodd\value{MF}le \else la \fi}}
\makeatother

\newcommand \rem{\rdb
\noi{\sl Remarque. }}

\newcommand \REM[1]{\rdb
\noi{\sl Remarque#1. }}

\newcommand \rems{\rdb
\noi{\sl Remarques. }}

\newcommand \exl{\rdb
\noi{\bf Exemple. }}

\newcommand \EXL[1]{\rdb
\noi{\bf Exemple: #1. }}

\newcommand \exls{\rdb
\noi{\bf Exemples. }}

\newcommand \thref[1] {théorème~\ref{#1}}
\newcommand \paref[1] {page~\pageref{#1}}
\newcommand \pstfref[1] {Posi\-tiv\-stel\-lensatz formel~\ref{#1}}
\newcommand \pstref[1] {Posi\-tiv\-stel\-lensatz~\ref{#1}}

\newcommand\oge{\leavevmode\raise.3ex\hbox{$\scriptscriptstyle\langle\!\langle\,$}}
\newcommand\feg{\leavevmode\raise.3ex\hbox{$\scriptscriptstyle\,\rangle\!\rangle$}}

\newcommand\gui[1]{\oge{#1}\feg}

\newcommand \facile{\begin{proof}
La démonstration est laissée \alecz.
\end{proof}
}

\newcommand \num {{n$^{\mathrm{ o}}$}}

\newcommand\comm{\rdb
\noi{\it Commentaire. }}

\newcommand\COM[1]{\rdb
\noi{\it Commentaire #1. }}

\newcommand\comms{\rdb
\noi{\it Commentaires. }}

\newcommand\Pb{\rdb
\noi{\bf Problème. }}

\newcommand\eoq{\hbox{}\nobreak
\vrule width 1.4mm height 1.4mm depth 0mm}

\newcommand \Cad {C'est-à-dire\xspace}
\newcommand \recu {récur\-rence\xspace}
\newcommand \hdr {hypo\-thèse de \recu}
\newcommand \cad {c'est-à-dire\xspace}
\newcommand \cade {c'est-à-dire en\-co\-re\xspace}
\newcommand \ssi {si, et seu\-lement si, }
\newcommand \ssiz {si, et seu\-lement si,~}
\newcommand \cnes {con\-di\-tion néces\-sai\-re et suf\-fi\-san\-te\xspace}
\newcommand \spdg {sans per\-te de géné\-ra\-lité\xspace}
\newcommand \Spdg {Sans per\-te de géné\-ra\-lité\xspace}

\newcommand \Propeq {Les pro\-pri\-é\-tés sui\-van\-tes sont 
équi\-va\-len\-tes.}
\newcommand \propeq {les pro\-pri\-é\-tés sui\-van\-tes sont 
équi\-va\-len\-tes.}

\newcommand \Kev {$\gK$-\evc}
\newcommand \Kevs {$\gK$-\evcs}

\newcommand \Lev {$\gL$-\evc}
\newcommand \Levs {$\gL$-\evcs}

\newcommand \Qev {$\QQ$-\evc}
\newcommand \Qevs {$\QQ$-\evcs}

\newcommand \kev {$\gk$-\evc}
\newcommand \kevs {$\gk$-\evcs}

\newcommand \lev {$\gl$-\evc}
\newcommand \levs {$\gl$-\evcs}

\newcommand \Alg {$\gA$-\alg}
\newcommand \Algs {$\gA$-\algs}

\newcommand \Blg {$\gB$-\alg}
\newcommand \Blgs {$\gB$-\algs}

\newcommand \Clg {$\gC$-\alg}
\newcommand \Clgs {$\gC$-\algs}

\newcommand \klg {$\gk$-\alg}
\newcommand \klgs {$\gk$-\algs}

\newcommand \llg {$\gl$-\alg}
\newcommand \llgs {$\gl$-\algs}

\newcommand \Klg {$\gK$-\alg}
\newcommand \Klgs {$\gK$-\algs}

\newcommand \Llg {$\gL$-\alg}
\newcommand \Llgs {$\gL$-\algs}

\newcommand \QQlg {$\QQ$-\alg}
\newcommand \QQlgs {$\QQ$-\algs}

\newcommand \Rlg {$\gR$-\alg}
\newcommand \Rlgs {$\gR$-\algs}

\newcommand \RRlg {$\RR$-\alg}
\newcommand \RRlgs {$\RR$-\algs}

\newcommand \ZZlg {$\ZZ$-\alg}
\newcommand \ZZlgs {$\ZZ$-\algs}

\newcommand \Amo {$\gA$-mo\-du\-le\xspace}
\newcommand \Amos {$\gA$-mo\-du\-les\xspace}

\newcommand \Bmo {$\gB$-mo\-du\-le\xspace}
\newcommand \Bmos {$\gB$-mo\-du\-les\xspace}

\newcommand \Cmo {$\gC$-mo\-du\-le\xspace}
\newcommand \Cmos {$\gC$-mo\-du\-les\xspace}

\newcommand \kmo {$\gk$-mo\-du\-le\xspace}
\newcommand \kmos {$\gk$-mo\-du\-les\xspace}

\newcommand \Kmo {$\gK$-mo\-du\-le\xspace}
\newcommand \Kmos {$\gK$-mo\-du\-les\xspace}

\newcommand \Lmo {$\gL$-mo\-du\-le\xspace}
\newcommand \Lmos {$\gL$-mo\-du\-les\xspace}

\newcommand \Ali {appli\-ca\-tion $\gA$-\lin}
\newcommand \Alis {appli\-ca\-tions $\gA$-\lins}

\newcommand \Kli {appli\-ca\-tion $\gK$-\lin}
\newcommand \Klis {appli\-ca\-tions $\gK$-\lins}

\newcommand \Bli {appli\-ca\-tion $\gB$-\lin}
\newcommand \Blis {appli\-ca\-tions $\gB$-\lins}

\newcommand \Cli {appli\-ca\-tion $\gC$-\lin}
\newcommand \Clis {appli\-ca\-tions $\gC$-\lins}

\newcommand \ac{algé\-bri\-quement clos\xspace}  

\newcommand \acl {an\-neau \icl}
\newcommand \acls {an\-neaux \icl}

\newcommand \adp {an\-neau de Pr\"u\-fer\xspace}
\newcommand \adps {an\-neaux de Pr\"u\-fer\xspace}

\newcommand \adpc {\adp \coh}
\newcommand \adpcs {\adps \cohs}

\newcommand \adu {\alg de décom\-po\-sition univer\-selle\xspace}
\newcommand \adus {\algs de décom\-po\-sition univer\-selle\xspace}

\newcommand \adv {anneau de valuation\xspace}
\newcommand \advs {anneaux de valuation\xspace}

\newcommand \advl {anneau \dvla} 
\newcommand \advls {anneaux \dvlas} 

\newcommand \Afr {Anneau \frl}
\newcommand \Afrs {Anneaux \frls}
\newcommand \afr {anneau \frl}
\newcommand \aFr {\hyperref[theorieAfr]{anneau \frl}\xspace}
\newcommand \afrs {anneaux \frls}

\newcommand \afrr {\afr réduit\xspace}
\newcommand \afrrs {\afrs réduits\xspace}
\newcommand \Afrrs {\Afrs réduits\xspace}

\newcommand \afrvr {\afr avec \ravs}
\newcommand \aFrvr {\hyperref[theorieAfrrv]{\afrvr}\xspace}
\newcommand \afrvrs {\afrs avec \ravs}

\newcommand \aftr {anneau réticulé \ftm réel\xspace}
\newcommand \aftrs {anneaux réticulés \ftm réels\xspace}

\newcommand \aG {\alg galoisienne\xspace}
\newcommand \aGs {\algs galoisiennes\xspace}

\newcommand \agB {\alg de Boole\xspace}
\newcommand \agBs {\algs de Boole\xspace}

\newcommand \agH {\alg de Heyting\xspace}
\newcommand \agHs {\algs de Heyting\xspace}

\newcommand \agq{algé\-bri\-que\xspace}
\newcommand \agqs{algé\-bri\-ques\xspace}

\newcommand \agqt{algé\-bri\-que\-ment\xspace}

\newcommand \aKr {anneau de Krull\xspace}
\newcommand \aKrs {anneaux de Krull\xspace}

\newcommand \alg {algè\-bre\xspace}
\newcommand \algs {algè\-bres\xspace}

\newcommand \algo{algo\-rithme\xspace}
\newcommand \algos{algo\-rithmes\xspace}

\newcommand \algq{al\-go\-rith\-mi\-que\xspace}
\newcommand \algqs{al\-go\-rith\-mi\-ques\xspace}

\newcommand \ali {appli\-ca\-tion \lin}
\newcommand \alis {appli\-ca\-tions \lins}

\newcommand \alo {an\-neau lo\-cal\xspace}
\newcommand \alos {an\-neaux lo\-caux\xspace}

\newcommand \algb {an\-neau \lgb}
\newcommand \algbs {an\-neaux \lgbs}

\newcommand \alrd {\alo \dcd}
\newcommand \alrds {\alos \dcds}

\newcommand \anar {anneau \ari}
\newcommand \anars {anneaux \aris}

\newcommand \anor {an\-neau nor\-mal\xspace}
\newcommand \anors {an\-neaux nor\-maux\xspace}

\newcommand \apf {\alg \pf}
\newcommand \apfs {\algs \pf}

\newcommand \apG {\alg pré\-galoisienne\xspace}
\newcommand \apGs {\algs pré\-galoisiennes\xspace}

\newcommand \arc {anneau réel clos\xspace}
\newcommand \aRc {\hyperref[theorieArc]{\arc}\xspace}
\newcommand \arcs {anneaux réels clos\xspace}

\newcommand \ari{arith\-mé\-tique\xspace}  
\newcommand \aris{arith\-mé\-tiques\xspace}  

\newcommand \Asr {Anneau \str}
\newcommand \Asrs {Anneaux \strs}
\newcommand \asr {anneau \str}
\newcommand \asrs {anneaux \strs}

\newcommand \asrvr {\asr avec \ravs}
\newcommand \asrvrs {\asrs avec \ravs}

\newcommand \atf {\alg \tf}
\newcommand \atfs {\algs \tf}

\newcommand \auto {auto\-mor\-phisme\xspace}
\newcommand \autos {auto\-mor\-phismes\xspace}


\newcommand \bdg {base de Gr\"obner\xspace}
\newcommand \bdgs {bases de Gr\"obner\xspace}

\newcommand \bdp {base de \dcn partielle\xspace}
\newcommand \bdps {bases de \dcn partielle\xspace}

\newcommand \bdf {base de \fap\xspace}

\newcommand \Bif {Borne infé\-rieure\xspace} %
\newcommand \bif {borne infé\-rieure\xspace} %
\newcommand \bifs {bornes infé\-rieures\xspace} %

\newcommand \bsp {borne supé\-rieure\xspace} %
\newcommand \bsps {borne supé\-rieures\xspace} %


\newcommand \cac{corps \ac}  

\newcommand \calf{calcul formel\xspace}  

\newcommand \cara{carac\-té\-ris\-tique\xspace}  
\newcommand \caras{carac\-té\-ris\-tiques\xspace}  

\newcommand \carn{carac\-té\-ri\-sation\xspace}  
\newcommand \carns{carac\-té\-ri\-sations\xspace}  

\newcommand \carar{carac\-té\-riser\xspace}

\newcommand \carf{de carac\-tère fini\xspace}  

\newcommand \cdi{corps discret\xspace}
\newcommand \cdis{corps discrets\xspace}
  
\newcommand \cdv{changement de variables\xspace}  
\newcommand \cdvs{changements de variables\xspace}  

\newcommand \cli {clô\-ture inté\-grale\xspace}

\newcommand \codi {corps ordonné discret\xspace}
\newcommand \codis {corps ordonnés discrets\xspace}

\newcommand \coe {coef\-fi\-cient\xspace}
\newcommand \coes {coef\-fi\-cients\xspace}

\newcommand \coh {co\-hé\-rent\xspace}
\newcommand \cohs {co\-hé\-rents\xspace}

\newcommand \cohc {co\-hé\-rence\xspace}

\newcommand \coli {combi\-nai\-son \lin}
\newcommand \colis {combi\-nai\-sons \lins}

\newcommand \com {co\-ma\-xi\-maux\xspace}
\newcommand \come {co\-ma\-xi\-males\xspace}

\newcommand \coo {coor\-donnée\xspace}
\newcommand \coos {coor\-données\xspace}

\newcommand \cop {complé\-men\-taire\xspace}
\newcommand \cops {complé\-men\-taires\xspace}

\newcommand \cosv {conser\-vative\xspace}
\newcommand \cosvs {conser\-vatives\xspace}

\newcommand \cOsv {\hyperref[defithconserv]{conser\-vative}\xspace}
\newcommand \cOsvs {\hyperref[defithconserv]{conser\-vatives}\xspace}

\newcommand \covr {corps ordonné avec \ravs}
\newcommand \covrs {corps ordonnés avec \ravs}

\newcommand \cpb {compa\-tible\xspace} 
\newcommand \cpbs {compa\-tibles\xspace} 

\newcommand \cpbt {compa\-tibi\-lité\xspace} 
\newcommand \cpbtz {compa\-tibi\-lité} 

\newcommand \crc {corps réel clos\xspace}
\newcommand \crcs {corps réels clos\xspace}

\newcommand \crcd {corps réel clos discret\xspace}
\newcommand \crcds {corps réels clos discrets\xspace}


\newcommand \dcd {rési\-duel\-lement dis\-cret\xspace}
\newcommand \dcds {rési\-duel\-lement dis\-crets\xspace}

\newcommand \dcn {décom\-po\-sition\xspace}
\newcommand \dcns {décom\-po\-sitions\xspace}

\newcommand \dcnb {\dcn bornée\xspace}

\newcommand \dcnc {\dcn complète\xspace}

\newcommand \dcnp {\dcn partielle\xspace}

\newcommand \dcp {décom\-posa\-ble\xspace}
\newcommand \dcps {décom\-posa\-bles\xspace}

\newcommand \ddk {dimension de~Krull\xspace}
\newcommand \ddi {de dimension infé\-rieure ou égale à~}

\newcommand \ddp {domaine de Pr\"u\-fer\xspace}
\newcommand \ddps {domaines de Pr\"u\-fer\xspace}

\newcommand \Demo{Démon\-stra\-tion\xspace}     

\newcommand \demo{démon\-stra\-tion\xspace}     
\newcommand \demos{démon\-stra\-tions\xspace}     

\newcommand \dems{démons\-tra\-tions\xspace}

\newcommand \deno{déno\-mi\-nateur\xspace}     
\newcommand \denos{déno\-mi\-nateurs\xspace}   

\newcommand \deter {déter\-mi\-nant\xspace}  
\newcommand \deters {déter\-mi\-nants\xspace}  
  
\newcommand \Dfn{Défi\-nition\xspace}  
\newcommand \Dfns{Défi\-nitions\xspace}  
\newcommand \dfn{défi\-nition\xspace}  
\newcommand \dfns{défi\-nitions\xspace}  

\newcommand \dftr {droite réticulée \ftm réelle\xspace}
\newcommand \dftrs {droites réticulées \ftm réelles\xspace}
  
\newcommand \dil{diffé\-rentiel\xspace}  
\newcommand \dils{diffé\-rentiels\xspace}  
\newcommand \dile{diffé\-ren\-tielle\xspace}  
\newcommand \diles{diffé\-ren\-tielles\xspace}  

\newcommand \dip{diviseur principal\xspace}
\newcommand \dips{diviseurs principaux\xspace}

\newcommand \discri{discri\-minant\xspace}  
\newcommand \discris{discri\-minants\xspace}  

\newcommand \divle {dimension divisorielle\xspace}

\newcommand \dit{distri\-bu\-ti\-vité\xspace}

\newcommand \dlg{d'élar\-gis\-sement\xspace}  

\newcommand \dok {domaine de Dedekind\xspace}
\newcommand \doks {domaines de Dedekind\xspace}

\newcommand \dvla {à diviseurs\xspace}
\newcommand \dvlas {à diviseurs\xspace}

\newcommand \dvld {\dvlt décom\-posé\xspace} %
\newcommand \dvlds {\dvlt décom\-posés\xspace} %

\newcommand \dvlg {diviso\-riel\xspace} 
\newcommand \dvlgs {diviso\-riels\xspace} 

\newcommand \dvli {\dvlt inver\-sible\xspace} 
\newcommand \dvlis {\dvlt inver\-sibles\xspace} 

\newcommand \dvlt {diviso\-riel\-lement\xspace} %

\newcommand \dvz {di\-viseur de zéro\xspace}
\newcommand \dvzs {di\-viseurs de zéro\xspace}

\newcommand \dve {divi\-si\-bi\-lité\xspace}

\newcommand \dvee {à \dve explicite\xspace}

\newcommand \dvr {diviseur\xspace}
\newcommand \dvrs {diviseurs\xspace}


\newcommand \Eds {Exten\-sion des sca\-laires\xspace}
\newcommand \edss {exten\-sions des sca\-laires\xspace}
\newcommand \eds {exten\-sion des sca\-laires\xspace}

\newcommand \eco {\elts \com}

\newcommand \egmt {éga\-lement\xspace}

\newcommand \egt {éga\-li\-té\xspace}
\newcommand \egts {éga\-li\-tés\xspace}

\newcommand \eli{élimi\-nation\xspace}  

\newcommand \elr{élé\-men\-taire\xspace}  
\newcommand \elrs{élé\-men\-taires\xspace}  

\newcommand \elrt{élé\-men\-tai\-rement\xspace}  

\newcommand \elt{élé\-ment\xspace}  
\newcommand \elts{élé\-ments\xspace}  

\def \endo {en\-do\-mor\-phisme\xspace}
\def \endos {en\-do\-mor\-phismes\xspace}

\newcommand \entrel {rela\-tion impli\-ca\-tive\xspace}
\newcommand \entrels {rela\-tions impli\-ca\-tives\xspace}

\newcommand\evc{es\-pa\-ce vec\-to\-riel\xspace} 
\newcommand\evcs{es\-pa\-ces vec\-to\-riels\xspace} 

\newcommand \eqv {équi\-valent\xspace}  
\newcommand \eqve {équi\-va\-lente\xspace}  
\newcommand \eqvs {équi\-valents\xspace}  
\newcommand \eqves {équi\-val\-entes\xspace}  

\newcommand \eqvc {équi\-va\-lence\xspace}  
\newcommand \eqvcs {équi\-va\-lences\xspace}  

\newcommand \esid {essen\-tiel\-lement iden\-tique\xspace}  
\newcommand \esids {essen\-tiel\-lement iden\-tiques\xspace}  

\newcommand \Esid {\hyperref[defitdyesidentiques]{\esid}\xspace}  
\newcommand \Esids {\hyperref[defitdyesidentiques]{\esids}\xspace}  

\newcommand \eseq {essen\-tiel\-lement \eqve}  
\newcommand \eseqs {essen\-tiel\-lement \eqves}  

\newcommand \Eseq {\hyperref[defitheseq]{\eseq}\xspace}  
\newcommand \Eseqs {\hyperref[defitheseq]{\eseqs}\xspace}

\newcommand \fab {\fcn bornée\xspace}
\newcommand \fabs {\fcns bornées\xspace}

\newcommand \fat {\fcn totale\xspace}
\newcommand \fats {\fcn totales\xspace}

\newcommand \fap {\fcn partielle\xspace}
\newcommand \faps {\fcns partielles\xspace}

\newcommand \fip {filtre pre\-mier\xspace}
\newcommand \fips {filtres pre\-miers\xspace}

\newcommand \fipma {\fip maxi\-mal\xspace}
\newcommand \fipmas {\fips maxi\-maux\xspace}

\newcommand \fcn {factorisation\xspace}
\newcommand \fcns {factorisations\xspace}

\newcommand \fdi {for\-te\-ment dis\-cret\xspace}
\newcommand \fdis {for\-te\-ment dis\-crets\xspace}

\newcommand \fsa {fermé \sagq}
\newcommand \fsas {fermés \sagqs}

\newcommand \fsagc {fonction \sagc}
\newcommand \fsagcs {fonctions \sagcs}

\newcommand \fmt {formel\-lement\xspace}

\newcommand \frl {for\-tement réticulé\xspace}
\newcommand \frle {for\-tement réticulée\xspace}
\newcommand \frls {for\-tement réticulés\xspace}

\newcommand \ftm {fortement\xspace}

\newcommand\gmt{géométrie\xspace}  
\newcommand\gmts{géométries\xspace}  

\newcommand\gaq{\gmt \agq}  

\newcommand\gmq{géomé\-trique\xspace}  
\newcommand\gmqs{géomé\-triques\xspace}  

\newcommand\gmqt{géomé\-tri\-quement\xspace}  

\newcommand\gne{géné\-ra\-lisé\xspace}  
\newcommand\gnee{géné\-ra\-lisée\xspace}  
\newcommand\gnes{géné\-ra\-lisés\xspace}  
\newcommand\gnees{géné\-ra\-lisées\xspace}  

\newcommand\gnl{géné\-ral\xspace}  
\newcommand\gnle{géné\-rale\xspace}  
\newcommand\gnls{géné\-raux\xspace}  
\newcommand\gnles{géné\-rales\xspace}  

\newcommand\gnlt{géné\-ra\-lement\xspace}  

\newcommand\gnn{géné\-ra\-li\-sa\-tion\xspace}  
\newcommand\gnns{géné\-ra\-li\-sa\-tions\xspace}  

\newcommand\gnq {géné\-rique\xspace}  
\newcommand\gnqs {géné\-riques\xspace}  

\newcommand\gnr{géné\-ra\-liser\xspace}  

\newcommand \gns{géné\-ra\-lise\xspace}

\newcommand \gnt{géné\-ra\-lité\xspace}
\newcommand \gnts{géné\-ra\-lités\xspace}

\newcommand \grl{groupe \rtl}
\newcommand \grls{groupes \rtls}

\newcommand \gRl {\hyperref[theorieGrl]{\grl}\xspace}
\newcommand \gRls {\hyperref[theorieGrl]{\grls}\xspace}

\newcommand\gtr{géné\-ra\-teur\xspace}  
\newcommand\gtrs{géné\-ra\-teurs\xspace}  


\newcommand \homo {ho\-mo\-mor\-phisme\xspace}
\newcommand \homos {ho\-mo\-mor\-phismes\xspace}

\newcommand \hmg {homo\-gène\xspace}
\newcommand \hmgs {homo\-gènes\xspace}

\newcommand \icftr {intervalle compact réticulé \ftm réel\xspace}
\newcommand \icftrs {intervalles compacts réticulés \ftm réels\xspace}

\newcommand \icl {inté\-gra\-lement clos\xspace}
\newcommand \icle {inté\-gra\-lement close\xspace}

\newcommand \icsr {intervalle compact \stm réticulé\xspace}
\newcommand \icsrs {intervalles compacts \stm réticulés\xspace}

\newcommand \icrc {intervalle compact réel clos\xspace}
\newcommand \icrcs {intervalles compact réels clos\xspace}

\newcommand \id {idéal\xspace}
\newcommand \ids {idéaux\xspace}

\newcommand \ida {\idt \agq}
\newcommand \idas {\idts \agqs}

\newcommand \idc  {\idt de Cramer\xspace}
\newcommand \idcs {\idts de Cramer\xspace}

\newcommand \idd {idéal déter\-minan\-tiel\xspace}
\newcommand \idds {idéaux déter\-minan\-tiels\xspace}

\newcommand \idema {idéal maxi\-mal\xspace}
\newcommand \idemas {idéaux maxi\-maux\xspace}

\newcommand \idep {idéal pre\-mier\xspace}
\newcommand \ideps {idéaux pre\-miers\xspace}

\newcommand \idemi {\idep minimal\xspace}
\newcommand \idemis {\ideps minimaux\xspace}

\newcommand \idf {idéal de Fitting\xspace}
\newcommand \idfs {idéaux de Fitting\xspace}

\newcommand \idif {idéal \dvlg fini\xspace}
\newcommand \idifs {idéaux \dvlgs finis\xspace}

\newcommand \idli {idéal \dvli\xspace} 
\newcommand \idlis {idéaux \dvlis\xspace} 

\newcommand \idm {idem\-potent\xspace}
\newcommand \idms {idem\-potents\xspace}
\newcommand \idme {idem\-potente\xspace}
\newcommand \idmes {idem\-potentes\xspace}

\newcommand \idp {idéal prin\-ci\-pal\xspace}
\newcommand \idps {idé\-aux prin\-ci\-paux\xspace}

\newcommand \idt {iden\-ti\-té\xspace}
\newcommand \idts {iden\-ti\-tés\xspace}

\newcommand \idtr {indé\-ter\-minée\xspace}
\newcommand \idtrs {indé\-ter\-minées\xspace}

\newcommand \ifr {idéal frac\-tion\-naire\xspace}
\newcommand \ifrs {idéaux frac\-tion\-naires\xspace}

\newcommand \imd {immé\-diat\xspace}
\newcommand \imde {immé\-diate\xspace}
\newcommand \imds {immé\-diats\xspace}
\newcommand \imdes {immé\-diates\xspace}

\newcommand \imdt {immé\-dia\-te\-ment\xspace}

\newcommand \indtr {inf-demi-treillis\xspace} 

\newcommand \inteq {intui\-ti\-vement \eqve}
\newcommand \inteqs {intui\-ti\-vement \eqves}

\newcommand \Inteq {\hyperref[defextintequiv]{\inteq}\xspace}
\newcommand \Inteqs {\hyperref[defextintequiv]{\inteqs}\xspace}

\newcommand \ing {in\-ver\-se \gne}
\newcommand \ings {in\-ver\-ses \gnes}

\newcommand \iMP {in\-ver\-se de Moo\-re-Pen\-ro\-se\xspace}
\newcommand \iMPs {in\-ver\-ses de Moo\-re-Pen\-ro\-se\xspace}

\newcommand \ipp {\idep poten\-tiel\xspace}
\newcommand \ipps {\ideps poten\-tiels\xspace}

\newcommand \ird {irré\-duc\-tible\xspace}
\newcommand \irds {irré\-duc\-tibles\xspace}

\newcommand \iso {iso\-mor\-phisme\xspace}
\newcommand \isos {iso\-mor\-phismes\xspace}

\newcommand \itf {idéal \tf}
\newcommand \itfs {idé\-aux \tf}

\newcommand \itid {intui\-ti\-vement iden\-tique\xspace}
\newcommand \itids {intui\-ti\-vement iden\-tiques\xspace}

\newcommand \iv {inversible\xspace}
\newcommand \ivs {inversibles\xspace}

\newcommand \ivdg {inverse divisoriel\xspace} 
\newcommand \ivdgs {inverses divisoriels\xspace} 

\newcommand \ivde {inverse divisorielle\xspace} 
\newcommand \ivdes {inverses divisorielles\xspace} 

\newcommand \ivda {inverse divisoriel\xspace} 
\newcommand \ivdas {inverses divisoriels\xspace} 


\newcommand \lgb {local-global\xspace}
\newcommand \lgbe {locale-globale\xspace}
\newcommand \lgbs {local-globals\xspace}

\newcommand \lin {liné\-aire\xspace}
\newcommand \lins {liné\-aires\xspace}

\newcommand \lint {liné\-ai\-rement\xspace}

\newcommand \lmo {\lot mono\-gène\xspace}
\newcommand \lmos {\lot mono\-gènes\xspace}

\newcommand \lnl {\lot \nl}
\newcommand \lnls {\lot \nls}

\newcommand \lot {loca\-lement\xspace}

\newcommand \lon {loca\-li\-sation\xspace}
\newcommand \lons {loca\-li\-sations\xspace}

\newcommand \lop {\lot prin\-cipal\xspace}
\newcommand \lops {\lot prin\-cipaux\xspace}

\newcommand \lsdz {\lot \sdz}

\newcommand \mdi {mo\-dule des \diles}

\newcommand \mlm {mo\-dule \lmo}
\newcommand \mlms {mo\-dules \lmos}

\newcommand \mlmo {ma\-tri\-ce de loca\-li\-sation
mono\-gène\xspace}
\newcommand \mlmos {ma\-tri\-ces de loca\-li\-sation
mono\-gène\xspace}

\newcommand \mlp {ma\-tri\-ce de loca\-li\-sation
prin\-ci\-pa\-le\xspace}
\newcommand \mlps {ma\-tri\-ces de loca\-li\-sation
prin\-ci\-pa\-le\xspace}

\newcommand \mo {mo\-no\"{\i}de\xspace}
\newcommand \mos {mo\-no\"{\i}des\xspace}

\newcommand \moco {\mos \com}

\newcommand \molo {morphisme de \lon\xspace}
\newcommand \molos {morphismes de \lon\xspace}

\newcommand \mom {mo\-nô\-me\xspace}
\newcommand \moms {mo\-nô\-mes\xspace}

\newcommand \moquo {morphisme de passage au quotient\xspace}
\newcommand \moquos {morphismes de passage au quotient\xspace}

\newcommand \mpf {mo\-dule \pf}
\newcommand \mpfs {mo\-dules \pf}

\newcommand \mpl {mo\-dule plat\xspace}
\newcommand \mpls {mo\-dules plats\xspace}

\newcommand \mpn {ma\-trice de \pn}
\newcommand \mpns {ma\-trices de \pn}

\newcommand \mprn {ma\-trice de \prn}
\newcommand \mprns {ma\-trices de \prn}

\newcommand \mptf {mo\-dule \ptf}
\newcommand \mptfs {mo\-dules \ptfs}

\newcommand \mrc {mo\-dule \prc}
\newcommand \mrcs {mo\-dules \prcs}

\newcommand \mtf {mo\-du\-le \tf}
\newcommand \mtfs {mo\-du\-les \tf}


\newcommand \ncr{néces\-saire\xspace}       
\newcommand \ncrs{néces\-saires\xspace}       

\newcommand \ncrt{néces\-sai\-rement\xspace}       

\newcommand \ndz {régu\-lier\xspace}
\newcommand \ndzs {régu\-liers\xspace}

\newcommand \nl {simple\xspace}
\newcommand \nls {simples\xspace}

\newcommand \noco {\noe \coh}
\newcommand \nocos {\noes \cohs}

\newcommand \Noe {Noether\xspace}

\newcommand \noe {noethé\-rien\xspace}
\newcommand \noes {noethé\-riens\xspace}
\newcommand \noee {noethé\-rienne\xspace}
\newcommand \noees {noethé\-riennes\xspace}

\newcommand \noet {noethé\-ria\-nité\xspace}

\newcommand \nst {Null\-stellen\-satz\xspace}
\newcommand \nsts {Null\-stellen\-s\"atze\xspace}

\newcommand \op{opé\-ra\-tion\xspace}  
\newcommand \ops{opé\-ra\-tions\xspace}
\newcommand \opari{\op\ari}  
\newcommand \oparis{\ops\aris}  
\newcommand \oparisv{\ops\arisv}  

\newcommand \oqc {ouvert \qc}
\newcommand \oqcs {ouverts \qcs}

\newcommand \ort{or\-tho\-go\-nal\xspace}  
\newcommand \orte{or\-tho\-go\-na\-le\xspace}  
\newcommand \orts{or\-tho\-go\-naux\xspace}  
\newcommand \ortes{or\-tho\-go\-na\-les\xspace}  


\newcommand \pa {couple saturé\xspace}
\newcommand \pas {couples saturés\xspace}
 
\newcommand \paral{paral\-lèle\xspace}  
\newcommand \parals{paal\-lèles\xspace}  

\newcommand \paralm{paral\-lè\-lement\xspace}

\newcommand \pb{pro\-blè\-me\xspace}  
\newcommand \pbs{pro\-blè\-mes\xspace}  

\newcommand \peq {purement équa\-tion\-nelle\xspace}
\newcommand \peqs {purement équa\-tion\-nelles\xspace}

\newcommand \pf {de \pn finie\xspace}

\newcommand \plc {rési\-duel\-lement \zed}
\newcommand \plcs {rési\-duel\-lement \zeds}

\newcommand \Plg {Prin\-cipe \lgb}
\newcommand \plg {prin\-cipe \lgb}
\newcommand \plgs {prin\-cipes \lgbs}

\newcommand \plga {\plg abs\-trait\xspace}
\newcommand \plgas {\plgs abs\-traits\xspace}

\newcommand \Plgc {\Plg con\-cret\xspace}
\newcommand \plgc {\plg con\-cret\xspace}
\newcommand \plgcs {\plgs con\-crets\xspace}

\newcommand \pn {présen\-ta\-tion\xspace}
\newcommand \pns {présen\-ta\-tions\xspace}

\newcommand \pog {\pol \hmg\xspace}
\newcommand \pogs {\pols \hmgs\xspace}

\newcommand \Pol {Poly\-nôme\xspace}
\newcommand \Pols {Poly\-nômes\xspace}

\newcommand \pol {poly\-nôme\xspace}
\newcommand \pols {poly\-nômes\xspace}

\newcommand \poll{poly\-nomial\xspace}  
\newcommand \polls{poly\-nomiaux\xspace}  
\newcommand \polle{poly\-no\-miale\xspace}  
\newcommand \polles{poly\-no\-miales\xspace}  

\newcommand \pollt{poly\-no\-mia\-lement\xspace}  

\newcommand \polfon {\pol fon\-da\-men\-tal\xspace}
\newcommand \polmu {\pol rang\xspace}
\newcommand \polmus {\pols rang\xspace}
\newcommand \polcar {\pol carac\-té\-ris\-tique\xspace}
\newcommand \polmin {\pol mini\-mal\xspace}

\newcommand \prc {\pro de rang constant\xspace}
\newcommand \prcs {\pros de rang constant\xspace}

\newcommand \prcc {prin\-ci\-pe de \rcc}
\newcommand \prca {prin\-ci\-pe de \rca}
\newcommand \prce {prin\-ci\-pe de \rce}

\newcommand \prmt {préci\-sé\-ment\xspace}
\newcommand \Prmt {Préci\-sé\-ment\xspace}

\newcommand \prn {pro\-jec\-tion\xspace}
\newcommand \prns {pro\-jec\-tions\xspace}

\newcommand \pro {pro\-jec\-tif\xspace}
\newcommand \pros {pro\-jec\-tifs\xspace}

\newcommand \prr {pro\-jec\-teur\xspace}
\newcommand \prrs {pro\-jec\-teurs\xspace}

\newcommand \Prt {Pro\-pri\-été\xspace}
\newcommand \Prts {Pro\-pri\-étés\xspace}
\newcommand \prt {pro\-pri\-été\xspace}
\newcommand \prts {pro\-pri\-étés\xspace}

\newcommand \ptf {\pro \tf}
\newcommand \ptfs {\pros \tf}

\newcommand \qc {quasi-compact\xspace}
\newcommand \qcs {quasi-compacts\xspace}

\newcommand \qi {qua\-si in\-tè\-gre\xspace}
\newcommand \qis {qua\-si in\-tè\-gres\xspace}

\newcommand \qnl {quasi-\nl}
\newcommand \qnls {quasi-\nls}

\newcommand \ralg {règle \agq}
\newcommand \ralgs {règles \agqs}

\newcommand \rav {racine virtuelle\xspace}
\newcommand \ravs {racines virtuelles\xspace}

\newcommand \rcc {\rcm con\-cret\xspace}
\newcommand \rca {\rcm abs\-trait\xspace}
\newcommand \rce {\rcc des é\-ga\-li\-tés\xspace}

\newcommand \rcm {recol\-lement\xspace}
\newcommand \rcms {recol\-lements\xspace}

\newcommand \rcv {recou\-vrement\xspace} 
\newcommand \rcvs {recou\-vrements\xspace}

\newcommand \rde {rela\-tion de dépen\-dance\xspace}
\newcommand \rdes {rela\-tions de dépen\-dance\xspace}

\newcommand \rdi {\rde inté\-grale\xspace}
\newcommand \rdis {\rdes inté\-grales\xspace}

\newcommand \rdl {\rde \lin}
\newcommand \rdls {\rdes \lins}

\newcommand \rdt {rési\-duel\-lement\xspace}

\newcommand \rdy {règle dyna\-mique\xspace}
\newcommand \rdys {règles dyna\-miques\xspace}

\newcommand \red {règle directe\xspace}
\newcommand \reds {règles directes\xspace}

\newcommand \rex {règle exis\-ten\-tielle simple\xspace}
\newcommand \rexs {règles exis\-ten\-tielles simples\xspace}

\newcommand \reX {\hyperref[defexistsimple]{règle exis\-ten\-tielle simple}\xspace}
\newcommand \reXs {\hyperref[defexistsimple]{règles exis\-ten\-tielles simples}\xspace}

\newcommand \rexri {règle exis\-ten\-tielle rigide\xspace}
\newcommand \rexris {règles exis\-ten\-tielles rigides\xspace}

\newcommand \rsim {règle de simplification\xspace}
\newcommand \rsims {règles de simplification\xspace}

\newcommand \rtl {réti\-culé\xspace}
\newcommand \rtls {réti\-culés\xspace}

\newcommand \rmq {\rcm de quotients\xspace} 
\newcommand \rvq {\rcv par quotients\xspace} 
\newcommand \rmqs {\rcms de quotients\xspace} %
\newcommand \rvqs {\rcvs par quotients\xspace} %

\newcommand \rpf {réduite-de-présen\-tation-finie\xspace}
\newcommand \rpfs {réduites-de-présen\-tation-finie\xspace}


\newcommand \sad {\salg dynamique\xspace}
\newcommand \sads {\salgs dynamiques\xspace}

\newcommand \sagq {semi\agq}
\newcommand \sagqs {semi\agqs}

\newcommand \sagc {\sagq continue\xspace}
\newcommand \sagcs {\sagqs continues\xspace}

\newcommand \salg {structure \agq}
\newcommand \salgs {structures \agqs}

\newcommand \scentrel {relation semi-implicative\xspace}
\newcommand \scentrels {relations semi-implicatives\xspace}

\newcommand \scf {schéma fini\-taire\xspace}
\newcommand \scfs {schémas fini\-taires\xspace}

\newcommand \scl {schéma \elr}
\newcommand \scls {schémas \elrs}

\newcommand \sdo {\sdr \orte}
\newcommand \sdos {\sdrs \ortes}

\newcommand \sdr {somme directe\xspace}
\newcommand \sdrs {sommes directes\xspace}

\newcommand \sdz {sans \dvz}

\newcommand \sfio {sys\-tème fondamental d'\idms ortho\-gonaux\xspace}
\newcommand \sfios {sys\-tèmes fondamentaux d'\idms ortho\-gonaux\xspace}

\newcommand \sgr {\sys \gtr}
\newcommand \sgrs {\syss \gtrs}

\newcommand \slgb {stricte\-ment \lgb}
\newcommand \slgbs {stricte\-ment \lgbs}

\newcommand \sli {\sys \lin}
\newcommand \slis {\syss \lins}

\newcommand \smq {symé\-trique\xspace}
\newcommand \smqs {symé\-triques\xspace}

\newcommand \spb {sépa\-rable\xspace}  
\newcommand \spbs {sépa\-rables\xspace}

\newcommand \spe {spéci\-fi\-cation\xspace}
\newcommand \spes {spéci\-fi\-cations\xspace}

\newcommand \spi {\spe incomplète\xspace}
\newcommand \spis {\spes incomplètes\xspace}

\newcommand \spl {sépa\-rable\xspace}  
\newcommand \spls {sépa\-rables\xspace}

\newcommand \spo {semipolynôme\xspace}
\newcommand \spos {semipolynômes\xspace}

\newcommand \spt{sépa\-ra\-bi\-lité\xspace}

\newcommand \srg {suite régu\-lière\xspace}
\newcommand \srgs {suites régu\-lières\xspace}

\newcommand \stf {strictement fini\xspace}
\newcommand \stfs {strictement finis\xspace}
\newcommand \stfe {strictement finie\xspace}
\newcommand \stfes {strictement finies\xspace}

\newcommand \stl {stablement libre\xspace}
\newcommand \stls {stablement libres\xspace}

\newcommand \stm {strictement\xspace}

\newcommand \str {\stm réticulé\xspace}
\newcommand \stre {\stm réticulée\xspace}
\newcommand \strs {\stm réticulés\xspace}
\newcommand \stres {\stm réticulées\xspace}

\newcommand \sul {supplé\-men\-taire\xspace}
\newcommand \suls {supplé\-men\-taires\xspace}

\newcommand \Sut {Support\xspace}
\newcommand \Suts {Supports\xspace}
\newcommand \sut {support\xspace}

\newcommand \syc {\sys de coordon\-nées\xspace}
\newcommand \sycs {\syss de coordon\-nées\xspace}

\newcommand \syp {\sys \poll}
\newcommand \syps {\syss \polls}

\newcommand \sys {sys\-tème\xspace}
\newcommand \syss {sys\-tèmes\xspace}

\newcommand \talg {théorie \agq}
\newcommand \talgs {théories \agqs}

\newcommand \tco {théorie cohé\-rente\xspace}
\newcommand \tcos {théories cohé\-rentes\xspace}

\newcommand \tdy {théorie dyna\-mique\xspace}
\newcommand \tdys {théories dyna\-miques\xspace}

\newcommand \tel {théorie exis\-ten\-tielle\xspace}
\newcommand \tels {théories exis\-ten\-tielles\xspace}

\newcommand \telri {théorie exis\-ten\-tielle rigide\xspace}
\newcommand \telris {théories exis\-ten\-tielles rigides\xspace}

\newcommand \tf {de type fini\xspace}

\newcommand \tfo {théorie formelle\xspace}
\newcommand \tfos {théorie formelles\xspace}

\newcommand \tgm {théorie \gmq}
\newcommand \tgms {théories \gmqs}

\newcommand \Tho {Théo\-rème\xspace}
\newcommand \Thos {Théo\-rèmes\xspace}
\newcommand \tho {théo\-rème\xspace}
\newcommand \thos {théo\-rèmes\xspace}

\newcommand \thoc {théo\-rème$\mathbf{^*}$~}

\newcommand \tpe {théorie \peq}
\newcommand \tpes {théories \peqs}

\newcommand \trdi {treil\-lis dis\-tri\-bu\-tif\xspace}
\newcommand \trdis {treil\-lis dis\-tri\-bu\-tifs\xspace}

\newcommand \trel {trans\-for\-mation \elr}
\newcommand \trels {trans\-for\-mations \elrs}

\newcommand \umd {unimo\-du\-laire\xspace}
\newcommand \umds {unimo\-du\-laires\xspace}

\newcommand \unt {uni\-taire\xspace}
\newcommand \unts {uni\-taires\xspace}

\newcommand \uvl {uni\-versel\xspace}
\newcommand \uvle {uni\-ver\-selle\xspace}
\newcommand \uvls {uni\-versels\xspace}
\newcommand \uvles {uni\-ver\-selles\xspace}


\newcommand \vfn {véri\-fi\-cation\xspace}
\newcommand \vfns {véri\-fi\-cations\xspace}

\newcommand \vmd {vec\-teur \umd}
\newcommand \vmds {vec\-teurs \umds}

\newcommand \zed {z\'{e}ro-di\-men\-sionnel\xspace}
\newcommand \zede {z\'{e}ro-di\-men\-sion\-nelle\xspace}
\newcommand \zeds {z\'{e}ro-di\-men\-sion\-nels\xspace}
\newcommand \zedes {z\'{e}ro-di\-men\-sion\-nelles\xspace}

\newcommand \zedr {\zed réduit\xspace}
\newcommand \zedrs {\zeds réduits\xspace}

\newcommand \zmt {\tho de Zariski-Grothen\-dieck\xspace}


\newcommand \cof {cons\-truc\-tif\xspace}
\newcommand \cofs {cons\-truc\-tifs\xspace}

\newcommand \cov {cons\-truc\-tive\xspace}
\newcommand \covs {cons\-truc\-tives\xspace}

\newcommand \coma {\maths\covs}
\newcommand \clama {\maths clas\-siques\xspace}

\renewcommand \cot {cons\-truc\-ti\-vement\xspace}

\newcommand \matn {mathé\-ma\-ticien\xspace}
\newcommand \matne {mathé\-ma\-ti\-cienne\xspace}
\newcommand \matns {mathé\-ma\-ticiens\xspace}
\newcommand \matnes {mathé\-ma\-ti\-ciennes\xspace}

\newcommand \maths {mathé\-ma\-tiques\xspace}
\newcommand \mathe {mathé\-ma\-tique\xspace}

\newcommand \prco {démons\-tration \cov}
\newcommand \prcos {démons\-trations \covs}



\theoremstyle{plain}
\newtheorem{ftheorem}{Théorème}[section]
\newtheorem{fthdef}[ftheorem]{Théorème et définition}
\newtheorem{fpstf}[ftheorem]{Positivstellensatz formel}
\newtheorem{fpst}[ftheorem]{Positivstellensatz}
\newtheorem{flemma}[ftheorem]{Lemme}
\newtheorem{fcorollary}[ftheorem]{Corolaire}
\newtheorem{fconjecture}[ftheorem]{Conjecture}
\newtheorem{fproposition}[ftheorem]{Proposition}
\newtheorem{fprpta}[ftheorem]{Propriétés attendues}
\newtheorem{fpropdef}[ftheorem]{Proposition et définition}
\newtheorem{ffact}[ftheorem]{Fait}
\newtheorem{fconvention}[ftheorem]{Convention}
\newtheorem{fplcc}[ftheorem]{Principe local-global concret}

\newtheorem{ftheoremc}[ftheorem]{Th\'{e}or\`{e}me\etoz}
\newtheorem{flemmac}[ftheorem]{Lemme\etoz}
\newtheorem{fcorollaryc}[ftheorem]{Corolaire\etoz}
\newtheorem{fproprietec}[ftheorem]{Propri\'{e}t\'{e}\etoz}
\newtheorem{fpropositionc}[ftheorem]{Proposition\etoz}
\newtheorem{ffactc}[ftheorem]{Fait\etoz}

\newtheorem{fatheorem}{Théorème}[section]
\newtheorem{falemma}[fatheorem]{Lemme}
\newtheorem{facorollary}[fatheorem]{Corolaire}
\newtheorem{fapropriete}[fatheorem]{Propriété}
\newtheorem{faproposition}[fatheorem]{Proposition}
\newtheorem{fapropdef}[fatheorem]{Proposition et définition}
\newtheorem{fafact}[fatheorem]{Fait}

\theoremstyle{definition}
\newtheorem{frstr}[ftheorem]{Règles structurelles}
\newtheorem{frstra}[ftheorem]{Règles structurelles admissibles}
\newtheorem{fdefinition}[ftheorem]{Définition}
\newtheorem{fdfni}[ftheorem]{Définition informelle}
\newtheorem{fdefinitions}[ftheorem]{Définitions}
\newtheorem{fexample}[ftheorem]{Exemple}
\newtheorem{fexamples}[ftheorem]{Exemples}
\newtheorem{fnotation}[ftheorem]{Notation}
\newtheorem{fproblem}[ftheorem]{Problème}
\newtheorem{fquestion}[ftheorem]{Question}
\newtheorem{fquestions}[ftheorem]{Questions}

\newtheorem{fdefinitionc}[ftheorem]{Définition\etoz}
\newtheorem{fdefinota}[ftheorem]{Définition et notation} 
\newtheorem{faquestion}[fatheorem]{Question}
\newtheorem{fadefinition}[fatheorem]{Définition}

\theoremstyle{remark}
\newtheorem{fremark}[ftheorem]{Remarque}
\newtheorem{fremarks}[ftheorem]{Remarques}
\newtheorem{faremark}[fatheorem]{Remarque}
\newtheorem{faremarks}[fatheorem]{Remarques}
\newtheorem{fcomment}[ftheorem]{Commentaire}

\newcommand{\sadr}{\sa{AP}}
\newcommand{\sadrp}{\sa{APp}}
\newcommand{\rdc}{rdr-clos}
\newcommand{\irdc}{idéal \rdc}
\renewcommand{\di}{\,{\vert}\,}
\newcommand{\ndi}{\nmid}
\newcommand{\dr}{\,\vert\hspace{-.22ex}\vert\,}
\newcommand{\drp}{\,\vert\hspace{-.22ex}\vert'\,}
\newcommand{\drd}{\;\cdot\hspace{-.7ex}\mid\,}
\newcommand \drI[1]{\,\big\vert\hspace{-.22ex}\big\vert_#1}
\newcommand \dri[1]{\,\vert\hspace{-.22ex}\vert_#1}
\newcommand \dre[1]{\,\big\vert\hspace{-.22ex}\big\vert^#1}
\newcommand \dreu[1]{\dre{{#1\uparrow\,}}}
\newcommand \dred[1]{\dre{{#1\downarrow\,}}}

\FrenchFootnotes

\title{Treillis distributifs et espaces spectraux,\\ un petit dictionnaire}
\author{Henri Lombardi, (\thanks{Laboratoire de Mathématiques de Besançon, Université Marie et Louis Pasteur. \url{http://hlombardi.free.fr/}
email: \texttt{henri.lombardi@univ-fcomte.fr}. 
}~) \today}
\date{}
\emptythanks\setcounter{footnote}{0}
\setcounter{equation}{0}
\maketitle

\rdb
\label{beginfrench}

\vspace{-1em}

\begin{abstract}
La catégorie des treillis distributifs et celle des espaces spectraux
sont antiéquivalentes (en \clama). Nous proposons ici un petit dictionnaire
pour cette antiéquivalence. Nous indiquons comment un certain nombre de théorèmes étranges des \clama obtiennent un contenu \cof grâce à cette antiéquivalence, même dans le cas, fréquent, où les points des espaces spectraux considérés n'ont pas de contenu \cof clair. 
\end{abstract}

\noindent Note: Cet article est une version française légèrement développée du chapitre
\textsl{Spectral spaces versus distributive lattices: a dictionary},
 dans le livre \gui{Advances in rings, modules and factorizations. Selected papers based on the presentations at the international conference on rings and factorizations, Graz, Austria, February 19--23, 2018}, paru en 2020 chez Springer, ISBN 978-3-030-43415-1. On a aussi corrigé quelques bugs typographiques et amélioré la bibliographie.

\newpage
\setcounter{tocdepth}{4}

\startcontents[french]
\printcontents[french]{}{1}{}

\section*{Introduction} 
\addcontentsline{toc}{section}{Introduction}

On donne un dictionnaire entre \clama et \coma pour les \prts liées aux espaces spectraux d'une part et à leurs traductions en termes des \trdis duaux d'autre part.

Ce travail se situe dans le style des \coma à la Bishop (\cite{fBi67,fBB85,fBR1987,fACMC,fMRR,fYen2015}).

\section{Treillis distributifs et espaces spectraux: généralités}\label{fsecdival0}
Références: \cite{fCC00,fCL05,fDST2019,fLom06,fLom-tgac,fSto37}, \cite[Chapter 4]{fBW74} et \cite[Chapitres XI et XIII]{fACMC}.

\subsection{L'article fondateur de Stone}

En langage moderne l'article \cite{fSto37} dit à très peu près ceci.
{
\sl La catégorie des \trdis est, en \clama, anti\eqve à la catégorie des espaces spectraux.}

\medskip En \clama
un {\sl \id premier} $\fp$ d'un \trdi $\gT\neq \Un$ est un \id dont
le complé\-mentaire $\ff$ est un filtre (qui est alors un {\sl
filtre premier}).  On a alors le treillis quotient $\gT/(\fp=0,\ff=1)\simeq\Deux$.  Il
revient au même de se donner un \idep de~$\gT$ ou un morphisme de
\trdis $\gT\rightarrow \Deux$.

Nous noterons $\theta_\fp:\gT\to\Deux$ l'\homo
associé à l'\idep $\fp$.

On vérifie facilement que si $S$ est une partie génératrice du
\trdi $\gT$, un \idep~$\fp$ de $\gT$ est complètement
caractérisé par sa trace sur $S$ (cf.  \cite{fCC00}).

\smallskip Le \textsl{spectre d'un \trdi $\gT$} est l'ensemble $\Spec
\,\gT$ de ses \ideps, muni de la topologie suivante: une base
d'ouverts est donnée par les $\DT(a)$ définis comme suit ($a\in \gT$)
$$
\DT(a)\eqdefi\sotq{\fp\in\Spec
\,\gT}{a\notin\fp}=\sotq{\fp}{\theta_\fp(a)=1} .
$$
On vérifie que
\begin{equation} \label{feqDa}
\left.\begin{array}{rclcrcl}
  \DT(a\vi b)   & =  & \DT(a)\cap \DT(b) ,&\quad & \DT(0)  & =  & 
\emptyset  ,\\
  \DT(a\vu b)   & =  & \DT(a)\cup \DT(b) ,&&  \DT(1) & =  &  
\Spec\,\gT.
  \end{array}
\right\}
\end{equation}

Le complémentaire de $\DT(a)$ est un fermé que l'on note $\VT(a)$.
L'adhérence de $\so\fp$ est l'ensemble des $\fq$ qui contiennent $\fp$.
Les points fermés sont les \idemas.

On étend la notation $\VT(a)$ comme suit: pour $I\subseteq\gT$, on
pose $\VT(I)\eqdefi\bigcap_{x\in I}\VT(x)$.  \hbox{Si $\cI_\gT(I)=\fII$}, on
a $\VT(I)=\VT(\fII)$.  On dit parfois que $\VT(I)$ est \textsl{la
sous-variété de $\SpecT$ associée à $I$}.

\medskip\noindent
{\bf Définition.} 
Un espace topologique homéomorphe à un espace $\Spec(\gT)$
est appelé un \textsl{espace spectral}. 
Tout \homo $\varphi :\gT\rightarrow \gT'$ de \trdis fournit
\textsl{par dualité}\footnote{On définit $\Spec(\varphi)(\fp)=\varphi^{-1}(\fp)$: on a alors avec $\fq=\Spec(\varphi)(\fp)$, $\theta_\fq=\theta_\fp\circ \varphi$.} une application continue $\Spec(\varphi):\Spec\,\gT'\rightarrow \Spec \,\gT$, (souvent notée $\varphi\sta$) qui est appelée une
\textsl{application spectrale}. 

\medskip Les espaces spectraux proviennent de l'article  \cite{fSto37}. Ils sont appelés  \textsl{espaces cohérents} dans \cite{fJoh1986}. \cite{fBW74} les appelle des \textsl{Stone spaces}.  
Le nom d'espace spectral vient de l'article  \cite{fHoc1969} qui démontre que tout espace spectral peut être obtenu comme spectre de Zariski d'un anneau commutatif.
Néanmoins ce résultat précis semble sans aucune importance réelle.

Avec la logique classique et l'axiome du choix, l'espace $\Spec \,\gT$
a \gui{suffisamment de points}: on peut retrouver le treillis $\gT$
à partir de son spectre.  Voici comment.

Tout d'abord on a le \tho de Krull.


\medskip\noindent
{\bf \Tho de Krull }\label{fThKrull} (en \clama).\\ 
{\sl Supposons que $\fJ$ est un \id, $\fF$ un filtre et 
$\fJ\cap\fF=\emptyset$.
Alors il existe un \idep $\fP$ tel \hbox{que $\fJ\subseteq\fP$} et
$\fP\cap\fF=\emptyset$.
  }

\medskip
On en déduit que:
\begin{itemize}
\item l'application $a\in\gT\,\mapsto\,\DT(a)\in\cP(\Spec\,\gT)$
est injective: elle identifie $\gT$ à un treillis d'ensembles 
(\textsl{\tho de
représentation de Birkhoff});
\item si $\varphi : \gT\to\gT'$ est un \homo injectif l'application
$\varphi\sta:\Spec\,\gT'\to\Spec\,\gT$ obtenue par dualité est 
surjective;
\item tout \id de $\gT$ est intersection des \ideps qui le 
contiennent;
\item L'application $\fII\mapsto \VT(\fII)$, des \ids de $\gT$ vers 
les fermés
de $\Spec\,\gT$ est un \iso d'ensembles ordonnés (pour l'inclusion 
et
l'inclusion renversée).
\end{itemize}

\smallskip On montre aussi que les \oqcs de $\Spec \,\gT$ sont exactement les
$\DT(a)$.  D'après les \egts (\ref{feqDa}) les \oqcs de $\Spec \,\gT$
forment un \trdi de parties de~$\Spec \,\gT$, isomorphe à $\gT$.

À partir d'un espace spectral $X$ on peut considérer le \trdi
$\OQC(X)$ formé par ses \oqcs.  Puisque pour tout \trdi $\gT$,
$\OQC(\Spec(\gT))$ est canoniquement isomorphe à $\gT$, pour tout
espace spectral $X$, $\Spec(\OQC(X))$ est canoniquement 
homéo\-morphe~à~$X$.

\smallskip   Pour qu'une application continue entre
espaces spectraux soit spectrale il faut et il suffit que l'image
réciproque de tout \oqc soit un \oqc.

L'article fondateur de Stone démontre pour l'essentiel que la catégorie spectrale ainsi définie (espaces spectraux et applications spectrales) est antiéquivalente à
celle des \trdis \cite[{II-3.3}, coherent locales]{fJoh1986}, \cite{fSto37}. En fait, cette antiéquivalence signifie pour l'essentiel que les \oqcs d'un 
espace~$\Spec\gT$ forment un \trdi canoniquement isomorphe à $\gT$. 
Pour les morphismes, il faut
se restreindre, parmi les applications continues, aux applications spectrales. L'énoncé devient nettement plus subtil lorsque l'on donne une \dfn des espaces spectraux en termes purement d'espaces topologiques, comme dans la remarque qui suit.


\medskip 
\rem
Une \dfn purement topologique des espaces spectraux 
est la sui\-vante~(\cite{fSto37}).
\begin{itemize}
\item L'espace est un espace de Kolmogoroff (i.e., de type $\mathrm{T}_0$): 
étant donnés deux points il existe un voisinage de l'un des deux qui ne contient pas l'autre.
\item L'espace est \qc.
\item L'intersection de deux \oqcs est un \oqc.
\item Tout ouvert est réunion d'\oqcs.
\item Pour tout fermé $F$ et pour tout ensemble $S$ d'\oqcs tels 
que 
$$\textstyle F\cap
\bigcap_{U\in S'} U\neq \emptyset\,\hbox{   pour toute partie finie  }\,S'
\,\hbox{  de  }\,S
$$ 
on a aussi
$F\cap \bigcap_{U\in S} U\neq \emptyset$.
\end{itemize}
En présence des quatre premières propriétés, la dernière peut se
reformuler comme suit d'après \cite{fHoc1969}. 
\begin{itemize}
\item Tout fermé irréductible admet un point
générique.
\end{itemize}

\smallskip En \coma les espaces spectraux manquent souvent de points, et on essaie de traduire les discours des \clama sur les espaces spectraux (très fréquents en algèbre) en des discours \cofs sur les \trdis correspondants.
On a en fait un dictionnaire assez complet que nous exposons dans la section \ref{fsecAntiEquiv}.

\smallskip  Deux autres topologies spectrales sont obtenues sur l'ensemble $\SpecT$ en variant les ouverts de base: on peut prendre les complémentaires
$\fV(a)$ des ouverts de base $\fD(a)$, ou encore les combinaisons booléennes d'ouverts $\fD(a)$. Dans ce dernier cas on parle de la \gui{topologie constructible} (dans la littérature anglaise, de la \gui{patch topology}).
Les \oqcs de ces espaces sont les \trdis respectivement isomorphes à $\gT\eci$,  (le treillis opposé à $\gT$) et à $\Bo(\gT)$ (l'\agB engendrée par $\gT$). 

\subsubsection*{Sous-espaces spectraux et treillis quotients} 
\addcontentsline{toc}{subsubsection}{Sous-espaces spectraux et treillis quotients}

Références: \cite{fBW74,fCLQ2006}.

 Le théorème suivant explique que la notion de \textsl{sous-espace spectral} est traduite par la notion de \textsl{\trdi quotient}. Quelques détails \suls sont ajoutés. Voir aussi le \thref{fth-dico-trdi-spec-mor1}.  

\begin{ftheorem}[Définition et caractérisations des sous-espaces 
spectraux]
\label{fpropSESP}  Soit $\gT'$ un treillis quotient de $\gT$ et $\pi:\gT\to\gT'$ la
projection canonique.  Notons $X'=\Spec\,\gT'$, $X=\Spec\,\gT$ et
$\pi^\star:X'\to X$ l'injection duale de $\pi$.  
\begin{enumerate}
\item  L'injection $\pi^\star$ identifie  $X'$ à un \textsl{sous-espace
topologique} de $X$. En outre $\OQC(X')=\sotq{U\cap X'}{U\in\OQC(X)}$. 
On dit que \emph{$X'$ est un sous-espace spectral de~$X$.}
\item Pour qu'une partie $X'$ d'un espace spectral $X$ soit un sous-espace
spectral il faut et suffit que les conditions suivantes soient 
vérifiées: \\
-- La topologie induite par $X$  fait de $X'$ un espace spectral, et\\
--  $\OQC(X')=\sotq{U\cap X'}{U\in\OQC(X)}$.
\item Une partie $X'$ d'un espace spectral $X$ est un sous-espace 
spectral \ssi
elle est fermée pour la topologie constructible.
\item Si $Z$ est une partie arbitraire d'un espace spectral 
$X=\Spec\,\gT$ son
adhérence pour la topologie constructible est égale à 
$X'=\Spec\,\gT'$ où
$\gT'$ est le treillis quotient de~$\gT$ défini par la relation de 
préordre
$\preceq$ suivante:
\begin{equation} \label{feqSSES}
a\preceq b\quad \Longleftrightarrow\quad (\DT(a)\cap Z)\subseteq 
(\DT(b)\cap Z)
\end{equation}
\end{enumerate}
\end{ftheorem}

\begin{fproposition}[Ouverts et fermés de base]
\label{fpropositionOFBSES}  Soit $\gT$ un \trdi et $X=\Spec\,\gT$.
\begin{enumerate}
\item $\DT(a)$ est un sous-espace spectral de $X$ canoniquement 
homéomorphe
à $
\Spec(\gT/(a=1))$.
\item  $\VT(b)$ est un sous-espace spectral de $X$ canoniquement 
homéomorphe
à $
\Spec(\gT/(b=0))$.
\end{enumerate}
\end{fproposition}

\begin{fproposition}[Sous-ensembles fermés de $\Spec\,\gT$]
\label{fpropositionFSES} ~
\begin{enumerate}
\item Un fermé arbitraire de $\Spec\,\gT$ est de la forme
$\VT(\fJ):=\bigcap_{x\in \fJ}\VT(x)$ où $\fJ$ est un \id arbitraire 
de $\gT$.
C'est un sous-espace spectral et il correspond au quotient 
$\gT/(\fJ=0)$.
\item
L'intersection d'une famille de fermés correspond au sup de la 
famille d'idéaux. La réunion de deux fermés correspond à l'intersection des deux idéaux.
\item \`A l'adhérence de $\DT(x)$ correspond le quotient 
$\gT/((0:x)=0)$.
\item\label{fenumtra}
Le treillis $\gT/((a:b)=0)$ est le quotient correspondant à 
$\ov{\VT(a)\cap
\DT(b)}$.
\end{enumerate}
\end{fproposition}

\subsubsection*{Recollement de treillis quotients et de sous-espaces spectraux}
\addcontentsline{toc}{subsubsection}{Recollement de treillis quotients et de sous-espaces spectraux}

En \alg commutative, si $(\xn)$ est un système d'\eco d'un anneau~$\gA$
le morphisme $\gA\to\prod_{i\in\lrbn}\gA[1/x_i]$ identifie $\gA$ à un sous-produit
fini de ses localisés. Ce résultat peut être considéré comme la version \cov d'un \plg abstrait qui dit que certaines \prts de $\gA$ sont satisfaites \ssi elles sont satisfaites après localisation en n'importe quel \idep.

\smallskip De la même manière on peut récupérer un \trdi 
à partir
d'un nombre fini de ses quotients,
si l'information qu'ils contiennent est \gui{suffisante}. On peut 
voir ceci au
choix comme une procédure de recollement (de passage du local au 
global), ou
comme une version du théorème des restes chinois pour les \trdis. Voyons 
les choses
plus précisément.

\begin{fdefinition}
\label{fdefRecolTD}
Soit $\gT$ un \trdi, $(\fa_i)_{i\in\lrbn}$ (resp. 
$(\ff_i)_{i\in\lrbn}$)
une famille finie d'\ids (resp. de filtres)  de $\gT$.  On dit que 
les \ids
$\fa_i$ \textsl{recouvrent $\gT$} si \hbox{$\bigcap_i\fa_i=\so{0}$}. De 
même on dit
que les filtres $\ff_i$ \textsl{recouvrent $\gT$} si 
$\bigcap_i\ff_i=\so{1}$.
\end{fdefinition}

Pour un \id $\fb$ nous écrivons $x\equiv y\mod\fb$ comme 
abréviation
pour  $x\equiv y\alb\mod (\fb=0)$. Rappelons que pour $s\in\gT$ le quotient $\gT/(s=0)$ est isomorphe au filtre principal $\uar s$ que l'on voit comme un \trdi dont l'\elt $0$ est $s$.

\begin{ffact}
\label{ffactRecolTD}
Soit $\gT$ un \trdi, $(\fa_i)_{i\in\lrbn}$ une famille finie 
d'\ids principaux ($\fa_i=\dar s_i$)  de
$\gT$ et $\fa=\bigcap_i\fa_i$.
\begin{enumerate}
\item Si $(x_i)$ est une famille d'\elts de $\gT$ telle que pour 
chaque $i,j$ on
a $x_i\equiv x_j\;\mod\;\fa_i\vu\fa_j$, alors il existe un unique $x$ 
modulo
$\fa$ vérifiant:  $x\equiv x_i\;\mod\;\fa_i\;(i=1,\ldots ,n)$.
\item Notons $\gT_i=\gT/(\fa_i=0)$, 
$\gT_{ij}=\gT_{ji}=\gT/(\fa_i\vu\fa_j=0)$,
$\pi_i:\gT\to\gT_i$ et $\pi_{ij}:\gT_i\to\gT_{ij}$ les projections 
canoniques.
Si les $\fa_i$ recouvrent $\gT$, $(\gT,(\pi_i)_{i\in\lrbn})$  est 
la limite
projective du diagramme $((\gT_i)_{1\leq i\leq n},(\gT_{ij})_{1\leq 
i<j\leq
n};(\pi_{ij})_{1\leq i\neq j\leq n})$ (voir la figure ci-après)
\item Soit maintenant $(\ff_i)_{i\in\lrbn}$ une famille finie de 
filtres principaux,
notons $\gT_i=\gT/(\ff_i=1)$, 
$\gT_{ij}=\gT_{ji}=\gT/(\ff_i\cup\ff_j=1)$,
$\pi_i:\gT\to\gT_i$ et $\pi_{ij}:\gT_i\to\gT_{ij}$ les projections 
canoniques.
Si les $\ff_i$ recouvrent $\gT$, $(\gT,(\pi_i)_{i\in\lrbn})$  est 
la limite
projective du diagramme $((\gT_i)_{1\leq i\leq n},(\gT_{ij})_{1\leq 
i<j\leq
n};(\pi_{ij})_{1\leq i\neq j\leq n})$.
\end{enumerate}
\end{ffact}
 {\small
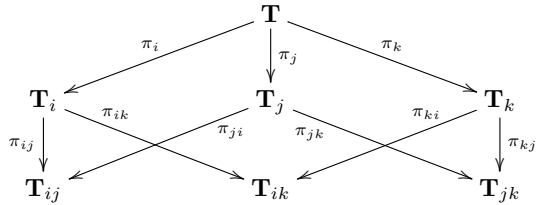
\hspace*{10em}{
\xymatrix @R=2em @C=7em{
          &  \gT \ar[rd]^{\pi _{k}}\ar[d]^{\pi _{j}}\ar[ld]_{\pi _{i}}\\
 \gT _i\ar[d]_{\pi _{ij}}\ar@/-0.75cm/[dr]^{\pi _{ik}} &
     \gT _j\ar@/-.8cm/[dl]_{\pi _{ji}}\ar@/-.8cm/[dr]^{\pi _{jk}} &
        \gT _k\ar@/-0.75cm/[dl]_{\pi _{ki}}\ar[d]^{\pi _{kj}} &
\\
 \gT _{ij}  & 
    \gT _{ik}   & 
      \gT _{jk}   
}
}}

\medskip Il y a aussi une procédure de recollement proprement dit\footnote{En algèbre commutative une telle procédure de recollement fonctionne pour les modules. Par contre recoller des anneaux commutatifs de manière formelle nécessite de passer à la catégorie des schémas de Grothendieck.}.

\begin{fproposition}[Recollement abstrait de \trdis]
\label{fpropRecolTD} 
  Supposons donnés un ensemble fini totalement ordonné~$I$ et dans la catégorie des \trdis  un diagramme

\snic{\big((\gT_i)_{i\in I},(\gT_{ij})_{i<j\in I},(\gT_{ijk})_{i<j<k\in I};
(\pi_{ij})_{i\neq j},(\pi_{ijk})_{i< j, j\neq k\neq i}\big)}

\noindent 
comme dans la figure ci-après, 
ainsi qu'une famille d'\elts 

\snic
{(s_{ij})_{i\neq j\in I}\in \prod\nolimits_{i\neq j\in I}\gT_{i}}

\noindent satisfaisant les conditions suivantes:
\begin{itemize}
\item le diagrammme est commutatif ($\pi_{ijk}\circ \pi_{ij}=\pi_{ikj}\circ \pi_{ik}$ pour tous $i$, $j$, $k$ distincts), 
\item pour $i\neq j$, $\pi_{ij}$ est un morphisme de passage au quotient par l'\id $\dar s_{ij}$,
\item pour $i$, $j$, $k$ distincts, $\pi_{ij}(s_{ik})=\pi_{ji}(s_{jk})$ et  $\pi_{ijk}$ est un morphisme de passage au quotient par \hbox{l'\id $\dar\pi_{ij}(s_{ik})$}.   
\end{itemize}

\medskip  {\small
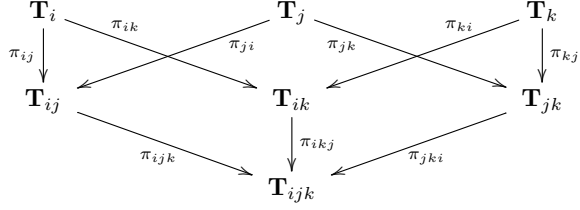
\hspace*{10em}
\xymatrix @R=2em @C=7em{
 \gT_i\ar[d]_{\pi _{ij}}\ar@/-0.75cm/[dr]^{\pi _{ik}} &
     \gT_j\ar@/-.8cm/[dl]_{\pi _{ji}}\ar@/-.8cm/[dr]^{\pi _{jk}} &
        \gT_k\ar@/-0.75cm/[dl]_{\pi _{ki}}\ar[d]^{\pi _{kj}} &
\\
 ~\gT_{ij}~ \ar[rd]_{\pi _{ijk}} & 
    ~\gT_{ik}~  \ar[d]^{\pi _{ikj}} & 
      ~\gT_{jk}~  \ar[ld]^{\pi _{jki}} 
\\
   &  ~\gT_{ijk}~ 
\\
}
}

\smallskip \noindent Alors si $\big(\gT\,;\,(\pi_i)_{i\in I}\big)$ est la limite projective du diagramme, les~\hbox{$\pi_i:\gT\to \gT_i$} forment un recouvrement par quotients principaux de $\gT$, et le diagramme est isomorphe à celui obtenu
dans le fait~\ref{ffactRecolTD}.
Plus précisément, il existe des $s_i\in\gT$ tels que chaque~$\pi_i$ est un morphisme de passage au quotient par l'\id $\dar s_i$ et $\pi_i(s_j)=s_{ij}$ pour tous $i\neq j$.

\noindent Le \tho analoque est valable pour les quotients par des filtres principaux.
\end{fproposition}

\begin{fremark} \label{fremRecSSSpec} 
{\rm Nous laissons \alec le soin de traduire les résultats précédents
en termes de recollements de sous-espaces spectraux en utilisant le dictionnaire fourni pas le théorème \ref{fpropSESP}.   
}\end{fremark}

\subsubsection*{Treillis et J-spectrum de Heitmann}
\addcontentsline{toc}{subsubsection}{Treillis et J-spectrum  de Heitmann}

Un \id $\fm$ d'un \trdi $\gT$ non trivial (i.e. distinct de $\Un$) est
dit \textsl{maximal} \hbox{si $\gT/(\fm=0)\,=\,\Deux$}, \cad si $1\notin\fm$ et
$\Tt x\in\gT\;(x\in\fm$ ou $\Ex y\in \fm \;x\vu y=1)$.

Il revient au même de dire qu'il s'agit d'un \id \gui{maximal
parmi les \ids stricts}.

En \clama on a le lemme suivant, qui vaut aussi bien pour le treillis trivial~$\Un$.
\begin{flemma}
\label{flemHspec1}
Dans un \trdi $\gT\neq\Un$ l'intersection des \idemas est égale à 
l'\id
$$ \sotq{a\in\gT}{\forall x\in\gT \;(a\vu x = 1 \Rightarrow  
x=1)}.
$$
On l'appelle le \textsl{radical de Jacobson de $\gT$}. On le note 
$\JT(0)$. \\
Plus généralement l'intersection des \idemas contenant un \id 
strict $\fJ$
est égale à l'\id
\begin{equation} \label{feqRJJ}
\JT(\fJ)\,=\, \sotq{a\in\gT}{\forall x\in\gT \;(a\vu x  = 1 
\Rightarrow \Ex
z\in \fJ \;\;z\vu x=1)}
\end{equation}
On l'appelle le \textsl{radical de Jacobson de l'\id $\fJ$}. En 
particulier:
\begin{equation} \label{feqRJb}
\JT(\dar b)=\sotq{a\in\gT}{\forall x\in\gT \;(\,a\vu x = 1\; 
\Rightarrow
\;\;b\vu x=1\,)}
\end{equation}
\end{flemma}

En \coma les \egts \pref{feqRJJ} et \pref{feqRJb} servent de \dfn.

Un quotient intéressant de $\gT$, qui n'est ni un quotient par un \id ni un 
quotient
par un filtre, est le treillis de Heitmann.


On appelle \textsl{treillis de Heitmann de $\gT$} et on note $\He(\gT)$ 
le treillis quotient de~$\gT$ obtenu en remplaçant sur $\gT$ la 
relation d'ordre $\leq_\gT $  par la relation de préordre 
$\preceq_{\He(\gT)}$
définie comme suit

\vspace{-1em}
\begin{equation} \label{feqdefHeT}
\begin{array}{rcl}\qquad 
a\preceq_{\He(\gT)} b & \equidef  &   \JT(a)\subseteq\JT(b)  
  \end{array}
  \end{equation}
Ce treillis quotient peut être identifié à l'ensemble des 
\ids $\JT(a)$,
avec la projection canonique
$$ \gT\longrightarrow \He(\gT),\quad a\longmapsto \JT(a)$$

Dans un article remarquable (\cite{fHei84}) R. Heitmann explique que la notion
usuelle de j-spectrum pour un anneau commutatif n'est pas la bonne 
dans le cas
non noethérien car elle ne correspond pas à un espace spectral au sens de
Stone. Il introduit la modification suivante de la \dfn usuelle: au lieu
de considérer l'ensemble des \ideps qui sont intersections 
d'\idemas il propose de considérer l'adhérence du spectre maximal dans le 
spectre premier, adhérence à prendre au sens de la topologie 
constructible (la
\gui{patch topology}).

\begin{fdefinitions}
\label{fdefHspec1}
Soit $\gT$ un \trdi.
\begin{enumerate}
\item  On note $\Max\,\gT$ le sous-espace topologique de $\Spec\,\gT$ formé par les \idemas de~$\gT$. 
On l'appelle le \textsl{spectre maximal de~$\gT$}.
\item  On note $\jspec\,\gT$ le sous-espace topologique de 
$\Spec\,\gT$ formé
par les $\fp$ qui vérifient l'égalité \hbox{$\JT(\fp)=\fp$}, \cad les  \ideps $\fp$ 
qui sont
intersections d'\idemas (c'est le \textsl{j-spectrum} \gui{usuel}).
\item On appelle \textsl{$\rJ$-spectre de Heitmann de $\gT$} et on note
$\Jspec\,\gT$ l'adhérence du spectre maximal dans $\Spec\,\gT$, 
adhérence à prendre au sens de la topologie constructible. Ce spectre est un sous-espace spectral de $\Spec\,\gT$, canoniquement homéomorphe à $\Spec(\He(\gT))$.
\item  On note $\Min\,\gT$ le sous-espace topologique de $\Spec\,\gT$ 
formé par les \ideps minimaux de $\gT$. On l'appelle le \textsl{spectre  minimal de $\gT$}.
\end{enumerate}
\end{fdefinitions}

Notez que malgré leurs dénominations, $\Max\,\gT$,  $\jspec\,\gT$ 
et
$\Min\,\gT$ ne sont pas en général des espaces spectraux.

\begin{ftheorem}
\label{fthDK3}
Pour tout \trdi $\gT$ l'espace  $\Jspec\,\gT$ est un sous-espace 
spectral de
$\Spec\,\gT$ canoniquement homéomorphe à $\Spec(\He(\gT))$. 
\end{ftheorem}

\subsection{Treillis distributifs et relations implicatives}

Une règle particulièrement importante
pour les \trdis, dite \textsl{coupure}, est la
suivante
\begin{equation}\label{fcoupure1}
 \,x\vi a\; \leq\;  b\,\vet\;  \,a\; \leq\; x\vu  b\,
 \vd  a \leq\;  b.
\end{equation}

Si $A\in\Pfe(\gT)$ (ensemble des parties finiment énumérées de~$\gT$)  on notera
$$\ndsp \Vu  A:=\Vu _{x\in A}x\qquad {\rm et}\qquad \Vi  A:=\Vi _{\!x\in A}x.
$$

On  note $A \vda B$ ou $A \vdash_\gT B$ la relation définie comme suit  sur l'ensemble $\Pfe(\gT)$:

\snic{A \vda B \; \; \equidef\; \; \Vi  A\;\leq \;
\Vu  B.}

Cette relation vérifie les axiomes suivants, dans lesquels on
écrit $x$ pour $\{x\}$ et $A, B$  \hbox{pour $ A\cup B$}.

\vspace{-.5em}
$$\arraycolsep3pt\begin{array}{rcrclll}
&    & x  &\vda& x    &\; &(R)     \\[1mm]
 \hbox{si } A \vda B &   \hbox{alors}  & A,A' &\vda& B,B'   &\; &(M)     \\[1mm]
\hbox{si } (A,x \vda B) \hbox{ et }  (A \vda B,x) 
& \hbox{alors} & A &\vda& B &\;
&(T).
\end{array}$$
On dit que la relation est \textsl{réflexive}, \label{fremotr} \textsl{monotone} et
\textsl{transitive}.
La troisième règle (transitivité) peut être vue comme une
\gnn de la règle (\ref{fcoupure1}) et s'appelle \egmt la
règle de \textsl{coupure}.

\begin{fdefinition}
\label{fdefEntrel}
Pour un ensemble $S$ arbitraire, une relation binaire sur  $\Pfe(S)$  qui est
réflexive, monotone et transitive est
appelée une {\sl \entrel} (en anglais, {\sl entailment relation}).
On note en général une telle relation avec le symbole $\vda$ ou $\,\vdash_S\,$.
\end{fdefinition}

Le \tho suivant est fondamental. Il dit que les
trois propriétés des \entrels sont exactement ce qu'il faut pour que
l'interprétation d'une relation implicative comme la trace de celle d'un \trdi soit adéquate.

\begin{ftheorem}[\Tho fondamental des \entrels]
\label{fthEntRel1} {\rm  Voir \cite[Theorem 1]{fCC00}, \cite[\hbox{XI-5.3}]{fACMC}, \cite[Satz~7]{fLor1951}}.\\
Soit un ensemble $S$  avec une \entrel
$\vdash_S$ sur $\Pfe(S)$. On considère le \trdi~$\gT$ défini par
\gtrs et relations comme suit: les \gtrs sont les
\elts de $S$ et les relations sont les

\snic {A\; \vdash_\gT \;  B}

\noindent chaque fois que $A\; \vdash_S \; B$.  Alors, pour tous $A$,
 $B$ dans $\Pfe(S)$,  on a

\snic {\hbox{si }A\; \vdash_\gT \;  B
\hbox{ alors } A\; \vdash_S \;  B.}

\noindent 
En particulier, deux \elts $x$ et $y$ de $S$ définissent le même \elt de $\gT$ \ssi
\hbox{on a} $x\; \vdash_S \;  y$ et $y\; \vdash_S \;  x$.
\end{ftheorem}

\smallskip 
\rem La relation $x\vdash_S y$ est à priori  une relation de préordre, et non une relation d'ordre, sur~$S$. Notons $\ov x$ l'\elt $x$ vu dans l'ensemble ordonné~$\ov{S}$ associé à ce préordre, et pour toute partie~$A$ de~$S$ notons $\ov A=\sotq{\ov x}{x\in A}$.
Dans l'énoncé du \tho on considère un \trdi $\gT$ qui donne sur $S$ la même \entrel que $\vdash_S$.
En toute rigueur, on aurait dû noter $ \ov A\; \vdash_\gT \; \ov B$
plutôt que $A\; \vdash_\gT \;  B$ pour tenir compte du fait que l'\egt dans $\gT$ est plus grossière que dans $S$. En particulier c'est $\ov{S}$, et non pas $S$, qui s'identifie à une partie de $\gT$.
\eoe

\section{Espaces spectraux en \alg}

Les espaces spectraux couramment utilisés en \alg peuvent très souvent (voire toujours) être compris
comme les spectres de \trdis qui proviennent de structures \agqs décrites par des théories \textsl{cohérentes}. 

Dans cette section nous donnons une description générale de ce type de situation et nous montrons comment cela fonctionne pour quelques espaces spectraux usuels.
\subsection{Structures \agqs dynamiques, treillis distributifs et spectres associés}

Références: \cite{fCLR01,fLom-tgac}.
Dans l'article \cite{fCLR01} sont introduites les notions de \gui{dynamical theory} et de \gui{dynamical proof}.
Voir \egmt l'article \cite{fBC2005} qui décrit un certain nombre d'avantages fournis par cette approche.

\Subsubsection{Théories dynamiques et \sads}

Les \textsl{\tdys}, introduites dans \cite{fCLR01}, sont une version purement calculatoire \gui{sans logique} des théories (formelles du premier ordre) cohérentes.

Elles utilisent uniquement des  \rdys, \cad des règles de la forme
\[
\Gamma  \vd   \Exists{\und{y^1}}\, \Delta_1
\vou \cdots\vou \Exists{\und{y^m}}\,\Delta_m
\]
où $\Gamma$ et les $\Delta_i$ sont des listes de formules atomiques du langage considéré.

Les axiomes d'une \tdy sont des \rdys et les \gui{\thos} sont les \rdys validées à partir des axiomes selon un processus \elr bien défini.
Le contenu calculatoire de \gui{$\Exists{\und{y}}\, \Delta$} est \gui{$\EXists{\hbox{\ des\ nouveaux\ param\`etres\ }\und{y}}\, \Delta$}.
Le contenu calculatoire de \gui{$ U \vou  V\,\vou W$} est \gui{ouvrir trois branches de calculs, dans la première les prédicats de la liste $U$ sont valides,  \dots}. 

Une \textsl{\sad} pour une \tdy $\sa{T}$ est obtenue en ajoutant une présentation $(G,R)$ \gui{par générateurs et relations}: les \gtrs sont les \elts de~$G$
et sont ajoutés comme constantes dans le langage. Les relations sont les \elts de~$R$, ce sont des \rdys sans variables libres ajoutées aux axiomes de la théorie.

Les structures \agqs purement équationnelles, comme celles des groupes, des \trdis, des \grls, des anneaux commutatifs, correspondent à des \tdys particulièrement simples, où tous les axiomes sont simplement des équations.

Les \tdys dont les axiomes ne contiennent ni $\vou$, ni $\Exists$, ni $\Bot$ au second membre sont appelées des \textsl{théories de Horn} (\talgs dans \cite{fCLR01}).
Par exemple les anneaux \zedrs ou les anneaux \qis sont décrits par des théories de Horn.

Un \rdy est dite \textsl{existentielle simple} si le second membre (la conclusion) est de la forme $\Exists \ux\; \Delta$ où $\Delta$ est une liste finie de formules atomiques.  Une \tdy est dite \textsl{existentielle} si ses axiomes sont tous des \ralgs ou existentielles simples (une \textsl{regular theory} dans la littérature anglaise).

Les théories \textsl{existentiellement rigides} sont les \tdys dans lesquelles les axiomes existentiels sont simples et correspondent à des existences uniques. 
Une théorie existentielle existentiellement rigide
est dite \textsl{cartésienne}.

\smallskip 
Une théorie cohérente est aussi appelée une \tgm finitaire.
Les \tgms générales admettent en plus des axiomes avec des disjonctions infinies dans le second membre. Nous parlerons dans cet article uniquement des \tgms finitaires.

Un résultat fondamental pour les \tdys dit que si l'on ajoute la logique classique à une \tdy on ne change pas les règles valides: la théorie formelle du premier ordre classique est une extension conservative de la \tdy \cite[Theorem 1.1]{fCLR01}.

\Subsubsection{Treillis distributifs associés à une \sad}

Considérons une  \sad $\gA=\big((G,R),\sa{T}\big)$ pour une \tdy $\sa{T}=(\cL,\cA)$.

\paragraph{Premier exemple.} Si $P(x,y)$ est un prédicat dans la signature, et si $\mathit{Tcl}=\mathrm{Tcl}(\gA)$ est l'ensemble des termes clos  de~$\gA$, on obtient une \entrel $\vdash_{\gA,P}$ sur $\mathit{Tcl}\times \mathit{Tcl}$ en posant
\begin{equation} \label{feq1}
\begin{aligned} 
 (a_1,b_1),\dots,(a_n,b_n)  &\,\,\vdash_{\gA,P}    (c_1,d_1),\dots,(c_m,d_m)   
  \qquad\quad     \equidef     \\[.3em] 
   P(a_1,b_1)\vet \dots\vet  P(a_n, b_n) & \Vdi{\gA}   P(c_1, d_1)\vou \dots\vou   P(c_m, d_m) 
 \end{aligned}
\end{equation}

Intuitivement le \trdi $\gT$ engendré par cette \entrel est le treillis des \gui{valeurs de vérité} du prédicat~$P$ dans la \sad $\gA$.
En effet, donner un \elt $\alpha:\gT\to\Deux$ de $\Spec\gT$ revient à attribuer la valeur $\Top$ ou la valeur $\Bot$ à $P(a,b)$ 
selon que $\alpha(a,b)=1$ ou $0$. Comme cette structure est incomplètement spécifiée, dynamique, les valeurs de vérité sont des \elts du treillis $\gT$ plutôt que de $\so{\Bot,\Top}$

\paragraph{Deuxième exemple, le treillis de Zariski d'un anneau commutatif.}~

\smallskip \noindent On considère une \tdy des anneaux locaux non triviaux \sa{Al}, par exemple basée sur la signature $(\,\cdot=0,\U(\cdot)\mathrel{;}\cdot+\cdot,\cdot\times \cdot,-\,\cdot,0,1\,)$.

La théorie est une extension de la théorie des anneaux commutatifs.
Ici $\U(x)$ est défini comme le prédicat d'inversibilité de $x$ au moyen des axiomes

\DeuxRegles{
\labu $\,\,\U(x)\vd \Exists y\,  xy=1$
}
{
\labu $\,\,xy=1\vd \U(x)$
}

\noindent et les axiomes des anneaux locaux non triviaux sont les suivants

\DeuxRegles {\Lab{fAL} $\,\, \U(x+y) \Vd \U(x) \vou \U(y)$
}
{
\labu $\,\,\U(0)\vd \,\Bot$}
  
Le \textsl{treillis de Zariski} $\ZarA$
d'un anneau commutatif $\gA$ est défini comme le \trdi engendré par la \entrel $\vdash_{\ZarA}$ sur $\gA$ définie par
\begin{equation} \label{feqZarclass}
\begin{aligned} 
 a_1,\dots,a_n  &\,\,\vdash_{\ZarA}    c_1,\dots,c_m   
  \qquad\quad     \equidef     \\[.3em] 
\U(a_1)\vet \dots\vet  \U(a_n) & \Vdi{\sA{Al}(\gA)}   \U(c_1)\vou \dots\vou   \U(c_m) 
 \end{aligned}
\end{equation}

Ici $\sa{Al}(\gA)$ est la \sad de type $\sa{Al}$ sur l'anneau $\gA$.

On a alors l'\eqvc suivante (que l'on appelle un \textsl{Nullstellensatz formel}) 
$$a_1,\dots,a_n \,\vdash_{\ZarA}    c_1,\dots,c_m \Longleftrightarrow \exists k>0\;\;(a_1 \cdots a_n)^k\in\gen{c_1,\dots,c_m}
$$

Vu le \nst formel, le treillis  $\ZarA$ s'identifie
à l'ensemble des idéaux $\DA(\xn)=\sqrt[\gA]{\gen{\xn}}$, avec
$\DA(\fj_1)\vi\DA(\fj_2)=\DA(\fj_1\fj_2)$ et
$\DA(\fj_1)\vu\DA(\fj_2)=\DA(\fj_1+\fj_2)$.

Le spectre de Zariski usuel est (canoniquement homéomorphe à) le spectre de $\ZarA$. Donner un point de $\SpecA$ (un \idep) revient en effet à donner un épimorphisme $\gA\to\gB$ où $\gB$ est un anneau local, ou encore à donner un modèle minimal de la théorie des anneaux locaux basés sur l'anneau $\gA$.
Cela correspond à l'intuition \gui{forcer l'anneau à être un anneau local}.

\paragraph{Plus généralement.} On peut considérer un ensemble $S$ de formules atomiques closes de la \sad~$\gA=\big((G,R),\sa T\big)$ et définir une \entrel correspondante (avec les $A_i$ et $B_j\in S$)

\vspace{-.6em}
\begin{equation} \label{feq2}
\begin{aligned}
 A_1,\dots,A_n  &\,\,\vdash_{\gA,S} B_1,\dots,B_m   
   \qquad\quad \equidef     \\[.3em] 
    A_1\vet \dots\vet  A_n &\Vdi{\gA} B_1\vou\dots\vou B_m 
 \end{aligned}
\end{equation}
On pourra noter $\Zar(\gA,S)$ le \trdi ainsi construit.

En fait, pour un anneau commutatif \(\gA\), nous avons défini \(\Zar\gA\) selon ce schéma, avec la \tdy \sa{Al}, la \sad $\sa {Al}(\gA)=\big((G,R),\sa Al\big)$, où \((G,R)\) est le diagramme positif de \(\gA\) comme anneau commutatif, et où \(S=\sotq{\rU(a)}{a\in\gA}\).

\paragraph{Points d'un spectre et modèles en \clama.} En général on choisit
pour ensemble $S$ les formules atomiques closes construites sur un seul ou sur un petit nombre de prédicats du langage, en s'arrangeant pour que les autres prédicats puissent être définis à partir de ceux de $S$ dans la théorie formelle du premier ordre correspondante (en \clama). Dans ce cas donner un point du spectre
$\Spec(\Zar(\gA,S))$ revient à peu près à donner un modèle de la \sad $\gA$.
Précisément, dans le cas d'une théorie existentiellement rigide, donner un point du spectre
$\Spec(\Zar(\gA,S))$ revient exactement à donner un modèle minimal de la \sad $\gA$. Le modèle est minimal au sens où tous ses \elts sont construits de manière unique à partir des \gtrs~$G$ de la \sad au moyen des symboles de fonction d'une part et par utilisation des axiomes existentiels rigides d'autre part. Le choix de l'ensemble~$S$ influe alors sur la topologie
de l'espace spectral associé. Deux choix différents de~$S$ donnent les mêmes points mais pas la même topologie sur l'espace spectral correspondant.

\paragraph{Le treillis de Zariski (complet) d'une \sad $\gA$} est défini en prenant pour~$S$ l'ensemble $\mathrm{Atcl}(\gA)$ de toutes les formules atomiques closes de~$\gA$. On le note $\Zar(\gA,\sa T)$ ou par un nom particulier correspondant à la théorie \sa{T}, par exemple $\Val\gA$
pour la théorie \sa{Val}.
L'espace spectral dual est appelé le \textsl{spectre de Zariski de la \sad $\gA$} ou avec un nom particulier. 

Lorsque la théorie $\sa{T}$ est sans axiomes existentiels, $\Zar(\gA)$
est (à isomorphisme canonique près) le \trdi correspondant à la \entrel sur $\mathrm{Atcl}(\gA)$ engendrée par les instantiations des axiomes de $\sa{T}$ obtenues en y remplaçant les variables par des termes clos arbitraires.

\subsection{Un cas très simple} 
\label{fsubsecspecbanal}

Considérons une \talg $\sa{T}$.
Toute \sad $\gS=\big((G,R),\sa{T}\big)$ de type \sa{T} définit une structure \agq ordinaire $\gB$, le modèle générique de $\gS$. 
Il n'y a alors pas de différence significative entre  \sads et  structures \agqs usuelles. Nous supposons que (l'ensemble sous-jacent à) $\gB$ est non vide.

Les modèles minimaux de $\gS$ sont (identifiables avec) les structures
quotients $\gC=\gB /\!\!\sim$ (au sens usuel d'une structure quotient pour une structure \agq).
Considérons le treillis de Zariski complet $\Zar (\gB) $,
les points de \(\Spec(\gB):=\Spec(\Zar(\gB))\) sont donc identifiables avec ces structures quotients de $\gB$. Et un \elt de \(\Zar(\gS)=\Zar(\gB)\) s'identifie à un quotient  de \(\gB\) obtenu en ajoutant un nombre fini de relations dans la \pn.  

\smallskip Nous examinons maintenant un exemple simple, mais tout à fait significatif. On considère la théorie $\sa{T}=\sa{Mod}_\gA$ (théorie \peq des modules sur un anneau commutatif fixé~$\gA$). 

Pour tout \Amo $M$, les modèles minimaux de la \sad $\sa{Mod}_\gA(M)$ sont les quotients de $M$. On peut donc identifier un point de \(\Spec(M)\) au sous-module~$N$ de~$M$ tel que~$M/N$ soit le quotient considéré.
Le \trdi correspondant $\Zar (M)$ est très simple, il est défini par la \entrel $\ym\vdash \xn$ sur~$M$ équivalente à

\Regles{\labu $\,\, y_1= 0,\dots\vet y_m= 0\vd x_1= 0\vou\dots\vou x_n= 0$}

\noindent \cad encore à 

\Regles{\labu 
$\,\,$l'un des $x_k$ est dans le sous-module $\gA y_1+\dots+\gA y_m=\gen{\ym}$ (\textsl{\nst formel pour l'algèbre linéaire}).}

Une base d'\oqcs de \(\Spec(M)\) est définie à partir des sous-modules \tf de \(M\)
$$
\fD(y_1\vi\dots\vi y_m):=\fD(y_1)\cap\dots\cap\fD(y_m) =\sotQ{N\subseteq M}{\gen{y_i}_{i\in\lrbm}\subseteq N} \quad \quad \hbox{pour des }y_i\in M.
$$

\smallskip Il reste à savoir si la considération de tels (\trdis et) espaces spectraux, un peu trop simples, peut produire des résultats intéressants en \clama.

\smallskip\noindent  \textsl{Note.} On peut préférer le treillis opposé à celui défini ci-dessus, ce qui donne une autre topologie sur $\Spec M$.

\subsection{Treillis et spectre réel d'un anneau commutatif} 
\label{fsubsecspecreel}
Le spectre réel d'un anneau commutatif correspond à l'intuition de \gui{forcer l'anneau à être un corps ordonné (discret\footnote{On demande que la relation d'ordre, et a fortiori la relation d'\egt, soit décidable.})}. Un point du spectre est ici considéré comme donné par un morphisme $\varphi:\gA\to\gK$ où $(\gK,\gC)$ est un corps ordonné\footnote{$\gC$ est le cône formé par les \elts $\geq 0$.} \gui{engendré par\footnote{En termes plus savants, $\varphi$ est un épimorphisme de la catégorie des anneaux commutatifs.}} l'image de~$\gA$ dans~$\gK$.  En outre, deux morphismes $\varphi:\gA\to\gK$ et $\varphi':\gA\to\gK'$ définissent le même point du spectre s'il existe un \iso de corps ordonnés $\psi:\gK\to\gK'$ qui fait commuter le diagramme convenable.

Si l'on note \gui{$x\geq 0$} le prédicat sur  $\gA$ correspondant à \gui{$\varphi(x)\geq 0$ dans $\gK$}, ce prédicat vérifie notamment les règles suivantes
.

\DeuxRegles{
\labu $\vd x^2\geq 0$
\labu $\,\,x\geq0\vet y\geq0\vd x+y\geq0$
\labu $\,\,x\geq0\vet y\geq0\vd xy\geq0$
}
{
\labu $\,\,-1\geq0\vd \Bot$
\labu $\,\,-xy\geq0\vd x\geq 0\vou y\geq 0$
}

On reconnait en fait les axiomes pour un cône premier dans un anneau. 
Autrement dit, si un prédicat (noté $x\geq 0$) sur $\gA$ satisfait ces règles, il définit un unique point du spectre réel de $\gA$, tel que défini précédemment. 

Mais pour obtenir la topologie usuelle du spectre réel,
il faut plutôt considérer le prédicat opposé $x<0$, et pour le confort de 
la lecture on prend le prédicat $x>0$, et les règles qu'il satisfait sont les duales de celle satisfaites par $-x\geq 0$.

\DeuxRegles{
\labu $\,\,-x^2>0\vd \Bot$
\labu $\,\,x+y>0\vd x>0\vou y>0$
\labu $\,\,xy>0\vd x>0 \vou -y>0$
}
{
\labu $\vd 1 >0$
\labu $\,\,x>0\vet y>0\vd xy>0$
}

On définit donc le \textsl{treillis réel} de l'anneau $\gA$, noté $\Reel(\gA)$
comme le \trdi défini par la \entrel sur (l'ensemble sous-jacent à) $\gA$ engendrée par les relations suivantes.

\DeuxRegles{
\labu $\,\,-x^2 \vdash  $
\labu $\,\,x+y \vdash x , y $
\labu $\,\,xy \vdash x  , -y $
}
{
\labu $\,\,\vdash 1  $
\labu $\,\,x , y \vdash xy $
}

On trouve alors que $\Spec(\Reel(\gA))$, noté $\Sper\gA$, s'identifie en \clama avec l'ensemble des cônes premiers de $\gA$, muni de la topologie spectrale définie par les ouverts 
$$\fR(\an)=\sotq{\fc\in \Sper\gA}{\&_{i=1}^n -a_i\notin \fc}.$$

Cette approche du spectre réel est proposée dans~\cite{fCC00}.

Du point de vue \cof on essaie alors de remplacer tout discours des \clama sur l'espace $\SperA$ par un discours \cof sur le treillis $\Reel \gA $.

Pour cela, un point important est le \textsl{Positivstellensatz formel} dont nous donnons ici l'énoncé \cof. Dans l'énoncé en \clama on remplacerait le point~\textsl{2} par: \textsl{pour tout morphisme d'anneaux commutatifs $\varphi$ de $\gA$ vers un corps ordonné $\gK$, si les~$\varphi(x_i)$ sont tous~$>0$, au moins l'un des $\varphi(a_j)$ est $>0$.} 
\begin{ftheorem}[Positivstellensatz formel] \label{fthPstformelreel} 
\Propeq 
\begin{enumerate}
\item On a $$
x_1,\dots,x_k\vdash a_1,\dots,a_n 
$$
dans le treillis $\Reel \gA$.
\item On a $$
x_1>0\vet\dots \vet x_k>0\vd a_1> 0\vou\dots\vou  a_n> 0 
$$ 
dans la théorie des corps ordonnés \gui{sur $\gA$}.
\item On a $$
x_1>0\vet\dots \vet x_k>0\vet a_1\leq  0\vet\dots\vet  a_n\leq  0\vd \Bot 
$$ 
dans la théorie des corps ordonnés \gui{sur $\gA$}.
\item On a une \egt $s+p=0$  dans $\gA$, où $s$ est dans le \mo engendré par les $x_i$ et~$p$ est dans le cône engendré par les~$x_i$ et les~$-a_j$.
\end{enumerate}
\end{ftheorem}
%
\begin{proof}[\textsl{Démonstration directe de 1 $\Rightarrow$ 4}]~\\ 
 On laisse le soin \alec de vérifier
que les axiomes indiqués plus haut pour le prédicat~\hbox{$x>0$}, se traduisent tous
par des \egts $s+p=0$ décrites dans le point~\textsl{3}. On vérifie aussi que la condition exprimée en~\textsl{4} définit une \entrel $\vdash'$ sur~$\gA$: réflexivité et monotonie sont faciles. Voyons la coupure, avec $X\vdash' z,A$ et $X,z\vdash' A$. On a par hypothèse deux \egts, pour~$z$ avant ou après $\vdash'$, 
 $$
(a)\; z^ks_1+p_1=-zq_1 ,\quad (b)\;s_2+p_2=zq_2.
 $$
Les $s_i$, $p_i$ et  $q_i$  ($i=1,2$) ne font intervenir que $X$ et $A$.
Dans $(a)$ on peut supposer $k=2\ell$ pair (sinon on multiplie par $z$).
En multipliant $(a)$ et $(b)$ on obtient une \egt du type $z^{k}s_3+p_3=-z^2q_1q_2$, i.e. aussi $z^{k}s_3=-q_3$. En élevant $(b)$ à la puissance $k$
on obtient une \egt $s_4+p_4=z^kq_4$. En combinant les deux dernières \egts
on obtient
\[
z^kq_4s_3=-q_3q_4=s_3s_4+s_3p_4,\quad \hbox{d'où } s_3s_4+s_3p_4+q_3q_4=0.
\]
\vspace{-1em}      
\end{proof}
%

\subsection{Treillis et spectre linéaire d'un groupe (abélien) réticulé} 

La structure de groupe (abélien) réticulé est obtenue à partir de la structure de groupe abélien en ajoutant une loi
$\vu$ de sup-demi treillis. La théorie \sa{Grl}  des \grls sur la  signature $$
(\cdot=0;\cdot+\cdot,-\cdot,\cdot\vu\cdot,0)
$$ 
est \peq.
Les identités suivantes expriment 
le fait que~$\vu$  définit un sup-demi treillis ainsi que la compatibilité de $\vu$ avec $+$.

\DeuxRegles{
\lab{sdt1} $\vd x\vu x=x $
\lab{sdt2} $\vd  x\vu y=y\vu x$
}
{
\lab{sdt3} $\vd  (x\vu y)\vu z=x\vu (y \vu z)$
\lab{grl} $\vd  x+(y\vu z)=(x+y)\;\vu\;(x+z)$
}

On obtient la théorie \sa{Glio} des groupes abéliens totalement ordonnés en ajoutant à \sa{Grl} l'axiome $\Vd x\geq 0 \vou -x\geq 0$. On en déduit en particulier la règle 

\Regles{\labu $\,\,a\geq 0\vet b\geq 0\vet a\vii b=0\vd a=0 \vou b=0$.}

Le spectre linéaire d'un \grl $\Gamma$ correspond à l'intuition de \gui{forcer le groupe à être totalement ordonné} (on dit aussi, surtout dans la littérature anglaise \gui{linéairement ordonné}). Un point du spectre est ici considéré comme donné par un groupe totalement ordonné~$G$ quotient de~$\Gamma$, ou ce qui revient au même par le noyau $H$ du morphisme canonique~\hbox{$\pi:\Gamma\to G$}.
Ce sous groupe $H$ est un \textsl{sous-groupe solide} (comme tout noyau d'un morphisme de \grls) \textsl{premier} (i.e. le quotient est totalement ordonné).

Le \textsl{treillis linéaire} de $\Gamma$, noté $\Glio(\Gamma)$ est alors celui engendré par la \entrel  sur (l'ensemble sous-jacent à) $\Gamma$ définie comme suit.
\[ 
\begin{aligned} 
 a_1,\dots,a_n    &\,\,\vdash_{\Glio \Gamma } \;b_1,\dots,b_m   
\qquad  \equidef     \\ 
 a_1\geq 0\vet \dots\vet a_n\geq 0 &\Vdi{\sA{Glio}(\Gamma)} 
 \, b_1\geq 0\vou\dots\vou b_m\geq 0
 \end{aligned}
\]   
Sur la deuxième ligne, $\sa{Glio}(\Gamma)$ est la \sad de type \sa{Glio} avec pour \pn le diagramme positif de $\Gamma$. 

La \rdy ci-dessus est successivement \eqve aux règles suivantes (en notant \(a^-= (-a) \vu 0\))

\Regles{\labu $\,\, {a_1^- = 0} \vet \dots \vet {a_n^-= 0} \vd 
 b_1^-= 0\vou\dots\vou b_m^-= 0$
\labu $\,\, (a_1^-\vu \dots\vu a_n^-)= 0 \vd 
 (b_1^-\vi\dots\vi b_m^-)= 0$}

On obtient ainsi un \textsl{\nst formel} pour cette \entrel.
\[ 
\begin{aligned} 
 a_1,\dots,a_n\vdash_{\Glio \Gamma } b_1,\dots,b_m   
\;\Longleftrightarrow\;      
\exists k>0\;     (b_1^- \vi \dots\vi b_m^-)\leq k(a_1^- \vu \dots\vu a_n^-)
 \end{aligned}
\]

\subsection{Treillis et spectre valuatif d'un anneau commutatif}
\label{fsubsecspecval}

Le spectre valuatif d'un anneau commutatif correspond à l'intuition de \gui{forcer l'anneau à être un corps valué}. Un point du spectre noté $\Spev\gA$ est ici considéré comme donné par un morphisme $\varphi:\gA\to\gK$ où $(\gK,\gV)$ est un corps valué\footnote{$\gV$ est un anneau de valuation de $\gK$.} \gui{engendré par\footnote{En termes plus savants, $\varphi$ est un épimorphisme de la catégorie des anneaux commutatifs.}} l'image de~$\gA$ dans~$\gK$.  En outre, deux morphismes $\varphi:\gA\to\gK$ et $\varphi':\gA\to\gK'$ définissent le même point du spectre s'il existe un \iso de corps valués $\psi:\gK\to\gK'$ qui fait commuter le diagramme convenable.

Si l'on note $x\di y$ le prédicat sur  $\gA\times \gA$ correspondant à \gui{$\varphi(x)$ divise\footnote{I.e., $\exists z\in \gV\,z\varphi(x)=\varphi(y)$.} $\varphi(y)$ dans $\gK$}, ce prédicat vérifie notamment les règles suivantes, qui définissent la \tdy \(\sa{val}\).

\DeuxRegles{
\labu $\vd 1 \di  0$
\labu $\vd -1 \di  1$
\labu $\,\,a \di  b \vd ac \di  bc$
\labu $\vd a \di  b \vou b\di a$
}
{
\labu $\,\,0 \di  1\vd \Bot$ 
\labu $\,\,a \di  b \vet b \di  c \vd a \di  c$
\labu $\,\,a \di  b \vet a \di  c \vd a \di  b + c$
\labu $\,\,ax \di  bx  \vd a \di  b \vou 0 \di x$
}
 
Si \(\gA\) est un anneau commutatif, la \sad $\sa{val}(\gA)$ est obtenue en ajoutant le diagramme positif de l'anneau \(\gA\).
 
Si un prédicat (noté $x\di y$) sur $\gA\times \gA$ satisfait les 
axiomes de \(\sa{val}(\gA)\), il définit un unique point du spectre valuatif de $\gA$, tel que défini précédemment. 

On définit donc le \textsl{treillis valuatif} de l'anneau $\gA$, noté $\val(\gA)$
comme le \trdi défini par la \entrel sur (l'ensemble sous-jacent à) $\gA\times \gA$ engendrée par les relations suivantes.

\DeuxRegles{
\labu $\,\,\vdash (1 ,  0)$
\labu $\,\,\vdash (-1 ,  1)$
\labu $\,\,(a ,  b) \vdash (ac ,  bc)$
\labu $\,\,\vdash (a ,  b) , (b, a)$
}
{
\labu $\,\,(0,1)\vdash $  
\labu $\,\,(a ,  b) , (b ,  c) \vdash (a ,  c)$
\labu $\,\,(a ,  b) , (a ,  c) \vdash (a ,  b + c)$
\labu $\,\,(ax ,  bx)  \vdash (a ,  b) , (0 , x)$
}

On trouve alors que les points de l'espace spectral $\Spec(\val\gA)$ s'identifient à ceux de $\Spev\gA$. La topologie spectrale de $\Spec(\val\gA)$ correspond dans $\Spev\gA$ à la topologie  engendrée par les ouverts de base 
$
\fD (a,b)=\sotq{\varphi\in \Spev\gA}{\varphi(a)\di\varphi(b)}.
$

On a enfin le \textsl{Valuativstellensatz formel} suivant.

\begin{ftheorem}[Valuativstellensatz formel pour la théorie $\sa{val}(\gA)$] \label{fthPstformelVal}~\\
Pour tout anneau commutatif $\gA$ \propeq 
\begin{enumerate}
\item   Dans la \sad~\(\sA{val}(\gA)\), on a 
\[a_1\di b_1\vet \dots\vet a_n\di b_n  \vd    c_1\di d_1 \vet \dots\vet c_m\di d_m\]  

\item Dans le treillis $\val\gA$, on a 
$$
 (a_1,b_1),\dots,(a_n,b_n)  \,\vdash \,   (c_1,d_1),\dots,(c_m,d_m) 
$$

\item En introduisant des \idtrs $x_i$ ($i\in\lrbn$) et $y_j$ ($j\in\lrbm$)
on a dans l'anneau  $\gA[\ux,\uy]$ une \egt de la forme 
$$
 d \big(1+\som_{j=1}^my_jP_j(\ux,\uy)\big)\in \gen{(x_ia_i-b_i)_{i\in\lrbn},(y_jd_j-c_j)_{j\in\lrbm}} 
$$ 
où $d$ est dans le \mo engendré par les $d_j$, et les $P_j(\xn,\ym)$ sont dans $\ZZ[\ux,\uy]$.
\end{enumerate}
\end{ftheorem}

\subsection{Treillis et J-spectrum de Heitmann d'un anneau commutatif}

Dans un anneau commutatif, le \textsl{radical de Jacobson d'un \id 
$\fJ$} est (du point de vue des \clama) l'intersection des \idemas qui contiennent $\fJ$. On
le note $\JA(\fJ)=\rJ(\gA,\fJ)$, ou encore $\rJ(\fJ)$ si le contexte est 
clair.
En \coma on utilise la \dfn suivante, classiquement équivalente:
\begin{equation} \label{feqRadJac}
\JA(\fJ)\eqdefi\sotq{x\in\gA}{\Tt y\in\gA,\;\; 1+xy
\hbox{ est inversible modulo } \fJ}
\end{equation}
On notera $\JA(x_1,\ldots ,x_n)=\rJ(\gA,x_1,\ldots ,x_n)$ pour
$\JA(\gen{x_1,\ldots ,x_n})$.  L'\id  $\JA(0)$ est appelé le
\textsl{radical de Jacobson de l'anneau $\gA$}.

On appelle \textsl{treillis de Heitmann} d'un anneau commutatif $\gA$ 
le treillis
$\He(\ZarA)$, on le note $\HeA$.  En fait $\HeA$ s'identifie
à l'ensemble des $\JA(x_1,\ldots ,x_n)$, avec
$\JA(\fj_1)\vi\JA(\fj_2)=\JA(\fj_1\fj_2)$ et
$\JA(\fj_1)\vu\JA(\fj_2)=\JA(\fj_1+\fj_2)$.

On note $\Jspec(\gA)$ l'espace spectral $\Jspec(\ZarA)=\Spec(\HeA)$. En \clama c'est l'adhérence du spectre maximal, pour la topologie constructible,
dans $\SpecA$. On l'appelle le \hbox{\textsl{{\rm J}-spectrum de Heitmann}}.
Dans le cas où $\gA$ est \noe, $\Jspec(\gA)$ coïncide avec $\jspec(\gA)$,
lequel est le sous-espace topologique de $\SpecA$ formé par les \ideps
qui sont intersections d'\idemas. Heitmann a montré que dans le cas non \noe 
utiliser $\Jspec(\gA)$ permet de généraliser des résultats obtenus avec 
$\jspec(\gA)$ dans le cas \noe.

\smallskip \rems ~\\
1. On a, avec $\gT=\ZarA$, $\Zar(\gA/\JA(0))\simeq\gT/\big(\JT(0)=0\big)$, et $\Heit(\gA)\simeq \Heit(\gA/\JA(0))$. Par 
contre il ne
semble pas qu'il y ait une $\gA$-algèbre $\gB$ naturellement 
attachée à
$\gA$ pour laquelle on ait $\Zar\,\gB\simeq\Heit\,\gA$.\\
2. En général $\JA(x_1,\ldots ,x_n)$ est un \id 
radical mais pas
le radical d'un \itf. \eoe

\section{Petit dictionnaire de l'anti\eqvc de catégories}
\label{fsecAntiEquiv}

Références: le théorème de Krull page~\pageref{fThKrull}, \cite{fBW74,fCC00,fCL2001-2018}.

\smallskip Le contexte de cette section est le suivant:  $f:\gT\to\gT'$ est un morphisme de treillis distributifs et $\Spec(f)$, noté $f\sta$, est le morphisme dual, de $X'=\Spec(\gT')$ vers $X=\Spec(\gT)$, dans la catégorie des espaces spectraux.

Les résultats énoncés pour les \trdis sont valides pour les anneaux commutatifs en considérant chaque fois le treillis ou le morphisme $\Zar(\bullet)$.

\smallskip On rappelle tout d'abord quelques \dfns usuelles en \clama.
\begin{itemize}
\item Le morphisme $f$ est dit \textsl{lying over} (en français, \textsl{il possède la propriété de relèvement}) lorsque~$f\sta$ est surjectif: tout \idep de $\gT$ est image réciproque d'un \idep de~$\gT'$.
\item  Le morphisme $f$ est dit \textsl{going up} (en français, \textsl{il possède la propriété de montée pour les chaînes d'\ideps}) lorsque l'on a: \textsl{si $\fq\in X'$, $f\sta(\fq)=\fp$, et $\fp\subseteq\fp_2$ dans~$ X $, il existe  $\fq_2\in X'$ tel que
$\fq\subseteq\fq_2$ et $f\sta(\fq_2)=\fp_2$}.
\item  De même $f$ est dit \textsl{going down} (en français, i\textsl{l possède la propriété de descente pour les chaînes d'\ideps}) lorsque l'on a: \textsl{si $\fq\in X'$, $f\sta(\fq)=\fp$, et $\fp\supseteq\fp_2$ dans~$ X $, il existe  $\fq_2\in X'$ tel que $\fq\supseteq\fq_2$ et  $f\sta(\fq_2)=\fp_2$}.
\item  On dit que le morphisme $f$ \textsl{possède la propriété d'incomparabilité} lorsque ses  \gui{fibres} sont formées d'\ideps deux à deux incomparables: si $\fq_1\subseteq \fq_2\in X$ et $f\sta(\fq_1)=f\sta(\fq_2)$ dans~$X'$ alors $\fq_1= \fq_2$.
\item  L'espace spectral $ X $ est dit \textsl{normal} si l'adhérence d'un point contient un et un seul point fermé (tout \idep de $\gT$ est contenu dans un unique \idema).
\item L'espace spectral  $\SpecT$ est dit \textsl{complètement normal} si pour tous $x,y,z$ avec $x\in\ov{\so z}$ et $y\in\ov{\so z}$ on a $x\in\ov{\so y}$ ou 
$y\in\ov{\so x}$. 
\end{itemize}

\subsection{Propriétés des morphismes}

\begin{ftheorem} \label{fth-dico-trdi-spec-mor1} \emph{\cite[Theorem~IV-2.6]{fBW74}}
 En \clama on a les équivalences suivantes.   
\begin{enumerate}
\item $f\sta$ est surjectif ($f$ est lying over) $\Longleftrightarrow$  $f$ est injectif $\Longleftrightarrow$ $f$ est un monomorphisme $\Longleftrightarrow$ $f\sta$ est un épimorphisme.
\item $f$ est un épimorphisme $\Longleftrightarrow$ $f\sta$ est un monomorphisme $\Longleftrightarrow$ $f\sta$ est injectif.
\item $f$ est surjectif\footnote{Autrement dit, puisque c'est une structure équationnelle, $f$  est un morphisme de passage au quotient.} $\Longleftrightarrow$ $f\sta$ est un isomorphisme sur son image, qui est un
sous-espace spectral de $Y$.
\end{enumerate} 
\end{ftheorem}

Il y a des épimorphismes de \trdis non surjectifs (voir \cite[\hbox{section V-8}]{fBW74}).
Cela correspond à la possibilité d'un morphisme bijectif entre espaces spectraux qui ne soit pas un isomorphisme. Par exemple le morphisme spectral bijectif  $\Spec(\Bo(\gT))\to\Spec\,\gT$ n'est pas (en général) un \iso et le morphisme de treillis $\gT\to\Bo(\gT)$\footnote{$\Bo(\gT)$ est l'\agB librement engendrée par $\gT$.} est un \gui{épimono} qui n'est pas (en général) surjectif.

\begin{flemma} \label{flemLYO}
Soit  $S$ un système générateur de $\gT$. Pour que $f$ soit lying over, il faut et suffit que pour tous
$a_1,\dots,a_n,b_1,\dots,b_m\in S$ soit satisfaite l'implication
$$
f(a_1),\dots,f(a_n)\vdi {\gT'} f(b_1),\dots,f(b_m)
\;\;\Rightarrow\;\;
a_1,\dots,a_n\vdi {\gT} b_1,\dots,b_m
$$ 
\end{flemma}

\begin{fproposition}[Going up versus lying over] \label{fpropGu} En \clama \propeq
\begin{enumerate}
\item  Pour tout \idep $\fq$ de $\gT' $, en notant $\fp=f^{-1}(\fq)$, le morphisme 
   $$f':\gT/(\fp=0)\to \gT'/(\fq=0)$$  
est injectif (i.e., lying over).
\item  Pour tout idéal $I$ de $\gT'$, et $J:=f^{-1}(I)$, le morphisme 
$f_I:\gT/(J=0)\to \gT'/(I=0)$ est injectif. 
\item  Pour tout $y\in \gT'$, avec $J=f^{-1}(\dar y)$ le morphisme 
$f_y:\gT/(J=0)\to \gT'/(y=0)$ est injectif. 
\item  Pour tous $a,c\in \gT$ et $y\in \gT'$ on a
$$
f(a)\,\vdash_{\gT'}\, f(c),\, y \quad\Longrightarrow\quad\exists x\in \gT \;\;\; a \vdash_{\gT}\, c ,\, x \;\;\; \hbox{et}\;\;\; f(j)\leq_{\gT'} x.
$$
\end{enumerate} 
\end{fproposition}
\begin{proof}
Voir \cite{fCL2001-2018}.
\end{proof}

\begin{ftheorem} \label{fth-dico-trdi-spec-mor2} 
 En \clama on a les équivalences suivantes.   
\begin{enumerate}
\item $f$ est going up $\Longleftrightarrow$  pour
tous $a,c\in\gT$ et $y\in\gT'$ on a
$$
f(a)\leq f(c)\vu y \;\Rightarrow\;\exists x\in\gT\; (a\leq c \vu x \hbox{ et } f(x)\leq y).
$$ 
\item $f$ est going down $\Longleftrightarrow$  pour
tous $a,c\in\gT$ et $y\in\gT'$ on a
$$
f(a)\geq f(c)\vi y \;\Rightarrow\;\exists x\in\gT\; (a\geq c \vi x \hbox{ et } f(x)\geq y).
$$
\item $f$ possède la propriété d'incomparabilité $\Longleftrightarrow$ $f$ est zéro-dimensionnel\footnote{Voir ci-dessous \ref{fth-dico-trdi-spec-dim2}.}.   
\end{enumerate} 
\end{ftheorem}
%
\begin{proof}
Voir \cite{fCL2001-2018}.
\end{proof}
%

\begin{ftheorem} \label{fth-dico-trdi-spec-mor3} 
En \clama \propeq   
\begin{enumerate}
\item $\Spec(f)$ est une application ouverte.
\item  Il existe une application $\wi f:\gT'\to \gT$
satisfaisant les \prts suivantes.
\begin{enumerate}
\item \label{fi2a} Pour tous $c\in\gT$ et $b\in\gT'$, $b\leq f(c) \Leftrightarrow \wi f(b)\leq c$. \\
En particulier $b\leq f(\wi f(b))$ et $\wi f(b_1\vu b_2)=\wi f(b_1)\vu \wi f(b_2)$.
\item \label{fi2b} Pour tous $a,c\in\gT$ et $b\in\gT'$, $f(a)\vi b\leq f(c) \Leftrightarrow a\vi\wi f(b)\leq c $.
\item \label{fi2c} Pour tous $a\in\gT$ et $b\in\gT'$, $\wi f(f(a)\vi b)=a\vi \wi f(b)$.
\item \label{fi2d} Pour tout $a\in\gT$, $\wi f(f(a))=\wi f(1)\vi a$. 
\end{enumerate}
\item  Il existe une application $\wi f:\gT'\to \gT$
satisfaisant la \prt \ref{fi2b}.  
\item  Pour tout $b\in \gT$ la \bif \smash{$\Vi\limits_{b\leq f(c)} c$} existe, et si on la note $\wi f(b)$, 
 la \prt \ref{fi2b} est satisfaite.  
\end{enumerate} 
\end{ftheorem}
Pour un traitement dans le cadre de la théorie des locales on peut voir \cite[section~1.6]{fBor3}. Nous donnons maintenant une \demo pour les espaces spectraux.
Les implications qui concernent le point \textsl{1} sont démontrées en \clama. Le reste est \cof.

Dans un ensemble ordonné $(A,\leq )$ nous notons $\dar a=\sotq{x\in A}{x\leq a}$.
\begin{flemma} \label{flemcroiss}
Soit $f:A\to A'$ une application croissante entre les ensembles ordonnés $(A,\leq )$ et $(A',\leq')$. Soit $b\in A'$,  un \elt $b_1\in A$ satisfait l'\eqvc 
\[ \forall x\in A\,\,\, (\,b\leq ' f(x) \,\Longleftrightarrow\, b_1 \leq  x\,)
\]
\vspace{-1.5em}

\noindent \ssi 
\begin{itemize}
\item d'une part $b\leq' f(b_1)$, 
\item et d'autre part 
{$b_1= \Vi_{x:b\leq' f(x)}x$}. 
\end{itemize}
En particulier, si~$b_1$ existe,  il est uniquement déterminé.   
\end{flemma}
%
\begin{proof}
Si $b_1$ satisfait l'\eqvc on a $b\leq' f(b_1)$ car $b_1\leq b_1$. Si $z\in A$ satisfait l'implication $\forall x\in A\,(b\leq ' f(x) \Rightarrow z \leq  x)$, on obtient $z\leq b_1$ car $b\leq' f(b_1)$. Ainsi lorsque~$b_1$ satisfait l'\eqvc il est l'\elt maximum de $S_b\eqdef\bigcap_{b\leq' f(x)}\dar x\subseteq A$, \cad la borne inférieure de $\sotq{x\in A}{b\leq 'f(x)}$. \\
Inversement, si un telle \bif $b_1$ existe, elle satisfait tout d'abord l'implication  $\forall x\in A\,(b\leq ' f(x) \Rightarrow b_1 \leq  x)$ car $b_1\in S_b$. Ensuite, si $b\leq 'f(b_1)$ on a l'implication réciproque $\forall x\in A\,( b_1 \leq  x\Rightarrow b\leq ' f(x))$
parce que si $b_1\leq x$ alors $b\leq' f(b_1)\leq 'f(x)$.
\end{proof}
\begin{proof}[\textsl{\Demo du \tho \ref{fth-dico-trdi-spec-mor3}}]~\\ 
\noindent \textsl{3} $\Rightarrow$ \textsl{2}. La \prt \textsl{\ref{fi2a}} est le cas particulier de \textsl{\ref{fi2b}} avec $a=1$. La \prt \textsl{\ref{fi2d}} est le cas particulier de \textsl{\ref{fi2c}} avec $b=1$. Il reste à voir que \textsl{\ref{fi2b}} implique \textsl{\ref{fi2c}}. En effet
\[ 
\begin{aligned} 
\wi f(f(a)\vi b)     &= \Vi\nolimits_{c:f(a)\vi b\leq f(c)}c \quad \hbox{(lemme \ref{flemcroiss})}\\
   &= \Vi\nolimits_{c:a\vi \wi f(b)\leq c}c \qquad \hbox{(point \textsl{\ref{fi2b}})}  \\ 
&=a\vi \wi f(b)     
 \end{aligned}
\]
\textsl{1} $\Rightarrow$ \textsl{3}. On suppose que l'application $f\sta:\Spec\gT'\to\Spec\gT$ est ouverte. Si $b\in \gT'$, l'\oqc $\fD_{\gT'}(b)=\fB$ a pour image un \oqc de $\gT$, qui est de la forme~\hbox{$f\sta(\fB)=\fD_{\gT}(\wi b)$} pour un unique $\wi b\in\gT$. On note $\wi b=\wi f(b)$ et cela fournit une application $\wi f:\gT'\to \gT$. 
\\
Il reste à voir que le point \textsl{\ref{fi2b}} est vérifié. 
Pour $a,c\in\gT$ on note $\fA=\fD_{\gT}(a)$, $\fC=\fD_{\gT}(c)$ \hbox{et $g=f\sta$}. L'\eqvc  \textsl{\ref{fi2b}} à vérifier s'écrit
\[ 
\begin{aligned} 
g^{-1}(\fA) \cap \fB \subseteq g^{-1}(\fC) \,\Longleftrightarrow\,  \fA \cap g(\fB) \subseteq \fC  
 \end{aligned}
\]  
Pour l'implication directe, on considère un $x\in \fB$ tel que $g(x)\in \fA$, on doit montrer que $g(x)\in \fC$. Or $x\in g^{-1}(\fA) \cap \fB$, donc $x\in g^{-1}(\fC)$, \cad $g(x)\in \fC$.\\
Pour l'implication réciproque, on transforme le second membre par $g^{-1}$. Cette opération respecte l'inclusion et l'intersection.
On obtient $g^{-1}(\fA) \cap g^{-1}(g(\fB)) \subseteq g^{-1}(\fC)$ et on conclut en notant que $\fB\subseteq g^{-1}(g(\fB))$.

\smallskip \noindent  \textsl{2} $\Rightarrow$ \textsl{1}.
On va montrer que $f\sta(\fD_{\gT'}(b))=\fD_{\gT}(\wi f(b))$.\\
Montrons tout d'abord $f\sta(\fD_{\gT'}(b))\subseteq\fD_{\gT}(\wi f(b))$.
Soit $\fp'\in\Spec\gT'$ avec $b\notin \fp'$ et soit 
$$\fp=f\sta(\fp')=f^{-1}(\fp').$$ 
Si on avait $\wi f(b)\in\fp$ on aurait $f(\wi f(b))\in f(\fp)\subseteq \fp'$ et, puisque $b\leq f(\wi f(b))$, $b\in\fp'$. Ainsi on a bien $\fp\in \fD_{\gT}(\wi f(b))$.\\
Pour l'inclusion opposée, considérons un $\fp\in\fD_{\gT}(\wi f(b))$. 
Comme $\wi f$ est croissante et respecte  les~$\vu$,  l'image réciproque $\fq={\wi f}^{-1}(\fp)$
est un idéal. 
\\
On a $b\notin \fq$ car sinon,  
\hbox{$\wi f(b)\in\wi f({\wi f}^{-1}(\fp)) \subseteq\fp$}.  
\\
Si $y\in \fq$ alors $\wi f(y)=z\in\fp$ donc $y\leq f(z)$ pour un $z\in\fp$ (point~\textsl{\ref{fi2a}}). Inversement si  $y\leq f(z)$ pour un $z\in\fp$ alors $\wi f(y)\leq \wi f(f(z))\leq z$ (point~\textsl{\ref{fi2d}}), donc $\wi f(y)\in\fp$.
Ainsi
$$ 
\fq={\wi f}^{-1}(\fp)=\sotq{y\in\gT'}{\exists z\in \fp\,\,y\leq f(z)}. 
$$ 
Donc $f^{-1}(\fq)=\sotq{x\in\gT}{\exists z\in \fp\,\,f(x)\leq f(z)}$. Mais $f(x)\leq f(z)$ équivaut à $x\vi \wi f(1)\leq z$ (point \textsl{\ref{fi2b}} avec $b=1$). En outre $\wi f(1)\notin\fp$ car $\wi f(b)\leq \wi f(1)$ et $\wi f(b)\notin\fp$. Ainsi 
$$
f^{-1}(\fq)=\sotQ{x\in\gT}{\exists z\in \fp\,\,x\vi \wi f(1)\leq z}=\sotQ{x\in\gT}{x\vi \wi f(1)\in\fp}=\fp .
$$
Il se peut cependant que $\fq$ ne soit pas un \idep. Considérons alors un \id $\fq'$ maximal parmi ceux qui vérifient $f^{-1}(\fq')=\fp$ et $\wi f(b)\notin\fq'$. On veut montrer que $\fq'$ est premier. Supposons donc qu'on ait $y_1$ et $y_2\in\gT'\setminus\fq'$ tels que $y=y_1\vi y_2\in\fq'$.
Par maximalité il y a un \elt $z_i\in \gT\setminus \fp$ de tel que $f(z_i)$
est dans l'\id engendré par $\fq'$ et $y_i$ ($i=1,2$), i.e. $f(z_i)\leq x_i\vu y_i$ avec $x_i\in\fq'$.  
En prenant  $z=z_1\vi z_2$ (qui est dans $\gT\setminus \fp$) et $x=x_1\vu x_2$ on obtient $f(z_i)\leq x\vu y_i$ puis $f(z)=f(z_1)\vi f(z_2)\leq x \vu y_i$, donc $f(z)\leq x \vu y\in\fq'$, et enfin $z\in f^{-1}(\fq')=\fp$: absurde.   

\smallskip \noindent   \textsl{4} $\Leftrightarrow$ \textsl{3}. D'après le lemme
\ref{flemcroiss} en notant que \textsl{\ref{fi2b}} implique \textsl{\ref{fi2a}}.
\end{proof}

\subsection{Propriétés de dimension}

\begin{ftheorem}[Dimension des espaces] \label{fth-dico-trdi-spec-dim1}  \emph{Voir \cite{fCL2003,fLom02}, \cite[chapitre XIII]{fACMC}.} En \clama \propeq
\begin{enumerate}
\item L'espace spectral $\Spec(\gT)$ est de dimension $\leq n$ (au sens des chaînes d'idéaux premiers)  
\item  
Pour toute suite $(x_0,\dots,x_n)$ dans $\gT$ il existe une suite \emph{complémentaire}  $(y_0,\dots,y_n)$ au sens suivant
\begin{equation}\label{feqC2G}
\left.\arraycolsep3pt
\begin{array}{rcl}
1& \vda  &   y_n, x_n\\
 y_n,  x_n & \vda  &  y_{n -1}, x_{n -1}  \\
\vdots~~~~& \vdots  &~~~~  \vdots \\
  y_1, x_1& \vda  &  y_0, x_0  \\
y_0, x_0& \vda  & 0     
\end{array}
\right\}
\end{equation}
\end{enumerate}
\end{ftheorem}

Par exemple, pour $n=2$ les inégalités dans le point \textsl{2} correspondent au dessin suivant dans~$\gT$.
$$\SCO{x_0}{x_1}{x_2}{y_0}{y_1}{y_2}$$

\begin{ftheorem}[Dimension des morphismes] \label{fth-dico-trdi-spec-dim2}  \emph{Voir \cite{fCL2001-2018}, \cite[section XIII-7]{fACMC}.} 
Soit $\gT\subseteq \gT'$ et $f$ le morphisme d'inclusion. En \clama \propeq
\begin{enumerate}
\item Le morphisme $\Spec(f):\Spec(\gT')\to\Spec(\gT)$ est de dimension $\leq n$.
\item  Pour toute liste $(x_0,\dots,x_n)$ dans $\gT'$
il existe un entier $k\geq 0$ et des éléments $a_1,\ldots,a_k\in \gT$ tels que  pour tout couple de parties 
complémentaires $(H,H')$ de $\{1,\ldots,k\}$, il existe $ y_0,\dots,y_n\in \gT'$ tels 
que
\begin{equation} \label{feqdefDiTrRel}
\begin{array}{rclll}
\Vi_{j\in H'} a_j & \vda  &  y_n,\;x_n    \\
y_n,\;x_n& \vda  &y_{n-1},\;x_{n-1}      \\
\vdots\qquad & \vdots  & \qquad  \vdots    \\
y_1,\;x_1& \vda  &  y_0,\;x_0    \\
y_0,\;x_0& \vda  &  \Vu_{j\in H} a_j    \\
\end{array}
\end{equation}
\end{enumerate} 
\end{ftheorem}
Par exemple pour la dimension relative $\leq 2$ cela correspond au dessin suivant dans $B$ avec $u=\Vi_{j\in H'} a_j$ et $i=\Vu_{j\in H} a_j$.
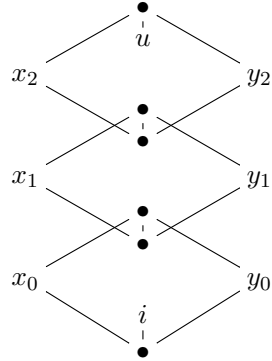
$$\SCOR{x_0}{x_1}{x_2}{y_0}{y_1}{y_2}{u}{i}$$

Notons que la dimension du morphisme $\gT\to\gT'$ est inférieure à la dimension de $\gT'$: prendre la liste vide ($k=0$) dans le point \textsl{2} du \tho~\ref{fth-dico-trdi-spec-dim2}. 

\smallskip Pour ce qui concerne les anneaux commutatifs, la dimension de Krull d'un anneau $\gA$ et la dimension de Krull relative d'un morphisme $\varphi:\gA\to\gB$  sont par \dfn celles $\ZarA$ et de~$\Zar\varphi$. Elles sont notées $\Kdim \gA$ et  $\Kdim \varphi$. 

On démontre en \clama  (voir \cite{fCL2001-2018}) que les \dfns classiques usuelles de $\Kdim \gA$ et  $\Kdim \varphi$ coïncident avec celles données ci-dessus. 

En outre l'article \cite{fCL2001-2018} démontre que les dimensions introduites dans l'ouvrage \cite{fACMC} coïncident avec celles définies ci-dessus. En particulier on a le résultat suivant.

\begin{fproposition} \label{fpropDimRelAc}
Soit  $\varphi:\gA\to\gB$ un morphisme d'anneaux commutatifs. Notons $\Abul$
l'anneau \zedr engendré par $\gA$. Alors,   $\Kdim\varphi=
\Kdim(\Abul\otimes_\gA\gB )$. 
\end{fproposition}

\subsection{Propriétés des espaces}

\begin{ftheorem} \label{fth-dico-trdi-spec-esp0} \cite[Chapter II, 4.2, 4.4, 4.9]{fJoh1986} \Propeq
\begin{enumerate}
\item L'espace spectral $\Spec(\gT)$ est séparé. 
\item L'espace topologique $\Spec(\gT)$ est un compact totalement discontinu. 
\item Le \trdi \(\gT\) est une \agB. 
\end{enumerate}
\end{ftheorem}

 Un \trdi $\gT$ est dit \textsl{normal} lorsque chaque fois que $a \vu b = 1$ dans $\gT$  il existe $x, y$ tels que   $a \vu x = b \vu y = 1$ et  $x \vi y = 0$. Voir \cite{fWeh2019,fDST2019}.
Notons qu'en remplaçant $x$ et $y$ par $x_1=x\vu(a\vi b)$ et $y_1=y\vu(a\vi b)$
on obtient $a \vu x_1 = b \vu y_1 = 1$ et  $x_1 \vi y_1 = a\vi b$.

\begin{ftheorem} \label{fth-dico-trdi-spec-esp1}  \Propeq
\begin{enumerate}
\item L'espace spectral $\Spec(\gT)$ est normal. 
\item Le \trdi \(\gT\) est normal. 
%
%
\end{enumerate}
\end{ftheorem}
\begin{ftheorem} \label{fth-dico-trdi-spec-esp2}  \Propeq
\begin{enumerate}
\item L'espace spectral $\Spec(\gT)$ est complètement normal. 
\item Tout intervalle $[a,b]$ de $\gT$, vu comme \trdi, est normal.  
\item Pour tous $a,b\in \gT$ il existe $x, y$ tels que   $a \vu b = a \vu y = x \vu b$ et  $x \vi y = 0$.
\end{enumerate}
\end{ftheorem}
\begin{ftheorem} \label{fth-dico-trdi-spec-esp3}  \Propeq
\begin{enumerate}
\item Dans $\Spec(\gT)$ tout \oqc est une réunion finie d'\oqcs irréductibles. 
\item Pour tous \(a_1,\dots,a_n,b_1,\dots,b_m\) 
on a  $a_1,\dots,a_n\vdash_\gT b_1,\dots,b_m$ \ssi on a un~\(j\) tel que
 $a_1,\dots,a_n\vdash_\gT b_j$. 
\item 
Le \trdi $\gT$ est construit à partir d'une  \sad pour une \talg. 
%
%
\end{enumerate}
\end{ftheorem}

\section{Quelques exemples}

Les théorèmes classiques donnés dans cette section ont maintenant un contenu \cof clair. C'était impossible avant les \dfns \covs des concepts utilisés.
Certaines versions que nous donnons sont nouvelles ou plus fortes que les résultats connus auparavant en \clama.

\subsection{Dimension relative, lying over, going up, going down}

Voir \cite{fCL2001-2018} et \cite[sections XIII-7 et XIII-9]{fACMC}.

\begin{ftheorem} 
\label{fthDim1} 
Soit  $\rho:\gA\to\gB$ un morphisme d'anneaux commutatifs ou de \trdis. 
 Supposons que  $\Kdim\,\gA\leq m$  et $\Kdim\,\rho\leq n$, 
alors  $\Kdim\,\gB\leq (m+1)(n+1)-1$.
\end{ftheorem}

\begin{ftheorem} \label{fthLYGUKdim}
Si un morphisme $\alpha:\gA\to \gB$ de \trdis ou d'anneaux commutatifs est lying over et going up
(ou bien lying over et going down) on a 
$\Kdim(\gA)\leq \Kdim(\gB)$. 
\end{ftheorem}

\begin{flemma}\label{fcor2thKdimMor}
On consid\`ere une \alg $\rho:\gA\to\gB$.
 Supposons que  $\gB$ est engendrée par des \elts primitivement  \agqs\footnote{Un \elt de $\gB$ est dit primitivement  \agq sur $\gA$
 s'il annule un \pol de $\AX$ dont les \coes sont \com.} sur $\gA$,
alors $\Kdim \rho\leq0$ et donc $\Kdim\gB\leq\Kdim\gA$.
\end{flemma}

\begin{flemma} 
\label{flemGU1} Soit $\varphi:\gA\to \gB$ un morphisme d'anneaux commutatifs. 
\Propeq
\begin{enumerate}
\item Le morphisme $\varphi$ est lying over.
\item Pour tout idéal  $\fa$ de~$\gA$  et tout $x\in \gA$, on a:
$\, \varphi(x)\in \varphi(\fa)\gB\, \Rightarrow \, x\in\sqrt[\gA]{\fa}.
$  
\end{enumerate}

\end{flemma}

\begin{flemma} 
\label{flemGU2} Soit $\varphi:\gA\to \gB$ un morphisme d'anneaux commutatifs. 
\Propeq
\begin{enumerate}
\item Le morphisme $\varphi$ est going up (i.e. le morphisme $\Zar\,\varphi$ est going up).
\item Pour tout idéal $\fb$ de $\gB$, en notant $\fa=\varphi^{-1}(\fb)$, le morphisme
$\varphi_\fb:\gA/\fa\to \gB/\fb$ obtenu par factorisation est lying over. 
\item Même chose en se limitant aux \ids $\fb$  de type fini.
\end{enumerate}
En \clama, on peut se limiter aux \ideps dans le point 2.
\end{flemma}

\begin{flemma} \label{flemGuAnFp}
Soit $\gA\subseteq \gB$ une $\gA$-\alg 
fidèlement plate sur $\gA$. Le morphisme $\gA\to \gB$ est lying over et going up. En conséquence $\Kdim\, \gA\leq \Kdim\, \gB$.
\end{flemma}

\begin{flemma}[Un going up classique] 
\label{fcorLy1} 
 Soit $\gA\subseteq \gB$ des anneaux commutatifs avec $\gB$ 
entier sur~$\gA$. Alors le morphisme $\gA\to \gB$ est lying over et going up. 
En conséquence $\Kdim\, \gA\leq \Kdim\, \gB$.
\end{flemma}

\begin{flemma}\label{flem1Gdown}
\label{flemGD2} Soit $\varphi:\gA\to \gB$ un morphisme d'anneaux commutatifs. 
\Propeq
\begin{enumerate}
\item Le morphisme $\varphi$ est going down.
\item Pour tous $b$, $a_1$, \ldots, $a_q\in \gA$ et $y\in \gB$ tels que~$\varphi(b)y\in\sqrt[\gB]{\gen{\varphi(a_1,\dots,a_q)}}$,
il existe 
$x_1$, \dots, $x_p\in \gA$ tels que:
$$
\gen{bx_1,\dots,bx_p}\subseteq  \sqrt[\gA]{\gen{a_1,\dots,a_q}} \;\hbox{ et }\;y\in\sqrt[\gB]{\gen{\varphi(x_1), \dots, \varphi(x_p)}}.
$$
\item (en \clama) Pour tout \idep $\fp$ de $\gB$, \hbox{avec $\fq=\varphi^{-1}(\fp)$}, le morphisme $\gA_\fq\to \gB_\fp$ obtenu par passage au quotient est lying over.
\end{enumerate}
\end{flemma}

\begin{ftheorem}[Going Down] 
\label{fthGDplat}  
Soit $\gA\subseteq \gB$ des anneaux. Le morphisme $\gA\to \gB$ est going down dans les deux cas suivants. 
\begin{enumerate}
\item $\gB$ est plat sur $\gA$.
\item $\gB$ est intègre et entier sur $\gA$, et $\gA$ est intégralement clos.
\end{enumerate}
\end{ftheorem}

\subsection{Théorèmes Kronecker, Forster-Swan, Serre et Bass}

Références: \cite{fCLQ2004,fCLQ2006} et \cite[chapitre XIV]{fACMC}.

\begin{ftheorem}[\Tho de Kronecker-Heitmann,
avec la dimension de Krull, non \noe]
\label{fthKroH} ~
\begin{enumerate}
\item Soit $n\geq 0$.
Si $\Kdim \gA <n$ et $b_1$, \dots, $b_n\in\gA$,  il existe
$x_1$, \dots, $x_n$ tels que pour tout $a\in\gA$,
$\DA(a,b_1,\dots,b_n) = \DA(b_1+ax_1,\dots,b_n+ax_n)$.
\item En conséquence, dans un anneau de dimension de Krull $\leq n$, tout
\itf a m\^{e}me nilradical qu'un \id engendré par au plus $n+1$ \elts.
\end{enumerate}
\end{ftheorem}

Pour un anneau commutatif $\gA$ on définit $\Jdim\gA$ (la J-dimension de $\gA$) comme égale à $\Kdim(\Heit\gA)$. On démontre en \clama que c'est la dimension du \hbox{J-spectrum} de Heitmann $\Jspec(\gA)$.

Dans les références citées ci-dessus, une autre notion de dimension est introduite, la dimension de Heitmann, notée $\Hdim(\gA)$. On a toujours 
$\Hdim(\gA)\leq \Jdim(\gA)\leq \Kdim(\gA)$. Nous donnons dans la suite les énoncés avec la J-dimension, mais il sont \egmt vrais avec la dimension de Heitmann.

\begin{fdefinition}\label{fdefiStableRange}
Soit $n\geq 0$. On dit qu'un anneau $\gA$ est de \textsl{stable range} \textsl{(de Bass)}  $\leq n$ lorsque l'on peut 
\gui{raccourcir} les \vmds de longueur $n+1$ au sens suivant:
$$1 \in \gen {a,\an}\;\Longrightarrow\;\exists \,\xn,\;1 \in \gen {a_1 + x_1a, \ldots, a_n + x_na}.$$
\end{fdefinition}

\begin{ftheorem}[\Tho de Bass-Heitmann, sans
\noet]
\label{fBass}  
Soit $n\geq 0$. Si $\Jdim \gA <n$, alors $\gA$ est de \emph{stable range} $\leq n$. 
En particulier  tout \Amo stablement libre de rang $\geq n$ est libre.
\end{ftheorem}

\begin{ftheorem}[\Tho Splitting Off de Serre, pour la $\Jdim$]
\label{fthSerre} ~\\
Soit $k\geq 1$ et soit $M$  un \Amo  \pro de rang $\geq k$, ou plus \gnlt
 isomorphe à l'image d'une matrice de rang $\geq k$.\\
Supposons que $\Jdim  \gA< k$.
Alors $M\simeq N\oplus \gA$ pour un certain module~$N$ isomorphe à l'image d'une matrice de rang $\geq k-1$. 
\end{ftheorem}

\begin{fcorollary}\label{fcorthSerre}
Soit un anneau $\gA$ tel que $\Jdim \gA\le h$  et soit $M$ un \Amo 
isomorphe à l'image d'une matrice de rang $\geq h+s$.  Alors~$M$
contient en facteur direct un sous-module libre de rang~$s$.  \Prmt, si~$M$
est l'image d'une matrice $F\in\gA^ {n\times m}$ de rang $\geq h+s$, on a $M= N\oplus L$
o\`u $L$ est libre de rang $s$ en facteur direct dans $\gA^n$, et $N$ est
l'image d'une matrice de \hbox{rang $\geq h$}.
\end{fcorollary}

Dans le théorème suivant, intervient la notion d'un \mtf   \lot
engendré par $r$ \elts. En \clama cela signifie qu'après \lon en n'importe quel \idema,  $M$ est engendré
par $k$ \elts. Une \dfn \cov \eqve consiste à dire que le $k$-ème \id de Fitting de $M$ est égal à $\gen{1}$.

\begin{ftheorem}[\Tho de Forster-Swan pour la $\Jdim$]
\label{fthSwan}  Soit $k\geq 0$ \hbox{et $r\geq 1$}.
Si $\Jdim  \gA\leq   k$,  et
si un \Amo $M=\gen{y_1 , \dots,  y_{k+r+s}}$ est   \lot
engendré par $r$ \elts, alors il est engendré par $k+r$
\elts, plus \prmt 
  on peut calculer 
$$
z_1,\, \ldots, \,z_{k+r}  
\in \gen{y_{k+r+1},\ldots,y_{k+r+s}}
$$
 tels que
$M$  soit engendré par $(y_1+z_1, \ldots, y_{k+r}+z_{k+r})$.
\end{ftheorem}

Étant donnés deux modules $M$ et $L$ on dira que $M$
\textsl{est simplifiable pour $L$} si $M\oplus L\simeq N\oplus L$
implique   $M \simeq N$.

\begin{ftheorem}[\Tho de simplification de Bass, pour la $\Jdim$]
\label{fthBassCancel2} ~\\
Soit $M$ un \Amo
\ptf de rang $\geq k$.
Si $\Jdim \gA< k$,
alors $M$ est simplifiable pour tout  \Amo \ptf:
 si $Q$ est \ptf et 
$M\oplus Q\simeq N\oplus Q$, alors~$M\simeq N$.
\end{ftheorem}

Les théorèmes \ref{fthSerre}, \ref{fthSwan} et \ref{fthBassCancel2} étaient conjecturés dans \cite{fHei84}.

\subsection{Autres résultats utilisant la dimension de Krull}
Dans l'ouvrage \cite{fACMC} le théorème XII-6.2 donne une \carn importante.
\textsl{Un anneau \coh \icl $\gA$ de \ddk $\leq 1$ est un domaine de Prüfer.} Cela explique \cot la \dfn aujourd'hui classique des \doks comme  anneaux \noes \icl de \ddk $1$, et le fait que, avec cette \dfn, en \clama, on est capable de démontrer que les \itfs sont \ivs. Il fallait pour cela remplacer dans l'hypothèse la \noet par la \cohc. Et il n'y a pas de \prco du fait que la \noet implique la \cohc.

On trouve  dans \cite[chapitre XVI]{fACMC} une \prco du \tho de Lequain-Simis, utilisant la dimension de Krull. 

On trouve dans \cite[section 2.6]{fYen2015} le résultat suivant (qui n'était pas connu auparavant) avec une \prco. \textsl{Si $\gA$ est un anneau de dimension de Krull $\leq d$, les modules stablement libres de rang $>d$ sur $\AX$ sont libres.}

\subsection{Application ouverte}

Un \tho d'\alg commutative dû à Grothendieck affirme que si un homomorphisme d'anneaux $\varphi:\gA\to\aqo{\AXn}{\lfs}$ est plat, le morphisme dual $\Spec \varphi$ est une application ouverte.

Cela serait intéressant d'en donner une \prco ainsi qu'une application
concrète, à trouver dans la littérature. On pourrait commencer par examiner le cas $\aqo{\AXn}{f}$: on sait que $\varphi$ est plat \ssi l'\id engendré par les \coes de $f$ est engendré par un idempotent.

\rdb
\addcontentsline{toc}{section}{Références}

\small

\stopcontents[french]

\end{document}